\documentclass[reqno]{amsart}
\usepackage{amssymb}
\usepackage{amsmath}
\usepackage{changes}
\usepackage[mathscr]{euscript}
\usepackage[small]{caption}
\usepackage{mathtools}
\usepackage{graphicx}
\graphicspath{ {images/} }
\usepackage{xcolor}
\usepackage{comment}
\usepackage{changes}

\makeatletter
\@addtoreset{equation}{section}
\makeatother

\newtheorem{theorem}{Theorem}[section]
\newtheorem{lemma}[theorem]{Lemma}
\newtheorem{proposition}[theorem]{Proposition}
\newtheorem{corollary}[theorem]{Corollary}

\newtheorem{definition}[theorem]{Definition}
\newtheorem{remark}[theorem]{Remark}

\newcounter{as}[section]

\newcommand{\mc}[1]{{\mathcal #1}}
\newcommand{\mf}[1]{{\mathfrak #1}}

\newcommand{\bb}[1]{{\mathbb #1}}
\newcommand{\bs}[1]{{\boldsymbol #1}}
\newcommand{\ms}[1]{{\mathscr #1}}

\newcommand{\<}{\langle}
\renewcommand{\>}{\rangle}

\definecolor{bblue}{rgb}{.2,0.2,.8}

\begin{document}

\title[Steady state large deviations]{Steady state large deviations
for one-dimensional, symmetric exclusion processes in weak contact
with reservoirs}

\begin{abstract}
Consider the symmetric exclusion process evolving on an interval and
weakly interacting at the end-points with reservoirs. Denote by
$I_{[0,T]} (\cdot)$ its dynamical large deviations functional and by
$V(\cdot)$ the associated quasi-potential, defined as
$V(\gamma) = \inf_{T>0} \inf_u I_{[0,T]} (u)$, where the infimum is
carried over all trajectories $u$ such that $u(0) = \bar\rho$,
$u(T) = \gamma$, and $\bar\rho$ is the stationary density profile.  We
derive the partial differential equation which describes the evolution
of the optimal trajectory, and deduce from this result the formula
obtained by Derrida, Hirschberg and Sadhu \cite{DHS2021} for the
quasi-potential through the representation of the steady state as a
product of matrices.
\end{abstract}

\author{A. Bouley, C. Erignoux, C. Landim}

\address{Laboratoire de Mathématiques Raphaël Salem
UFR des Sciences et Techniques
Université de Rouen Normandie
Avenue de l'Université, BP.12
76801 Saint-Étienne-du-Rouvray, France
e-mail: \texttt{angele.bouley@ens-rennes.fr}}

\address{Equipe PARADYSE, Bureau B211 \\
Centre INRIA Lille Nord-Europe \\
Park Plaza, Parc scientifique de la Haute-Borne, 40 Avenue Halley\\
Bâtiment B, 59650 Villeneuve-d'Ascq \\
France.
e-mail: \texttt{clement.erignoux@inria.fr}}

\address{IMPA, Estrada Dona Castorina 110, CEP 22460 Rio de Janeiro, Brasil
and CNRS UMR 6085, Université de Rouen, France. \\
e-mail: \texttt{landim@impa.br} }

\maketitle

\section{Introduction}
\label{sec0}

Non-equilibrium steady states have attracted a lot of interest in the
last decades, as a first step towards the understanding of far from
equilibrium behavior. We refer to the reviews \cite{D2007, D2011,
BDGJL2015}, the recent works \cite{DHS2021, FGK2020} and references
therein. These states display many interesting phenomena, such as
non-local thermodynamic functionals, dynamical phase transitions and
long range correlations, \cite{DLS2002, BDLW2008, BDGJL2011}. Many of
these properties can be derived from the quasi-potential, the
functional which plays a role analogous to the free energy in
equilibrium.

We consider the symmetric exclusion process evolving in the interval
$[0,1]$ and in contact with reservoirs at the end points. In the case
of a strong interaction of the system with the reservoirs, the
boundary conditions do not appear in the thermodynamic functionals and
the effect of the boundary is not clear. To investigate the influence
of the boundaries, we examine in this article the case of weak
interactions.

With strong interactions with the reservoirs, the density at the
boundaries take immediately the value of the reservoirs densities and
remain fixed, while the density in the bulk evolves according to the
heat equation. That is to say, the density profile evolves according
to the heat equation with Dirichlet boundary conditions. Even at the
level of the dynamical large deviations, the densities at the boundary
are kept fixed, and only the density at the interior may fluctuate,
\cite{KL, BDGJL, BDGJL2003}.

In constrast, for exclusion processes with weak interactions with the
reservoirs, the densities at the boundaries evolve in time. Actually,
the particles' density $u$ solves the heat equation with Robin boundary
conditions \cite{BMNS2017}. Namely,
\begin{equation}
\label{1-06} 
\begin{cases}
\partial_tu \,=\, \Delta u \quad (t,x) \, \in\, (0,T) \,\times\, (0,1) \\
(\nabla u) (t,0) \,=\, A^{-1} [\,u(t,0)-\alpha\,] \quad t\,\in\, (0,T) \\
(\nabla u)(t,1) \,=\, B^{-1} [\, \beta - u (t,1)] \quad t\,\in\, (0,T)
\\
u(0,x) \,=\, \gamma (x)\quad x\,\in\, [0,1] \;.
\end{cases}
\end{equation}

\noindent In this formula, $\alpha$, $\beta \in (0,1)$ represent the
density at the left, right reservoirs, respectively, $A$, $B>0$ the
intensity of the interaction with the left, right reservoirs,
respectively, and $\gamma: [0,1] \to [0,1]$ the initial density
profile. Moreover, $\nabla u$ stands for the partial derivative in
space of $u$, $\partial_t u$ for its partial derivative in time and
$\Delta u$ for the Laplacian of $u$ in the space variable.

The weak interaction of the system with the boundaries also modifies
the thermodynamical variables by adding boundary terms.  The
Hamiltonian, denoted by $\ms H(\gamma, F)$, becomes
\begin{equation}
\label{1-09}
\begin{aligned}
\ms H(\gamma, F) \; =\;
\; & -\; \big\langle \nabla \gamma \, ,\, \nabla F  \big\rangle 
\;+\; \big\langle \sigma(\gamma) \,,\, \big(\nabla F\big)^2  \big\rangle \\
\; & +\; \mf b_{\alpha, A} \big(\, \gamma(0)\, ,\, F (0)\, \big)
\;+\; \mf b_{\beta, B} \big(\, \gamma(1)\, ,\, F (1)\, \big)\;.
\end{aligned}
\end{equation}
In this formula and below, $\< \,\cdot\,,\,\cdot\,\>$ represents the
scalar product in $\ms L^2([0,1])$, $\sigma: [0,1] \to \bb R$, given by
$\color{bblue} \sigma(a) = a (1-a)$, is the mobility of the exclusion
process, and for $0< \varrho <1$, $D>0$, $0<a<1$, $M\in \bb R$,
\begin{equation}
\label{4-05}
\mf b_{\varrho, D} (a,M) \;=\; \frac{1}{D}\, \Big\{ [1-a]\, \varrho \,
[e^M-1] \;+\; a\, [1-\varrho] \, [e^{-M}-1] \, \Big\}  \;.
\end{equation}
In the Hamiltonian formalism of classical mechanics, density profiles
$\gamma: [0,1] \to [0,1]$ play the role of position and external
fields $F:[0,1] \to \bb R$, the one of momentum.

The dynamical large deviations functional associated to the
Hamiltonian $\ms H$ is given by
\begin{equation}
\label{7-01b}
I_{[0,T]} (u) \; =\;
\sup_{H}
\int_0^T \big\{\,
\big\langle \partial_t u_t,  H_t  \big\rangle
\,-\, \ms H(u_t, H_t) \, \big\} \; dt \; ,
\end{equation}
where the supremum is carried over all smooth functions
$H: [0,T]\times [0,1] \to \bb R$. The functional
$I_{[0,T]} (u)$ specifies the cost of observing a fluctuation
$u(t)$, $0\le t\le T$. In particular, $I_{[0,T]} (u) =0$ if $u$ follows
the hydrodynamic equation \eqref{1-06}.

Let $\bar\rho$ be the unique stationary solution to the equation
\eqref{1-06}. That is, $\bar\rho$ is the solution to the elliptic
equation
\begin{equation}
\label{1-05}
\begin{cases}
\Delta \rho  \,=\,  0\\
(\nabla \rho) (0) \,=\, A^{-1} [\,\rho (0)-\alpha\,] \\
(\nabla \rho)(1) \,=\, B^{-1} [\, \beta - \rho (1)] \;.
\end{cases}
\end{equation}
An elementary computation yields that
$\bar\rho$ is given by
\begin{equation*}
\bar\rho(x) \;=\; \frac{\alpha (1+B) + \beta A}{1+B+A}
\;+\; \frac{(\beta-\alpha)\, x}{1+B+A}\;\cdot 
\end{equation*}
Note that $\bar\rho$ is the linear interpolation between
$\bar\rho(-A) = \alpha$ and $\bar\rho(1+B) = \beta$.

Denote by $V(\,\cdot\,)$ the quasi-potential associated to the rate
function $I_{[0,T]}(\,\cdot\,)$. It is given by
\begin{equation*}
V(\gamma) \;:=\; \inf_{T>0}\; \inf_{u(\cdot)} \; I_{[0,T]}(u) \; ,
\end{equation*}
where the infimum is carried over all paths $u$ such that
$u(0) = \bar\rho$, $u(T)=\gamma$.  The quasi-potential $V(\gamma)$
measures the minimal cost to produce a profile $\gamma$ starting from
$\bar\rho$. It is also the rate functional of the large deviations
principle for the density profile under the steady state \cite{BDGJL}.

By using a representation of the steady state of the exclusion process
as a product of matrices, Derrida, Hirschberg and Sadhu \cite{DHS2021}
proved that the quasi-potential can be expressed as
\begin{equation*}
\begin{gathered}
V (\gamma) \;=  \;
\int_{0}^1  \Big\{\, \gamma (x) \, \log
\frac {\gamma (x)}{F(x)} + \big[1 - \gamma (x)\big]
\log \frac {1- \gamma (x)}{1- F(x)}
+ \log \frac {\nabla F(x)}{[\beta - \alpha]}\,
\Big\} \; dx \\
\;+\; A\,
\ln \frac{F(0) - \alpha}{A(\beta-\alpha)}
\;+\; B\, \ln \frac{\beta - F(1)}{B (\beta-\alpha)} \;,
\end{gathered}
\end{equation*}
where $F$ solves the non linear boundary
value problem
\begin{equation}
\label{3-01}
\left\{
\begin{aligned}
& \vphantom{\Big(}
{\displaystyle \Delta F = 
\big( \gamma - F \big) \frac{\big( \nabla F \big)^2}{F(1-F)} 
\qquad \textrm{ in $(0,1)$}  }\; , \\
& \vphantom{\Big(}
{\displaystyle \nabla F(0) = A^{-1} [F(0)-\alpha] \;, \quad
\nabla F(1) = B^{-1} [\beta - F(1)] }\; . 
\end{aligned}
\right.
\end{equation}
This result extends to exclusion processes with weak interactions at
the boundaries a theorem of Derrida, Lebowitz and Speer \cite{DLS2002}
for the case with strong interactions,

In this article, we provide an alternative proof of this result, based
on the strategy delineated in \cite{BDGJL} and carried out in
\cite{BDGJL2003} for exclusion processes with strong interactions at
the boundaries.

Using the Hamiltonian formalism, we derive a formal equation for the
path which solves the variational problem \eqref{7-01b}. The optimal
trajectory corresponds to a pair $(u(t), F(t))$ which solves a system
of coupled equations, see \eqref{4-11c}--\eqref{4-08c} below. While it
might seem, at a first glance, hopeless to prove any property of the
solutions to this pair of equations, it turns out that $F_t$ evolves
according to an autonomous equation, actually, according to the
hydrodynamic equation \eqref{1-06}. This remarkable property is the
key point and permits to prove all properties needed to show that the
candidate obtained from the heuristic argument is indeed the optimal
path.

According to \cite{BDGJL}, this optimal path, which describe how the
system adjusts to create a fluctuation of the density, corresponds to
the typical trajectory for the adjoint dynamics, reversed in time. In
particular, this approach reveals the adjoint hydrodynamic
equation. It is obtained from the hydrodynamic equation by adding a
non-local drift and modifying the boundary densities and intensities
of interaction which become time-dependent (cf. Remark \ref{rm5}).

To our knowledge this is one of the few examples of an interacting
particle system whose steady state is not known explicitly and whose
quasi-potential can be computed \cite{DLS2002, BDGJL2003, B2010,
DHS2021}. Moreover, the presence of boundary terms in the
Hamiltonian-Jacobi equation for the quasi-potential modifies entirely
the analysis of thermodynamic transformations of non-equilibrium
states carried out in \cite{BGJL2012} in the case of Dirichlet
boundary conditions. This is left for a future work, together with the
emergence of dynamical phase transitions \cite{BDGJL2011} and the
static large deviations \cite{BG2004, F2009, LT2018, FLT2019}.

\section{Notation and Results}
\label{sec1}

\smallskip\noindent{\bf The model.}  We consider one-dimensional,
symmetric exclusion processes in weak contact with boundary
reservoirs. Fix $N\ge 1$, and let $\color{bblue} \mf e_N = 1/N$,
$\color{bblue} \mf r_N = 1 - (1/N)$,
$\color{bblue} \Lambda_N=\{\mf e_N ,\dots, (N-2)\, \mf e_N, \mf
r_N\}$.  The sate space is represented by
$\color{bblue} \Omega_N=\{0,1\}^{\Lambda_N}$ and the configurations by
the Greek letters $\eta$, $\xi$ so that $\eta_x$, $x\in \Lambda_N$,
represents the number of particles at site $x$ for the configuration
$\eta$.

Fix throughout this article, $\color{bblue} 0<\alpha \le \beta<1$,
$\color{bblue} A>0$, $\color{bblue} B>0$.  The generator of the Markov
process considered here, represented by
$\ms L_N = \ms L^{\alpha, A, \beta, B}_N$, is given by
\begin{equation*}
\ms L_N \;=\;  L^{\rm lb}_{N} \;+\;  L^{\rm bulk}_{N}
\;+\;  L^{\rm rb}_{N}\;.
\end{equation*}
In this formula, for every function $f:\Omega_N\to \bb R$,
\begin{equation*}
(L^{\rm bulk}_{N}  f) (\eta) \;=\; N^2\,
\sum_{x\in \Lambda^o_N} [\, f(\sigma^{x,x+\mf e_N} \eta) \,-\, f(\eta)\,] \;, 
\end{equation*}
where $\Lambda^o_N$ represents the interior of $\Lambda_N$,
$\color{bblue} \Lambda^o_N := \Lambda_N \setminus \{\mf r_N\} = \{\mf
e_N ,\dots, (N-2)\, \mf e_N\}$, and
\begin{equation*}
\begin{gathered}
(L^{\rm lb}_N f) (\eta) \;=\; \frac{N}{A}\,
\big[\, (1-\eta_{\mf e_N})\, \alpha \,+\, (1-\alpha)\, \eta_{\mf e_N}\, \big]
\big[\, f(\sigma^{\mf e_N} \eta) \,-\, f(\eta)\, \big] \;, \\
(L^{\rm rb}_N f) (\eta) \;=\; \frac{N}{B}\,
\big[\, (1-\eta_{\mf r_N})\, \beta \,+\, (1-\beta)\, \eta_{\mf r_N}\, \big]
\big[\, f(\sigma^{\mf r_N} \eta) \,-\, f(\eta)\, \big] 
\;.
\end{gathered}
\end{equation*}
From now on, we omit the subindex $N$ of $\mf e_N$.  In the
formulas above, 
\begin{equation*}
(\sigma^{x,x+\mf e}\eta)_y \;=\;
\begin{cases}
\eta_y &\mbox{ if } y \neq x,x+\mf e\\
\eta_{x+\mf e} &\mbox{ if } y=x \\
\eta_x &\mbox{ if } y=x+\mf e
\end{cases}
\quad \mbox{ and }
\quad
(\sigma^{x}\eta)_y \;=\;
\begin{cases}
\eta_y &\mbox{ if } y\neq x\\
1-\eta_x &\mbox{ if } y=x \;.
\end{cases}
\end{equation*}

For a metric space $\bb X$, denote by $\color{bblue} D([0,T], \bb X)$,
$T>0$, the space of right-continuous functions
$\mf x\colon [0,T] \to \bb X$, with left-limits, endowed with the
Skorohod topology and its associated Borel $\sigma$-algebra. The
elements of $D([0,T],\Omega_N)$ are represent by
$\color{bblue} \bs \eta (\cdot)$.

For a probability measure $\mu$ on $\Omega_N$, let
$\color{bblue} \bb P^N_\mu$ be the measure on $D([0,T],\Omega_N)$
induced by the continuous-time Markov process associated to the
generator $\ms L_N$ starting from $\mu$. When the measure $\mu$ is the
Dirac measure concentrated at a configuration $\eta\in \Omega_N$, that
is $\mu = \delta_{\eta}$, we represent
$\color{bblue} \bb P^N_{\delta_\eta}$ simply by
$\color{bblue} \bb P^N_\eta$.  Expectation with respect to $\bb P^N_\mu$,
$\bb P^N_\eta$ is denoted by $\color{bblue} \bb E^N_\mu$,
$\color{bblue} \bb E^N_\eta$, respectively. When the context permits
we remove the index $N$ from the notation.

\smallskip\noindent{\bf Hydrodynamic limit.}  Denote by
$\color{bblue} \mathscr{M}$ the set of non-negative measures on
$[0,1]$ with total mass bounded by $1$ endowed with the weak
topology. Recall that this topology is metrisable and that, with this
topology, $\ms M$ is a relatively compact space. For a continuous
function $F:[0,1] \to \bb R$ and a measure $\pi\in\ms M$, denote by
$\<\pi, F\>$ the integral of $F$ with respect to $\mu$:
\begin{equation*}
\<\pi, F\> \;=\; \int F(x) \, \pi(dx)\;.
\end{equation*}

Given a configuration $\eta\in \Omega_N$, denote by $\pi = \pi(\eta)$
the measure in $\ms M$ obtained by assigning a mass $N^{-1}$ to the
position of each particle:
\begin{equation*}
\pi \;=\; \pi (\eta) \;=\;
\frac{1}{N}\sum_{x\in\Lambda_N}\eta_x\, \delta_{x}\;.
\end{equation*}
The measure $\pi$ is called the \emph{empirical measure}.

Denote by
$\color{bblue} \bs \pi: D([0,T],\Omega_N) \to D([0,T],\ms M)$ the map
which associates to a trajectory $\bs \eta (\cdot)$ its empirical
measure: 
\begin{equation*}
\bs \pi (t) \;=\; \pi (\bs \eta(t)) \;=\;
\sum_{x\in\Lambda_N}\eta_x(t) \, \delta_{x}\;.
\end{equation*}
For a probability measure $\mu$ in $\Omega_N$, let $\bb Q^N_{\mu}$ be
the measure on $D([0,T], \ms M)$ given by
$\color{bblue} \bb Q^N_{\mu} = \bb P^N_{\mu} \,\circ\, \bs \pi^{-1}$.

The first result, due to \cite{BMNS2017}, establishes the hydrodynamic
behavior of the empirical measure.

\begin{theorem}
\label{mt1}
Fix $T>0$, a density profile $\gamma \colon [0,1] \to [0,1]$, and sequence
$(\nu^N : N\ge 1)$ of probability measures on $\Omega_N$ associated to
$\gamma$ in the sense that
\begin{equation*}
\lim_{N\to\infty}\nu^N\Big[\, 
\Big|\, \<\pi\,,\, G\> \,-\, \int_0^1 \gamma (x)\, G (x)\, dx \,\Big|
\,>\, \delta\,\Big] \;=\; 0
\end{equation*}
for all continuous functions $G\colon [0,1] \to \bb R$ and $\delta>0$.
Then, the sequence of probability measures $\bb Q^N_{\nu^N}$ converges
to the probability measure $\bb Q$ concentrated on the trajectory
$\pi(t,dx) = u(t,x)\, dx$, where $u$ is the unique weak solution to
the heat equation with Robin's boundary conditions \eqref{1-06}.
\end{theorem}

We refer to Appendix \ref{sec05} for the definition of weak solutions to
equation \eqref{1-06} and some of its properties.

\smallskip\noindent{\bf Dynamical large deviations.}
For $T>0$ and positive integers $m,n$, denote by
$\color{bblue} C^{m,n}([0,T]\times[0,1])$ the space of functions
$G\colon [0,T]\times[0,1] \to\bb R$ with $m$ derivatives in time, $n$
derivatives in space which are continuous up to the boundary. Denote
by $\color{bblue} C^{m,n}_0([0,T]\times[0,1])$ the set of functions in
$C^{m,n}([0,T]\times[0,1])$ which vanish at the endpoints of $[0,1]$,
i.e.\ $G\in C^{m,n}([0,T]\times[0,1])$ belongs to
$C^{m,n}_0([0,T]\times[0,1])$ if and only if $G(t,0) = G(t,1)=0$ for
all $t\in[0,T]$.

Denote by $\ms M_{\rm ac}$ the subset of $\ms M$ of all measures which
are absolutely continuous with respect to the Lebesgue measure and
whose density takes values in the interval $[0,1]$:
$\color{bblue} \ms M_{\rm ac} = \{\pi \in \ms M : \pi(dx) = \gamma (x)
\, dx \;\;\text{and}\;\; 0\,\le\, \gamma(x) \,\le\, 1\,\}$.

For $T>0$, let the {\it energy}
$\mc Q_{[0,T]}: D([0,T], \ms M_{\rm ac}) \to [0,\infty]$ be given by
\begin{align*}
& \mc Q_{[0,T]} (\pi) \;=\;  \\
& \quad
\sup_{G} \Big\{ \int_0^Tdt \int_{0}^1  u(t,x) \,
(\nabla G)(t,x) \; dx \;-\; \frac 12  \int_0^Tdt \int_{0}^1 
\sigma (u(t,x)) \, G(t,x)^2  \; dx\, \Big\}\;, 
\end{align*}
where $\pi(t,dx) = u (t,x)\, dx$ and the supremum is carried over all
smooth functions $G: [0,T]\times (0,1) \to\bb R$ with compact support.

\begin{remark}
\label{rm2}
Hereafter, we abuse of notation writing $\gamma \in \ms M_{\rm ac}$ to
mean that the measure $\gamma(x)\, dx$ belongs to $\ms M_{\rm ac}$.
Moreover, for functionals
$\Phi \colon D([0,T], \ms M_{\rm ac}) \to \bb R$,
$W: \ms M_{\rm ac} \to \bb R$, we often write $\Phi(u)$, $W(\gamma)$
instead of $\Phi(\pi)$, $W(\mu)$ when $\pi(t,dx) = u(t,x)\, dx$,
$\mu(dx) = \gamma (x)\, dx$.
\end{remark}

\noindent{\bf Notational convention:} For a function
$v: I \times [0,1] \to \bb R$, where $I$ is a subset of $\bb R$, $v_t$
and $v(t)$, $t\in I$, represent the function $w: [0,1] \to \bb R$
defined by $w(x) = v(t,x)$. We use the letters $\gamma$, $\phi$,
$\psi$ to represent densities [elements of $\ms M_{\rm ac}$], $u$,
$v$, $w$ for trajectories of densities [elements of
$D(\bb R_+, \ms M_{\rm ac})$], and $F$, $G$, $H$ for external fields,
usually functions in $C(\bb R_+ \times [0,1])$. \smallskip

By \cite[Lemma 4.1]{BLM09}, the energy $\mc Q_{[0,T]}$ is convex and
lower semicontinuous. Moreover, if $\mc Q_{[0,T]}(u)$ is finite, $u$
has a generalized space derivative, denoted by $\nabla u$, and
\begin{equation*}
\mc Q_{[0,T]} (u) \;=\; \frac 12 \int_0^T dt\, \int_{0}^1 
\frac{(\nabla u_t)^2}{\sigma (u_t)} \; dx \;\cdot 
\end{equation*}

Fix a trajectory $\pi \in D([0,T], \ms M_{\rm ac})$,
$\pi(t,dx) = u(t,x) \, dx$, with finite energy,
$\mc Q_{[0,T]} (u) <\infty$. In particular,
$\int_{0}^T dt \int_0^1 (\nabla u_t)^2 \; dx$ is finite. By
\cite[Assertion 48, page 1030]{Zeid80}, the trace of $u$ at the spatial
boundary of the cylinder $\color{bblue} \Omega_T = [0,T] \times [0,1]$
is well defined. That is, the maps $t\mapsto u(t, 0)$,
$t\mapsto u(t, 1)$ are well defined and belong to
$\ms L^2([0,T])$. Moreover, since for almost all $t\in [0,T]$,
$\int_{0}^1 (\nabla u_t)^2 \; dx$ is finite, for these values of $t$,
$u(t,\cdot)$ is H\"older-continuous and $u(t,0)$ and $u(t,1)$ are well
defined.

Denote by $\<\,\cdot\,,\,\cdot\,\>$ the usual scalar product in $\ms
L^2([0,1])$: 
\begin{equation*}
\<\, f \,,\, g \,\> \;=\; \int_0^1 f(x) \, g(x) \; dx\;, \quad
f\,,\, g\,\in\, \ms L^2([0,1])\;.
\end{equation*}
Fix a function $\gamma\colon [0,1] \to [0,1]$, which corresponds to
the initial profile.  Denote by
$\color{bblue} D_{\mc E} ([0,T], \ms M_{\rm ac})$ the set of
trajectories in $D ([0,T], \ms M_{\rm ac})$ with finite energy, and by
$\color{bblue} D_{\gamma, \mc E} ([0,T], \ms M_{\rm ac})$ the set of
trajectories with finite energy and which start from $\gamma$, $u_0
(\cdot) = \gamma(\cdot)$ a.s.

Recall, from \eqref{4-05}, the definition of
$\mf b_{\varrho, D} (a,M)$ and, from Remark \ref{rm2}, the convention
on notation.  For each $H$ in $C^{1,2}([0,T]\times[0,1])$, let
$J_{T,H} \colon D_{\mc E}([0,T], \ms M_{\rm ac})\longrightarrow \bb R$
be the functional given by
\begin{equation}
\label{1-01}
\begin{aligned}
& J_{T,H} ( u ) \; =\; \big\langle u_T, H_T \big\rangle 
\;-\; \langle u_0 , H_0  \rangle
\;-\; \int_0^{T} \big\langle u_t, \partial_t H_t  \big\rangle \; dt
\\
&\quad \; -\; \int_0^{T}  \big\langle u_t , \Delta H_t \big\rangle \; dt
\;+\;  \int_0^{T} u_t(1)\,  \nabla H_t (1) \; dt 
\; -\;   \int_0^{T} u_t(0)\,  \nabla H_t(0) \; dt  \\
&\quad 
\;-\; \int_0^{T}  \big\langle \sigma( u_t ), 
\big( \nabla H_t \big)^2 \big\rangle \; dt \\
&\quad  \;-\; \int_0^T \Big\{\, 
\mf b_{\alpha, A} \big(\, u_t(0)\, ,\, H_t (0)\, \big)
\;+\; \mf b_{\beta, B} \big(\, u_t(1)\, ,\, H_t(1)\, \big)\,
\Big\}\, dt \; .
\end{aligned}
\end{equation}
The right-hand side is well defined because the functions
$u(\cdot, 0)$, $u(\cdot, 1)$ belong to $\ms L^2([0,T])$.

Since trajectories in $D_{\mc E}([0,T], \ms M_{\rm
ac})$ have generalized space-derivatives, we may integrate by parts
the second line and write the functional $J_{T,H} ( \, \cdot\,)$
as
\begin{equation}
\label{1-01b}
\begin{aligned}
J_{T,H} ( u ) \; & =\; \big\langle u_T, H_T \big\rangle 
\;-\; \langle u_0 , H_0  \rangle
\;-\; \int_0^{T} \big\langle u_t, \partial_t H_t  \big\rangle \; dt
\\
& +\; \int_0^{T} \big\langle \nabla u_t , \nabla H_t \big\rangle \; dt
-\; \int_0^{T}  \big\langle \sigma( u_t ), 
\big( \nabla H_t \big)^2 \big\rangle \; dt \\
&-\; \int_0^T \Big\{\, 
\mf b_{\alpha, A} \big(\, u_t(0)\, ,\, H_t (0)\, \big)
\;+\; \mf b_{\beta, B} \big(\, u_t(1)\, ,\, H_t(1)\, \big)\,
\Big\}\, dt \; .
\end{aligned}
\end{equation}

Let
$I_{[0,T]} \colon D_{\mc E} ([0,T],\ms M_{\rm ac}) \rightarrow
[0,+\infty]$ be the functional defined by
\begin{equation*}
I_{[0,T]} (\pi) \; :=\;
\sup_{H\in   C^{1,2} ([0,T]\times[0,1])} J_{T,H}(\pi)\; .
\end{equation*}
Fix a density profile $\gamma$ in $\ms M_{\rm ac}$, and let
$I_{[0,T]} ( \, \cdot\, | \gamma) \colon D ([0,T], \ms M) \to \bb R$
be given by
\begin{equation}
\label{1-02}
I_{[0,T]}  ( \pi | \gamma) \;=\; 
\left\{
\begin{aligned}
& I_{[0,T]} ( \pi ) \quad 
\text{if}\;\; \pi \, \in\, D_{\gamma, \mc E}([0,T], \ms M_{\rm
ac})\;, \\
& \infty\quad  \text{otherwise}\;.
\end{aligned}
\right.
\end{equation}
We review in Section \ref{sec01} some properties of the functional
$I_{[0,T]} (\cdot | \gamma)$ obtained in \cite{FGLN2021}. Next result
is the main theorem in \cite{FGLN2021}.

\begin{theorem}
\label{s02}
Fix $T>0$ and a measure $\pi(dx) = \gamma(x)\, dx$ in
$\ms M_{\rm ac}$.  Consider a sequence $\eta^N$ of configurations
associated to $\gamma$. Then the measure $\bb Q_{\eta^N}$ satisfies a
large deviation principle with speed $N$ and rate function
$I_{[0,T]}(\cdot|\gamma)$. Namely, for each closed set
$\mc C \subset D([0,T], \ms M)$ and each open set
$\mc O \subset D([0,T], \ms M)$,
\begin{equation*}
\begin{aligned}
& \limsup_{N\to\infty} \frac 1N \log \bb P_{\eta^N} [ \bs \pi \in \mc C]
\;\le\; - \inf_{\pi \in \mc C} I_{[0,T]} (\pi | \rho)  
\\
&\quad \liminf_{N\to\infty} \frac 1N \log \bb P_{\eta^N} [ \bs \pi \in
\mc O] \;\ge\; - \inf_{\pi \in \mc O} I_{[0,T]} (\pi | \rho) \; .
\end{aligned}
\end{equation*}
\end{theorem}

\smallskip\noindent{\bf The quasi-potential.}  Denote by
$V:\ms M_{\rm ac} \to \bb R_+$ the quasi-potential associated to the
rate function $I_{[0,T]}(\,\cdot\,|\, \gamma)$. It is given by
\begin{equation}
\label{1-04} 
V(\gamma) \;:=\; \inf_{T>0}\; \inf_{u(\cdot)} \;
I_{[0,T]}(u\,|\, \bar\rho) \; ,
\end{equation}
where the infimum is carried over all paths $u$ in
$D([0,T], \ms M_{\rm ac})$ such that $u(0) = \bar\rho$, $u(T)=\gamma$.
The quasi-potential $V(\gamma)$ measures the minimal cost to produce a
profile $\gamma$ starting from $\bar\rho$. 

Denote by $\color{bblue} C^1\big([0,1]\big)$ the space of once
continuously differentiable functions $F:[0,1]\to \bb R$ endowed with
the norm
$\color{bblue} \|F\|_{C^1} := \sup_{x\in[0,1]} \big\{ |F(x)| + |\nabla
F(x)| \big\}$.  Let $\mc {F}$ be the space of monotone $C^1$
functions:
\begin{equation}
\label{2-01}
\mc {F} :=\big\{\, F \in  C^1\big([0,1]\big) \,:\:
\alpha \,<\, F(x)\, < \, \beta \;,\;
\nabla F (x) > 0 \;\; \forall\;  x\in [0,1] \, \big\} \; .
\end{equation}
Denote by $\ms G_{\rm bulk}$,
$\ms G \colon \ms M_{\rm ac} \times \mc F \to \bb R$ the
functionals given by
\begin{equation*}
\begin{gathered}
\ms G_{\rm bulk} (\gamma , F) :=  
\int_{0}^1  \Big\{\, \gamma (x) \, \log
\frac {\gamma (x)}{F(x)} + \big[1 - \gamma (x)\big]
\log \frac {1- \gamma (x)}{1- F(x)}
+ \log \frac {\nabla F(x)}{[\beta - \alpha]}\,
\Big\} \; dx\;, \\
\ms G (\gamma, F) \;:=\; \ms G_{\rm bulk} (\gamma, F)  \;+\; A\,
\ln \frac{F(0) - \alpha}{A(\beta-\alpha)}
\;+\; B\, \ln \frac{\beta - F(1)}{B (\beta-\alpha)} \;\cdot
\end{gathered}
\end{equation*}
Define $S_0$, $S \colon \ms M_{\rm ac} \to \bb R$, by
\begin{equation}
\label{2-03}
S_0 (\gamma) \;:=\; \sup_{F\in\mc F} \ms G (\gamma, F)  \;, \quad
S (\gamma) \;:=\; S_0 (\gamma)  \,-\, S_0 (\bar\rho)  \;.
\end{equation}

\smallskip\noindent{\bf Main results.}  The first main assertion of
the article, Theorem \ref{t01}, affirms that, for each
$\gamma\in \ms M_{\rm ac}$, the non-linear boundary-value problem
\eqref{3-01} has a unique solution in $\mc F$.  Its precise statement
is left to Section \ref{sec02} because it requires some notation.

\begin{theorem}
\label{t02}
The functional $S\colon \ms M_{\rm ac} \to \bb R$ defined in
\eqref{2-03} is bounded, convex and lower semi-continuous.  Moreover,
$S_0(\gamma) = \ms G (\, \gamma, F(\gamma) \,)$, where $F(\gamma)$ is the
solution to \eqref{3-01}.
\end{theorem}

\begin{remark}
\label{rm08}
When $\gamma = \bar\rho$, $F=\bar\rho$ is the solution to
\eqref{3-01}. Replacing $F$ by $\bar\rho$ in the formula for $\ms G$
yields that $S_0(\bar\rho) \,=\, -\, (1+A+B) \, \log (1+A+B)$.
\end{remark}

Next theorem asserts that the functionals defined through the
variational problems \eqref{1-04} and \eqref{2-03} coincide. In
particular, it gives an ``explicit'' formula for the dynamical
variational problem \eqref{1-04} defining the quasi-potential.

\begin{theorem}
\label{qp=s}
For each $\gamma \in \ms M_{\rm ac}$, $V(\gamma)=S(\gamma)$. In
particular, the functional $S$ is non-negative. That is, the
functional $S_0$ attains its minimum at $\bar\rho$.
\end{theorem}

\begin{remark}
Theorems \ref{t02} and \ref{qp=s} formalize the arguments presented in
\cite{DHS2021}, where Derrida, Hirschberg and Sadhu derived the steady
state large deviations functional by representing the steady state as
a product of matrices.
\end{remark}

\begin{remark}
\label{rm4}
If $\alpha = \beta$, the exclusion dynamics is reversible and the
stationary state is the Bernoulli product measure with density
$\alpha$. In particular, in this case
\begin{equation*}
S(\gamma) \;=\; \int_{0}^1  \Big\{\, \gamma (x) \, \log
\frac {\gamma (x)}{\alpha} + \big[1 - \gamma (x)\big]
\log \frac {1- \gamma (x)}{1- \alpha}\, \Big\} \; dx\;.
\end{equation*}
The same strategy as in the proof of Theorem \ref{qp=s} yields that
$V=S$. The arguments are much simpler because the adjoint dynamics
coincides with the original one as the process is reversible.
\end{remark}

The method of the proof of Theorems \ref{t02} and \ref{qp=s} is the
one proposed in \cite{BDGJL} and carried out in \cite{BDGJL2003} for
exclusion processes with strong interaction with the reservoirs.

The fact that the densities are not fixed by the dynamics at the
boundary and the presence of exponential terms at the boundary in the
dynamical large deviations functionals (cf. the definition of
$\mf b_{\varrho, D}$), introduce many new difficulties. The proof of
the uniqueness of solutions to \eqref{3-01} is one of them (cf. Proof
of Theorem \ref{t01}).

The proof of the upper bound for the quasi-potential is a second
example. As the boundary conditions are not fixed, to prove that the
solutions to the adjoint hydrodynamic equations are bounded away from
$0$ and $1$, we need to investigate the behavior of the solutions at
the boundary. This is done in the proof of Proposition \ref{l10}.

The article is organized as follows. In Section \ref{sec06} we present
a heuristic derivation of Theorem \ref{qp=s} based on the Hamiltonian
formalism of rational mechanics. This argument explains the strategy
adopted in the following sections.  In Section \ref{sec01}, we recall
some properties of the dynamical rate function
$I_{[0,T]} ( \,\cdot\, | \gamma)$ obtained in
\cite{FGLN2021}. Theorems \ref{t02}, \ref{qp=s} are proved in Sections
\ref{sec02}, \ref{sec03}, respectively. In Appendices \ref{sec04} and
\ref{sec05} we present some results on the Robin Laplacian and on
solutions to heat equations with mixed boundary conditions needed in
the proofs of the main theorems.


\section{Sketch of the proof of Theorem \ref{qp=s}.}
\label{sec06}

The arguments below are formal, but explain the idea of the proof.  We
follow the strategy proposed in \cite{BDGJL2003}, in the context of
boundary driven symmetric simple exclusion processes with strong
interaction with the boundaries, to derive a formula for the
quasi-potential based on the Hamiltonian formalism.  We also introduce
the hydrodynamic equation of the adjoint process, which describes how the
dynamics acts to create an anomalous density profile.  This section
also serves as a road map to prove Theorem \ref{qp=s} in other
contexts.

Recall from \eqref{1-09} the definition of the Hamiltonian $\ms H$.
With this notation and an integration by parts in time, the functional
$I_{[0,T]}$ can be written as
\begin{equation}
\label{7-01}
I_{[0,T]} (\pi | \gamma) \; =\;
\sup_{H}
\int_0^T \big\{\,
\big\langle \partial_t u_t,  H_t  \big\rangle
\,-\, \ms H(u_t, H_t) \, \big\} \; dt \; .
\end{equation}
Hence, the functional $I_{[0,T]} (\,\cdot\,|\gamma)$ corresponds to
the action functional associated to the Hamiltonian $\ms H$.

A variational calculation yields that the quasi-potential satisfies
the Hamilton-Jacobi equation: for every $\gamma\in \ms M_{\rm ac}$,
\begin{equation}
\label{7-04}
\ms H \big (\,  \gamma \,,\, \frac{\delta V}{\delta \gamma}\, \big)  \;=\; 0\;.
\end{equation}
where $\delta V/\delta \gamma$ stands for the functional derivative of $V$.

Fix $\gamma\in \ms M_{\rm ac}$ and let
$\color{bblue} \Gamma = \log [\gamma/(1-\gamma)] - \log[F/(1-F)]$ for
some function $F$ taking values in the interval $(0,1)$. Lemma
\ref{l06} asserts that, if $\gamma$ is smooth and bounded away from
$0$ and $1$, $\Gamma$ solves the Hamilton-Jacobi equation
\begin{equation}
\label{7-02}
\ms H(\gamma, \Gamma) \;=\; 0
\end{equation}
if $F$ is the solution to \eqref{3-01}, that is, if $F=F(\gamma)$ with
the notation introduced in the statement of Theorem
\ref{t02}. By \eqref{7-04} and \eqref{7-02},
\begin{equation}
\label{7-03}
\frac{\delta V}{\delta \gamma} \;=\;
\log \frac{\gamma}{1-\gamma} \;-\; \log \frac{F}{1-F} \;\cdot
\end{equation}

To build a functional $S$ which satisfies \eqref{7-03}, we look
for a functional $\ms W(\gamma, F)$, with two properties:
\begin{itemize}
\item[(a)] For every $\gamma \in \ms M_{\rm ac}$,
\begin{equation*}
\frac{\delta \ms W}{\delta \gamma} (\gamma, F) \;=\;
\log \frac{\gamma}{1-\gamma} \,-\,
\log \frac{F}{1-F}\;,
\end{equation*}
\item[(b)] For each $\gamma\in \ms M_{\rm ac}$, the solution
$F(\gamma)$ of equation \eqref{3-01} is a critical point of
$\ms W(\gamma, \cdot)$;
\end{itemize}
Under these assumptions, defining $S(\gamma)$ as
$\ms W(\gamma, F(\gamma))$, we have
\begin{equation*}
\frac{\delta S}{\delta \gamma} (\gamma) \;=\;
\frac{\delta \ms W}{\delta \gamma} (\gamma, F(\gamma)) \;+\;
\frac{\delta \ms W}{\delta F} (\gamma, F (\gamma)) \,
\frac{\delta F}{\delta \gamma} (\gamma) \;=\;
\frac{\delta \ms W}{\delta \gamma} (\gamma, F(\gamma)) \;\cdot
\end{equation*}
The last identity follows from property (b) of the functional $\ms W$
[$\delta \ms W/\delta F = 0$ at $(\gamma, F (\gamma))$]. By property
(a), the right-hand side is equal to
$\log [\gamma/(1-\gamma)] - \log[F(\gamma)/(1-F(\gamma))]$, proving
that \eqref{7-03} is fulfilled.

This computation explains the introduction of the functional
$\ms G(\gamma, F)$, defined below \eqref{2-03}. It is obtained by
integrating \eqref{7-03} in $\gamma$ and adding terms which depend
only on $F$ to match condition (b). The functional $\ms G$ satisfies
properties (a) and (b), as it is easy to show that \eqref{3-01}
corresponds to the Euler-Lagrange equation of the functional
$\ms G(\gamma, \,\cdot\,)$.

We turn to the proof that $V=S$. Fix $\gamma \in \ms M_{\rm ac}$,
$T>0$ and a trajectory $u_t$, $0\le t\le T$, such that
$u_0= \bar\rho$, $u_T=\gamma$. Let $F_t$ be the solution to
\eqref{3-01} with $u_t$ replacing $\gamma$, $F_t = F(u_t)$. By
\eqref{7-01}, \eqref{7-02},
\begin{equation*}
I_{[0,T]} (\pi | \gamma) \; \ge \;
\int_0^T \big\langle \partial_t u_t,  \Gamma_t  \big\rangle
\; dt \; ,
\end{equation*}
where $\Gamma_t = \log [u_t/(1-u_t)] - \log[F_t/(1-F_t)]$.  In view of
\eqref{7-03}, replacing $\Gamma_t$ by $(\delta S/\delta \gamma)(u_t)$
yields that
\begin{equation*}
I_{[0,T]} (\pi | \gamma) \; \ge \; S(u_T) \,-\, S(u_0)\;=\;
S(\gamma) \,-\, S(\bar\rho) \; ,
\end{equation*}
so that
\begin{equation*}
V(\gamma) \;\ge \;
S(\gamma) \,-\, S(\bar\rho) \; .
\end{equation*}

We proceed with the upper bound. By \cite{BDGJL}, the optimal
trajectory for the variational problem \eqref{1-04} is the
hydrodynamic trajectory of the adjoint dynamics reversed in
time. Moreover, according to \cite{BDGJL}, the adjoint dynamics is
given by
\begin{equation*}
\partial_t v \,=\,  -\, \Delta v  \,+\, 2\,  \nabla \big( \, \sigma (v) 
\, \nabla  \frac{\delta S}{\delta v} \, \big)
\end{equation*}
In view of \eqref{7-03}, replacing $\delta S/\delta v$ by
$\log [v_t/(1-v_t)] - \log [F_t/(1-F_t)]$, where $F_t$ is the solution
to \eqref{3-01} with $v_t$ in place of $\gamma$ yields the equation
\begin{equation*}
\partial_t v \,=\,  \Delta v  \,-\, 2\,  \nabla \big( \, \sigma (v) 
\, \nabla  \log \frac{F}{1-F} \, \big)\;.
\end{equation*}
Adding the boundary and initial conditions, as well as the equation
for $F$, the previous equation becomes the system of equations
\begin{equation}
\label{4-11c}
\left\{
\begin{aligned}
& \partial_t v_t \,=\,  \Delta v_t  \,-\, 2\,  \nabla \big( \, \sigma (v_t) 
\, \nabla R_t \, \big)   \quad (t,x) \in
(0,\infty)\times (0,1) \;, \\
& \nabla v_t (1) \,-\, 2\, \sigma(v_t(1)) \, \nabla R_t(1) \,=\,
\mf p_{1-\beta, B} \big(\, v_t(1)\,,\, R_t(1)\, \big) \;, \\
& \nabla v_t (0) \,-\, 2\, \sigma(v_t(0)) \, \nabla R_t(0) \,=\,
-\, \mf p_{1-\alpha, A} \big(\, v_t(0)\,,\, R_t(0)\, \big) \;, \\
& v_0(\cdot) \,=\, \gamma (\cdot) \;, \quad  x\in [0,1]\;,
\end{aligned}
\right. 
\end{equation}
\begin{equation}
\label{4-08c}
\left\{
\begin{aligned}
& \Delta F_t = 
\big( v_t - F_t \big) \frac{\big( \nabla F_t \big)^2}{F_t(1-F_t)}
\quad (t,x) \in (0,\infty)\times (0,1) \; , \\
& 
\nabla F_t(0) = A^{-1} [F_t(0)-\alpha] \;, \quad
\nabla F_t(1) = B^{-1} [\beta - F_t(1)] \; .
\end{aligned}
\right.
\end{equation}
In this formula, $ R_t = \log [ F_t/(1-F_t)]$ and,
for $0< \varrho <1$, $D>0$, $0<a<1$, $M\in \bb R$, 
\begin{equation}
\label{5-02a}
\mf p_{\varrho, D} (a,M) \;=\; \frac{1}{D}\, \Big\{ [1-a]\, \varrho \,
e^M \;-\; a\, [1-\varrho] \, e^{-M}  \, \Big\}  \;.
\end{equation}

The first part of the proof of the upper bound consists in showing
that this trajectory is indeed the optimal one.  Lemma \ref{l08},
whose proof relies on the explicit expression for the rate functional
presented in Lemma \ref{l09}, states that this trajectory is optimal
provided the solution $v(t)$ to this equation relaxes to $\bar\rho$ as
$t\to\infty$.

To prove that $v_t$ relaxes to $\bar\rho$ or any other property of the
non-local system of equations \eqref{4-11c}--\eqref{4-08c} looks
hopeless. It turns out, however, that these equations can be expressed
in a simple form. The reason is that $F_t$ in
\eqref{4-11c}--\eqref{4-08c}, which corresponds to the momentum in the
Hamiltonian formalism, evolves according to an autonomous equation, a
remarkable and unexpected property.

Fix $\gamma\in \ms M_{\rm ac}$, and denote by $F^{(\gamma)}$ the
solution to \eqref{3-01}. Let $F^{(\gamma)}_t$ be the solution to the
heat equation \eqref{1-06} with initial condition $F^{(\gamma)}$
instead of $\gamma$. Define $v^{(\gamma)}_t$ as
\begin{equation}
\label{2-06} 
v^{(\gamma)}(t) \;:=\; F^{(\gamma)}(t) \; +\;  F^{(\gamma)}(t)\,
[1-F^{(\gamma)}(t)] \,
\frac {\Delta F^{(\gamma)}(t)}{ \big(\nabla F^{(\gamma)}(t)\big)^2} \; \cdot  
\end{equation}
By \eqref{3-01}, $v^{(\gamma)}(0)=\gamma$. Actually,
$F^{(\gamma)}_t = F(v^{(\gamma)}_t)$ for all $t\ge 0$, where
$F(v^{(\gamma)}_t)$ is the solution to \eqref{3-01} with $\gamma$
replaced by $v^{(\gamma)}_t$.

Proposition \ref{l10} asserts that for each $\gamma\in C^1([0,1])$
the pair $(v^{(\gamma)}_t, F^{(\gamma)}_t)$ solves the system of
equations \eqref{4-11c}--\eqref{4-08c}. This result provides,
therefore, an alternative and simple formulation of these equations.
Moreover, by Lemma \ref{l12},
$\lim_{t\to\infty} v^{(\gamma)}(t)=\bar\rho$, and, by Lemma \ref{l14},
the optimal path which solves the variational problem \eqref{1-04},
represented by $u^{(\gamma)}_{\rm opt}(t)$, defined on the time
interval $(-\infty, 0]$ instead of $[0,+\infty)$ as in \eqref{1-04},
is given by $u^{(\gamma)}_{\rm opt}(t)=v^{(\gamma)}(-t)$. Note that
$u^{(\gamma)}_{\rm opt}(0) = \gamma$,
$\lim_{t\to-\infty} u^{(\gamma)}_{\rm opt}(t) = \bar\rho$. Hence,
$u^{(\gamma)}_{\rm opt}$ connects $\bar\rho$ to $\gamma$ in a infinite
time window.

\begin{remark}
\label{rm5}
The boundary conditions in \eqref{4-11c} can be written as
\begin{equation}
\label{4-11d}
\left\{
\begin{aligned}
& \nabla v_t (0) \,-\, 2\, \sigma(v_t(0)) \, \nabla R_t(0) \,=\,
\frac{1}{A^*_t} \,[\, v(t,0) - \alpha^*_t \,]  \;.
\\
& \nabla v_t (1) \,-\, 2\, \sigma(v_t(1)) \, \nabla R_t(1) \,=\,
\frac{1}{B^*_t} \,[\, \beta^*_t - v(t,1)\,] \;, 
\end{aligned}
\right. 
\end{equation}
where
\begin{gather*}
\frac{1}{B^*_t}\;=\;
\frac{1}{B}\, \big\{\, (1-\beta)\, e^{R_t(1)} \;+\; \beta\,
e^{-R_t(1)}\, \big\} \;,
\quad \beta^*_t \;=\;
\frac{(1-\beta)\, e^{R_t(1)}}
{(1-\beta)\, e^{R_t(1)} \;+\; \beta\, e^{-R_t(1)}} \\
\frac{1}{A^*_t}\;=\; \frac{1}{A}\, \big\{\, (1-\alpha)\, e^{R_t(0)} \;+\;
\alpha\, e^{-R_t(0)} \, \big\} \;, \quad \alpha^*_t \;=\;
\frac{(1-\alpha)\, e^{R_t(0)}}
{(1-\alpha)\, e^{R_t(0)} \;+\; \alpha\, e^{-R_t(0)}}\;\cdot
\end{gather*}

In view of \eqref{4-11d}, equation \eqref{4-11c} corresponds to the
hydrodynamic equation of the weakly asymmetric exclusion process with
weak interactions at the boundary (cf. \cite{FGLN2021}).  Note that
this equation carries a positive drift to the right because
$\nabla \log [F/(1-F)] >0$. Moreover, the boundary densities and the
intensity of the interactions $\alpha$, $\beta$, $A$, $B$ are
time-dependent and given by $\alpha^*_t$, $\beta^*_t$, $A^*_t$,
$B^*_t$, respectively.
\end{remark}

\begin{remark}
\label{rm7}
As mentioned above, equations \eqref{4-11c}--\eqref{4-08c} represent
the adjoint hydrodynamic equation [that is, the PDE which describes
the evolution of the density under the adjoint dynamics]. Hence, in
the adjoint dynamics, the density evolves according to a weakly
asymmetric exclusion process. The drift at time $t$ is
$\log [F_t/(1-F_t)]$, where $F_t$ is the solution to \eqref{3-01} with
$\gamma$ replaced by the density profile at time $t$. The boundary
densities and intensities are given by the equations below
\eqref{4-11d}. 
\end{remark}

\begin{remark}
\label{rm6}
Notwithstanding the fact that the boundary densities have been
modified and a drift added, a straightforward computation shows that
the stationary profile of equation \eqref{4-11c} is still $\bar\rho$,
the stationary profile of the hydrodynamic equation \eqref{1-06}. More
precisely, as $F_t = F^{(\gamma)}$ solves equation \eqref{1-06},
$F_t \to \bar\rho$ as $t\to\infty$. Replace in equation \eqref{4-11c}
$R_t$ by $\log [\bar\rho /(1-\bar\rho)]$ and consider the associated
stationary equation [that is, replace $\partial_t v_t$ by $0$,
consider this equation in the space variable only and remove the
initial condition]. It's easy to check that $\bar\rho$ fulfills this
stationary equation.
\end{remark}

\section{The dynamical rate function}
\label{sec01}

For the reader's convenience, we recall here some properties of the
rate functional $I_{[0,T]} ( \,\cdot\, | \gamma)$ proved in
\cite{FGLN2021}.

The first estimate asserts that the cost of a trajectory in a interval
$[0,T]$ is bounded by the sum of its cost in the intervals $[0,S]$ and
$[S,T]$.

Let $\color{bblue} \tau_r u: \bb R_+ \times [0,1] \to \bb R$, $r>0$,
be the function defined by $\tau_r u (t,x) = u(t+r,x)$. For all
$\pi(t,dx) = u(t,x)\, dx$ in $D([0,T],\ms M_{\rm ac})$ and $0<S<T$,
\begin{equation}
\label{4-02b}
I_{[0,T]} (u\, |\, \gamma) \;\le\; I_{[0,S]} (u\, |\, \gamma)
\;+\; I_{[0,T-S]} (\tau_S u\, |\, u(S,\,\cdot\,))\;.
\end{equation}

\begin{theorem}
\label{mt2b} 
Fix $T>0$ and $\gamma\in \ms M_{\rm ac}$.  The function
$I_{[0,T]}(\cdot|\gamma):D([0,T],\mathcal{M})\to[0,\infty]$ is convex,
lower semicontinuous and has compact level sets.
\end{theorem}

\begin{definition}
\label{d04b}
Given $\gamma\in \ms M_{\rm ac}$, let $\Pi_\gamma$ be the collection
of all paths $\pi(t,dx) = u(t,x) dx$ in $D([0,T], \ms M_{\rm ac})$
such that
\begin{itemize}
\item[(a)] There exists $\mf t >0$, such that $u$ follows the hydrodynamic
equation \eqref{1-06} in the time interval $[0, \mf t]$. In
particular, $u(0,\cdot) = \gamma (\cdot)$.

\item[(b)] For every $0<\delta \le T$, there exists $\epsilon>0$ such that
$\epsilon \le u(t,x) \le 1-\epsilon$ for all $(t,x)$ in
$[\delta , T]\times[0,1]$;

\item[(c)] $u$ is smooth on $(0,T]\times [0,1]$.  
\end{itemize}
\end{definition}

\begin{theorem}
\label{mt3b} 
Fix $\gamma\in \ms M_{\rm ac}$.  For all $\pi$ in $D([0,T],\ms M)$
such that $I_{[0,T]}(\pi|\gamma)<\infty$, there exists a sequence
$\{\pi^{n}:n\ge1\}$ in $\Pi_\gamma$ such that $\pi^{n}$ converges to
$\pi$ in $D([0,T],\ms M)$ and $I_{[0,T]}(\pi^{n}|\gamma)$ converges to
$I_{[0,T]}(\pi|\gamma)$. Moreover, if there exists $\epsilon_0 >0$
such that $\epsilon_0 \le \gamma \le 1-\epsilon_0$, condition (b) in
Definition \ref{d04b} can be replaced by the existence of $\epsilon >0$
such that $\epsilon \le u(t,x) \le 1-\epsilon$ for all
$(t,x) \in [0,T] \times [0,1]$.
\end{theorem}

Let $\color{bblue} \Omega_T$ be the cylinder $(0,T)\times (0,1)$.  Fix
$\pi$ in $D([0,T], \ms M_{\rm ac})$, $\pi(t,dx) = u(t,x)\, dx$.  Let
$\color{bblue} \mc H^1(\Omega_T)$ be the Hilbert spaces induced by the
sets $C^\infty (\Omega_T)$ endowed with the scalar products,
$\<\!\< G,H \>\!\>_{1,2}$ defined by
\begin{equation*}
\begin{gathered}
\<\!\< G,H \>\!\>_{1,2} \;=\; 
\int_0^T dt \int_0^1 G_t\, H_t\; dx \;+\;
\int_0^T dt \int_0^1 \nabla G_t\, \nabla H_t\; dx\;.
\end{gathered}
\end{equation*}

Recall from \eqref{5-02a} the definition of
$\mf p_{\varrho, D} (a,M)$.  For $0< \varrho <1$, $D>0$, $0<a<1$,
$M\in \bb R$, let
\begin{equation}
\label{5-02}
\mf c_{\varrho, D} (a,M) \;=\; \frac{1}{D}\, \Big\{ [1-a]\, \varrho \,
[1 - e^M + M e^M ] \;+\; a\, [1-\varrho] \,
[\, 1 - e^{-M}  - M  e^{-M}] \, \Big\}  \;.
\end{equation}

\smallskip\noindent{\it Note:} for a trajectory $u_t$ such that
$\delta \le u (t,x)\le 1-\delta$ for all
$(t,x) \in [0,T] \times [0,1]$, the space $\mc H^1(\Omega_T)$
introduced above coincides with the space $\mc H^1(\sigma(u))$
introduced in \cite{FGLN2021}.

\begin{lemma}
\label{l09} 
Fix a trajectory $\pi$ in $D([0,T], \ms M_{\rm ac})$,
$\pi(t,dx) = u(t,x)\, dx$. Assume that
$u\in C^{1,2} ([0,T]\times[0,1])$, there exists $\delta>0$ such that
$\delta \le u(t,x) \le 1-\delta$ for all
$(t,x) \in [0,T] \times [0,1]$ and $I_{[0,T]} (u \,|\, \gamma)$ is
finite, where $\gamma = u_0$.  Then, there exists a function $H$ in
$\mc H^1(\Omega_T)$ such that $u$ is the unique weak solution to
\begin{equation}
\label{5-01} 
\left\{
\begin{aligned}
& \partial_t u \;=\; \Delta u \,-\,
2\, \nabla \{ \sigma(u) \ \nabla H \}\; , \\
& \nabla u_t (1) \,-\, 2\, \sigma(u_t(1)) \, \nabla H_t(1) \,=\,
\mf p_{\beta, B} \big(\, u_t(1)\,,\, H_t(1)\, \big) \;, \\
& \nabla u_t (0) \,-\, 2\, \sigma(u_t(0)) \, \nabla H_t(0) \,=\,
-\, \mf p_{\alpha, A} \big(\, u_t(0)\,,\, H_t(0)\, \big) \;, \\
& u(0, \cdot) = \gamma (\cdot)\; .
\end{aligned}
\right.
\end{equation}
Moreover, 
\begin{equation}
\label{5-03} 
\begin{aligned}
I_{[0,T]} (u \,|\, \gamma) \; & =\; \int_0^T \<\, \sigma(u_t)\,,\,
( \nabla H_t)^2\, \> \; dt +\; \int_0^T
\mf c_{\beta, B} \big(\, u_t(1)\,,\, H_t(1)\, \big)\; dt \\
\; & +\; \int_0^T
\mf c_{\alpha, A} \big(\, u_t(0)\,,\, H_t(0)\, \big)\; dt \;.
\end{aligned}
\end{equation}
\end{lemma}

Weak solutions to equation \eqref{5-01} are introduced in Definition
\ref{d03}. Theorem \ref{mt5} states that for each $\gamma\in\ms
M_{\rm ac}$ there exists one and only one weak solution.

\section{The Euler-Lagrange equation for $S$}
\label{sec02}

The Euler--Lagrange equation associated to the variational problem
\eqref{2-03} is given by the non-linear equation with Robin boundary
conditions \eqref{3-01}. In this section, we provide a precise meaning
to this equation, prove existence and uniqueness of solutions, and
prove Theorem \ref{t02}. The approach is taken from \cite{BDGJL2003},
but there is a serious technical difficulty in the proof of
uniqueness. The idea there is to extend the problem to the interval
$[-A, 1+B]$, see Lemma \ref{l03} and the proof of Theorem \ref{t01}.

Recall the definition of $\mc F$ introduced in \eqref{2-01}. For
$F\in \mc F$, let
\begin{equation}
\label{3-02}
\mc R_\gamma (x) \;=\; \mc R_\gamma (F\,;\, x)=
\big[\, \gamma(x)-F(x) \,\big]\,
\frac{\nabla F (x)}{F(x)\, [1-F(x)]}\;, 
\end{equation}
With this notation, equation \eqref{3-01} takes the form
\begin{equation}
\label{3-04}
\left\{
\begin{aligned}
& \Delta F \;=\;  \nabla F\, \mc R_\gamma \;, \\
& \vphantom{\Big(}
{\displaystyle \nabla F(0) = A^{-1} [F(0)-\alpha] \;, \quad
\nabla F(1) = B^{-1} [\beta - F(1)] }\; . 
\end{aligned}
\right.
\end{equation}

To prove the existence and uniqueness of a solution to \eqref{3-01},
following \cite{BDGJL2003}, we formulate \eqref{3-04} as the
integro--differential equation
\begin{equation}
\label{3-03}
F(x) \;=\; \alpha \;+\; (\beta-\alpha)\,
\frac{A \;+\; \int_{0}^x
\exp \{ \int_{0}^y \mc R_\gamma(F;z)\, dz \}\, dy}
{A + \int_{0}^1  \exp \{ \int_{0}^y \mc R_\gamma(F;z)\, dz \}
\, dy + B\, 
\exp \{ \int_{0}^1 \mc R_\gamma(F;y)\, dy\}} \; .
\end{equation}

\begin{remark}
\label{rm02}
If $\gamma=\bar\rho$, then $F= \bar\rho$ solves \eqref{3-01} and
\eqref{3-03}.  Moreover, if $F\in C^2([0,1])$ is a solution to the
problem (\ref{3-01}) such that $\nabla F(x)>0$ for $x\in[0,1]$, then
$F$ is also a solution to the integro--differential equation
(\ref{3-03}).  Conversely, if $F\in C^1([0,1])$ is a solution to
(\ref{3-03}), then the boundary conditions in (\ref{3-01}) are
satisfied. Moreover, $\nabla F(x)>0$, $\Delta F (x)$ exists for almost
every $x$ and the differential equation in (\ref{3-01}) holds almost
everywhere.  Furthermore, if $\gamma\in C([0,1])$, then
$F\in C^2([0,1])$ and (\ref{3-01}) holds everywhere.
\end{remark}

\begin{theorem}
\label{t01} 
For each $\gamma\in\ms M_{\rm ac}$, there exists a unique $F\in \mc F$
which solves \eqref{3-03}.
\end{theorem}

The existence is proven by applying Schauder's fixed point
theorem. The argument requires some notation.  For each
$\gamma\in\ms {M}_{\rm ac}$ consider the map
$\color{bblue} \mc K_\gamma : \mc{F} \to C^1\big([0,1]\big)$ defined
by
\begin{equation}
\label{Kg} 
\mc K_\gamma(F) (x) \;:=\;  
\alpha \;+\; (\beta-\alpha)\,
\frac{A \;+\; \int_{0}^x
\exp \{ \int_{0}^y \mc R_\gamma(F;z)\, dz \}\, dy}
{A + \int_{0}^1  \exp \{ \int_{0}^y \mc R_\gamma(F;z)\, dz \}
\, dy + B\, 
\exp \{ \int_{0}^1 \mc R_\gamma(F;y)\, dy\}} \; .
\end{equation}

Let $p$ and $q$ be given by
\begin{equation}
\label{3-15}
p\;:=\; \frac{\alpha\, (\beta - \alpha)}{A\alpha + (B +1) \beta}
\, \frac{1-\beta}{1-\alpha}\;, \quad
q\;:=\; \frac{(1-\alpha)\, (\beta - \alpha)}
{A(1-\alpha) + (B +1) (1-\beta)} \, \frac{\beta}{\alpha}\;.
\end{equation}
Note that $0<p<q$ because
\begin{equation*}
\frac{p}{\beta-\alpha} \;=\;
\frac{1}{A + (B +1) (\beta/\alpha)} \, \frac{1-\beta}{1-\alpha}
\;<\;
\frac{1}{A + (B +1) [(1-\beta)/(1-\alpha)]} \, \frac{\beta}{\alpha}
\;=\; \frac{q}{\beta-\alpha} \;\cdot
\end{equation*}
The inequality above follows from the fact that $(1-\beta)\alpha <
\beta (1-\alpha)$ as $\alpha <\beta$.

Denote by $\mc B_{\rm bc}$ the subset of functions in $C^1([0,1])$
which satisfy the boundary conditions of the Euler-Lagrange equation
\eqref{3-01}:
\begin{equation*} 
\mc B_{\rm bc} \;:=\; \big\{ \, F \in C^1 ([0,1]) :
\nabla F(0) = A^{-1} [F(0)-\alpha] \;, \;\;
\nabla F(1) = B^{-1}  [\beta - F(1)]   \, \big\} \;,
\end{equation*}
and by $\mc B$ the subset of $\mc B_{\rm bc}$ given by
\begin{equation*} 
\mc B \;:=\; \big\{ \, F \in \mc B_{\rm bc} :
p\le \nabla F(x) \le q \;\;\forall\; x\in [0,1] \, \big\} \;.
\end{equation*}
Note that $\mc B_{\rm bc}$, $\mc B$ are closed and convex, and that
$\mc B$ is contained in $\mc F$. To establish this last assertion,
write that
$F(x) \ge F(0) = \alpha + A \nabla F(0) \ge \alpha + Ap > \alpha$
because $\nabla F(x)\ge p>0$. A similar argument shows that
$F(x) \le F(1) = \beta - B \nabla F (1) < \beta - Bp < \beta$. In
particular, for every $F\in \mc B$,
\begin{equation}
\label{3-07}
\alpha \,+\,  Ap \;\le\; F(x) \;\le\; \beta \,-\, Bp \;.
\end{equation}

\begin{lemma}
\label{l01}
Fix $\gamma\in\ms M_{\rm ac}$. Then,
\begin{itemize}
\item [(a)] The functional $\mc K_\gamma$ is a continuous map;
\item [(b)] $\mc K_\gamma(\mc F)\subset\mc B$;
\item [(c)] There exists a finite constant $C_0$, such that
\begin{equation*}
\big|\, \nabla \mc K_\gamma(F) (x) \,-\,
\nabla \mc K_\gamma(F) (y)\,\big|\;\le\;
C_0\, |x-y|
\end{equation*}
for all $F\in \mc B$, $x$, $y\in [0,1]$.
\end{itemize}
\end{lemma}

\begin{proof}
Assertion (a) follows from the definitions of $\mc R_\gamma$ and
$\mc K_\gamma$. We turn to (b). Fix $F\in\mc F$. It is easy to show
that $\mc K_\gamma(F)$ satisfies the boundary condition of
\eqref{3-01}. It remains to derive the bounds on the derivative of
$\mc K_\gamma(F)$. As $0\le \gamma\le 1$, $-F \le \gamma - F \le
1-F$. Therefore, as $\nabla F\ge 0$ and $\alpha \le F\le \beta$,
\begin{equation*}
\frac{-\nabla F}{1-F} \;\le\; \mc R_\gamma \;\le\; \frac{\nabla F}{F}\;\cdot
\end{equation*}
It follows from these inequalities that
\begin{equation*}
\frac{1-\beta}{1-\alpha}\;\le\; \exp\Big\{ \int_0^x \mc R_\gamma
(F;y)\; dy \, \Big\}\;\le\; \frac{\beta}{\alpha}
\end{equation*}
for all $0\le x\le 1$. Reporting these bounds in the definition of
$\mc K_\gamma (F)$ yields that $p\le \mc K_\gamma (F) \le q$, as
claimed.

The proof of the last assertion of the lemma relies on the previous
two bounds and the bound $p\le \nabla F \le q$ which holds for all functions
in $\mc B$.
\end{proof}

\begin{corollary}
\label{l02}
The integro-differential equation \eqref{3-03} has a solution in
$\mc{B}$.
\end{corollary}

\begin{proof}
By Schauder's fixed point theorem, it is enough to show that
$\mc{K}_\gamma (\mc{B})$ has a compact closure in $C^1([0,1])$. By
Ascoli--Arzela theorem, this property holds provided
$\nabla \mc K_\gamma (F)$ is Lipschitz continuous, uniformly for
$F\in\mc B$. This is the content of assertion (c) of the lemma.
\end{proof}

We turn to uniqueness.  The proof relies on an argument used to prove
uniqueness of solutions to equation \eqref{3-01} with Dirichlet
boundary conditions. First, inspired by \cite{DHS2021}, we turn the
pair $(\gamma, F)$ defined on the interval $[0,1]$, of solutions to
\eqref{3-03} [that is, with Robin boundary conditions] into a pair
$(\gamma_{\rm ext}, F_{\rm ext})$ defined on the interval $[-A,1+B]$,
of solutions to \eqref{3-04} with Dirichlet boundary conditions.

Fix $\gamma\in\ms M_{\rm ac}$, and let $F\in \mc F$ be a solution to
equation \eqref{3-03}. We extend $\gamma$ and $F$ to the interval
$[-A, 1+B]$ as follows. $F_{\rm ext}$ coincides with $F$ on $[0,1]$,
is linear in the complement and $F_{\rm ext}(-A) = \alpha$,
$F_{\rm ext}(1+B) = \beta$.  $\gamma_{\rm ext}$ coincides with
$\gamma$ on $[0,1]$ and is equal to $F_{\rm ext}$ in the
complement. Hence, $F_{\rm ext}$,
$\gamma_{\rm ext} : [-A, 1+B] \to \bb R$ are given by
\begin{equation*}
F_{\rm ext} (x) \;=\;
\begin{cases}
\alpha \,+\, \nabla F(0) \, [A + x\,] & \text{for $x\in [-A,0)$}\;, \\
F(x) & \text{for $x\in [0,1]$}\;, \\
\beta \,+\, \nabla F(1) \, [x - B - 1\,] & \text{for $x\in (1,1+B]$}\;, \\
\end{cases}
\end{equation*}
and
\begin{equation*}
\gamma_{\rm ext} (x) \;=\;
\begin{cases}
\gamma (x) & \text{for $x\in [0,1]$}\;, \\
F_{\rm ext}(x) & \text{otherwise}\;. \\
\end{cases}
\end{equation*}

Note that $F_{\rm ext}$ belongs to $C^1([-A, 1+B])$,
$F_{\rm ext}(-A)=\alpha$, $F_{\rm ext}(1+B)=\beta$. Moreover, since
$F = \mc K_\gamma(F)$ and $\mc K_\gamma(\mc F) \subset \mc B$, on the
interval $[0,1]$, $p \le \nabla F_{\rm ext}(x) \le q$. Hence,
$p\wedge \alpha \le \nabla F_{\rm ext}(x) \le q \vee \beta$, and
$F_{\rm ext}$ belongs to the set $\mc B_{\rm ext}$ defined by
\begin{gather*}
\mc B_{\rm bc, ext} \;:=\; \big\{ \, G \in C^1([-A,1+B]) :
G(-A) = \alpha \;,\; G(1+B) = \beta\, \big\} \;, \\
\mc B_{\rm ext} \;:=\; \big\{ \, G \in \mc B_{\rm bc, ext} :
p\wedge \alpha \le \nabla F(x) \le q \vee \beta
\;\;\forall\; x\in [-A,1+B] \, \big\} \;.
\end{gather*}

Fix $\varphi$ in $\ms M_{\rm ac}([-A, 1+B])$.  With Dirichlet boundary
conditions on the interval $[-A, 1+B]$, the problem \eqref{3-04}
becomes the integro-differential equation
\begin{equation}
\label{3-08}
G(x) \;=\; \alpha \;+\; (\beta-\alpha)\,
\frac{\int_{-A}^x
\exp \{ \int_{-A}^y \mc R_\varphi(G;z)\, dz \}\, dy}
{\int_{-A}^{1+B}  \exp \{ \int_{-A}^y \mc R_\varphi (G;z)\, dz \}
\, dy } \;,
\end{equation}
where $\mc R_\varphi(G;z)$ is given by \eqref{3-02}.

\begin{lemma}
\label{l03}
Fix $\gamma\in\ms M_{\rm ac}$, and let $F\in \mc F$ be a solution to
equation \eqref{3-03}.  Then, $F_{\rm ext}$ is a solution to
\eqref{3-08} for $\varphi = \gamma_{\rm ext}$.
\end{lemma}

\begin{proof}
The assertion follows from a straightforward computation. The result
holds for the following reason. On the interval $[0,1]$, the identity
\eqref{3-04} is in force because $F$ is a solution to \eqref{3-03}.
On the other hand, On the complement, \eqref{3-04} holds because both
sides of the identity \eqref{3-04} vanish. The left-hand side because
$F_{\rm ext}$ is linear on $[0,1]^c$, and the right-hand side because
$\gamma_{\rm ext} = F_{\rm ext}$.
\end{proof}

\begin{proof}[Proof of Theorem \ref{t01}]
Fix $\gamma\in\ms M_{\rm ac}$.  Existence has been proven in Corollary
\ref{l02}. To prove uniqueness, consider two solutions $F^{(1)}$,
$F^{(2)}$, and recall the definition of $\gamma^{(j)}_{\rm ext}$,
$F^{(j)}_{\rm ext}$, $j=1$, $2$.

By Lemma \ref{l03}, $F^{(1)}_{\rm ext}$, $F^{(2)}_{\rm ext}$ are
solutions to \eqref{3-08}, with $\varphi = \gamma^{(1)}_{\rm ext}$,
$\gamma^{(2)}_{\rm ext}$, respectively. Therefore, since these functions solve
\eqref{3-04} almost everywhere,
\begin{equation}
\label{3-05b}
(\nabla F^{(j)}_{\rm ext}) (x) \; =\; (\nabla F^{(j)}_{\rm ext}) (-A)
\; +\; \int_{-A}^x 
(\nabla F^{(j)}_{\rm ext}) (y) \;
\mc R_\gamma (F^{(j)}_{\rm ext};y)\; dy
\end{equation}
for $j=1$, $2$ and all $x$ in $[-A,1+B]$.

Assume that
$(\nabla F^{(1)}_{\rm ext}) (-A) = (\nabla F^{(2)}_{\rm ext})
(-A)$. In this case, by definition of $\gamma^{(j)}_{\rm ext}$,
$\gamma^{(1)}_{\rm ext} (x) = \gamma^{(2)}_{\rm ext} (x)$ for all
$x\in [-A,0)$. Since this identity always holds for $x\in [0,1]$,
$\gamma^{(1)}_{\rm ext} (x) = \gamma^{(2)}_{\rm ext} (x)$ for all
$x\in [-A,1]$.  By \eqref{3-05b}, elementary bounds and Gronwall's
inequality,
$(\nabla F^{(1)}_{\rm ext})(x) = (\nabla F^{(2)}_{\rm ext})(x)$ for
all $x\in [-A,1]$ so that
$F^{(1)}_{\rm ext} (x) = F^{(2)}_{\rm ext}(x)$ for all $x$ in this
interval due to the boundary condition satisfied by
$F^{(1)}_{\rm ext}$, $F^{(2)}_{\rm ext}$ at $-A$.

Since
$(\nabla F^{(1)}_{\rm ext}) (1) = (\nabla F^{(2)}_{\rm ext}) (1)$,
$\gamma^{(1)}_{\rm ext} (x) = \gamma^{(2)}_{\rm ext} (x)$ also for
$x\in [1,1+B]$. The same Gronwall's argument permits to extends the
identity $F^{(1)}_{\rm ext} (x) = F^{(2)}_{\rm ext} (x)$ to
$x\in [1, 1+B]$. This concludes the argument in the case
$(\nabla F^{(1)}_{\rm ext}) (-A) = (\nabla F^{(2)}_{\rm ext}) (-A)$.

Assume, by contradiction, that
$(\nabla F^{(1)}_{\rm ext}) (-A) < (\nabla F^{(2)}_{\rm
ext})(-A)$. The argument in this case relies on the following
identity.  As $F^{(1)}_{\rm ext}$ is strictly increasing, \eqref{3-04}
yields that
\begin{equation*}
\nabla\,  \frac {F^{(j)}_{\rm ext}\, [1-F^{(j)}_{\rm ext}]}
{\nabla  F^{(j)}_{\rm ext}}  \; =\; 1 - F^{(j)}_{\rm ext}
- \gamma^{(j)}_{\rm ext}
\end{equation*}
holds almost everywhere. Hence, as $F^{(j)}_{\rm ext}(-A) = \alpha$,
\begin{equation}
\label{3-06}
\frac{ F^{(j)}_{\rm ext}(x) \, [ 1-F^{(j)}_{\rm ext}(x)] }
{(\nabla F^{(j)}_{\rm ext}) (x)}
\;=\; \frac{ \alpha \, [ 1-\alpha] }
{(\nabla F^{(j)}_{\rm ext}) (-A)}
\;+\; \int_{-A}^x [\, 1-F^{(j)}_{\rm ext}(y) -\gamma^{(j)}_{\rm ext}(y) \,]\;dy
\end{equation}
for all $x$ in $[-A,1+B]$.

Let
$x_0 := \inf\{ y\in (-A,1+B] : F^{(j)}_{\rm ext}(y) = F^{(2)}_{\rm
ext}(y)\}$. The point $x_0$ belongs to $(-A,1+B]$ because
$F^{(1)}_{\rm ext}(-A)= F^{(2)}_{\rm ext} (-A)$,
$(\nabla F^{(1)}_{\rm ext}) (-A) < (\nabla F^{(2)}_{\rm ext}) (-A)$
and $F^{(1)}_{\rm ext}(1+B)= F^{(2)}_{\rm ext} (1+B)$.  Actually, it
can not belong to $[-A,0]$ because the functions $F^{(j)}_{\rm ext}$
are linear in this interval.

Suppose that $x_0 \in (0, 1]$. By definition of $x_0$,
$F^{(1)}_{\rm ext}(x) < F^{(2)}_{\rm ext}(x)$ for all $x\in (-A,x_0)$.
Thus, $\gamma^{(1)}_{\rm ext}(x) \le \gamma^{(2)}_{\rm ext}(x)$ for
all $x$ in this interval. On the other hand,
$F^{(1)}_{\rm ext}(x_0)=F^{(2)}_{\rm ext}(x_0)$ and
$(\nabla F^{(1)}_{\rm ext}) (x_0)\geq (\nabla F^{(1)}_{\rm ext}) (x_0)$.
Therefore, by \eqref{3-06},
\begin{equation*}
\frac{ F^{(1)}_{\rm ext}(x_0) [ 1-F^{(1)}_{\rm ext}(x_0)] }
{(\nabla F^{(1)}_{\rm ext}) (x_0)} \;>\;
\frac{ F^{(2)}_{\rm ext} (x_0)
[ 1-F^{(2)}_{\rm ext}(x_0)] } {(\nabla F^{(2)}_{\rm ext}) (x_0)}
\end{equation*}
or, equivalently,
$(\nabla F^{(1)}_{\rm ext})(x_0) < (\nabla F^{(2)}_{\rm ext}) (x_0)$,
which is a contradiction.

We turn to the case where $x_0 \in (1,1+B]$. By definition of $x_0$,
$F^{(1)}_{\rm ext}(x) < F^{(2)}_{\rm ext}(x)$ for all $x\in (-A,x_0)$.
Since the functions $F^{(j)}_{\rm ext}$ are linear in $[1,1+B]$, this
entails that $x_0 =1+B$ and that
$\gamma^{(1)}_{\rm ext}(x) \le \gamma^{(2)}_{\rm ext}(x)$ for all
$x\in [-A, 1+B]$. We may repeat the argument of the previous paragraph
to conclude that
$(\nabla F^{(1)}_{\rm ext}) (1+B) < (\nabla F^{(2)}_{\rm ext}) (1+B)$,
which is a contradiction. This completes the proof of the theorem.
\end{proof}

\begin{proposition}
\label{p01} 
For each $\gamma\in\ms M_{\rm ac}$, denote by $F= F(\gamma)$ the unique
solution in $\mc F$ of \eqref{3-03}. Then,

\begin{itemize}
\item[{(i)}] If $\gamma\in C([0,1])$, then $F(\gamma)\in C^2([0,1])$
and it is the unique solution in $\mc F \cap C^2([0,1])$ of
\eqref{3-03};

\item[{(ii)}] If $\gamma_n$ converges to $\gamma$ in $\ms M_{\rm ac}$ as
$n\to\infty$, then $F_n=F(\gamma_n)$ converges to $F=F(\gamma)$ in
$C^1([0,1])$;
\end{itemize}
\end{proposition}

\begin{proof}
Existence in assertion (i) follows from Theorem \ref{t01} and identity
(\ref{3-03}), which holds for all points $x$ in $[0,1]$ because
$\gamma$ is continuous. Uniqueness follows from Theorem \ref{t01}.
\smallskip

To prove (ii), let $\gamma_n$ be a sequence converging to $\gamma$ in
$\ms M_{\rm ac}$ and denote by $F_n=F(\gamma_n)$ the corresponding
solution to \eqref{3-03}. By Lemma \ref{l01}.(c) and Ascoli--Arzela
theorem, the sequence $F_n$ is relatively compact in
$C^1\big([0,1]\big)$. It remains to show uniqueness of its limit
points.  Consider a subsequence $n_j$ and assume that $F_{n_j}$
converges to $G$ in $C^1\big( [0,1]\big)$. Since $\gamma_{n_j}$
converges to $\gamma$ in $\ms M_{\rm ac}$ and $F_{n_j}$ converges to
$G$ in $C^1\big( [0,1]\big)$, by \eqref{Kg}
$\mc K_{\gamma_{n_j}} (F_{n_j})$ converges to $\mc K_{\gamma} (G)$. In
particular,
$G = \lim_j F_{n_j} = \lim_j \mc K_{\gamma_{n_j}} (F_{n_j}) = \mc
K_{\gamma} (G)$. Hence, by uniqueness of the solutions to
\eqref{3-03}, $G=F(\gamma)$.  This shows that $F(\gamma)$ is the
unique limit point of the sequence $F_n$, and concludes the proof of
(ii).
\end{proof}

Fix a trajectory $u(t,\cdot)$, and denote by $F(t,\cdot)$ the function
given by $F(t,x)= F(u(t,\cdot))(x)$. In the next lemma, we derive
smoothness properties of $F$ in terms of the ones of $u$.  To prove
this result, it is convenient to introduce a new variable. Let
$\color{bblue} \varphi_- := \log[ \alpha / (1-\alpha)]$,
$\color{bblue} \varphi_+ := \log[ \beta/(1-\beta)]$ and denote by
$\widetilde {\mc F}$ be the space of monotone $C^1$ functions given by
\begin{equation}
\label{3-18}
\widetilde{\mc {F}} :=\Big\{\, \varphi \in  C^1([0,1]) \,:\:
\varphi_- \, < \, \varphi (x)\, < \, \varphi_+ \;,\;
\varphi '(x) > 0 \;\; \forall\;  x\in [0,1] \, \Big\} \; .
\end{equation}
Denote by $\Phi : \mc F \to \widetilde{\mc {F}}$ the map given by
\begin{equation}
\label{3-17}
\Phi(F) \;=\; \log \frac {F}{1-F}\;\cdot
\end{equation}
Clearly, $\color{bblue}\Phi^{-1}(\varphi) = e^\varphi/[1+e^\varphi]$.
The advantage of working with $\varphi = \Phi(F)$ instead of $F$ lies
in the fact that, as a function of $\varphi$, the functional
$\ms G_{\rm bulk}$, defined above \eqref{2-03}, is concave. This
property plays a crucial role in the sequel.

In terms of the variable $\varphi$ the Euler-Lagrange equation
\eqref{3-01} becomes
\begin{equation}
\label{3-09}
\left\{
\begin{aligned}
& -\, \nabla\, \Big( \frac{1}{ \nabla \varphi}\Big)
\;+\; \frac{1}{ 1+ e^{\varphi}} \;=\; \gamma
\qquad \text{for}\;\; x\in (0,1) 
\\
& \nabla \varphi (0) \;=\; -\,  \frac{1}{A}\,
\big\{\, (1+ e^{-\varphi (0) })\, \alpha \,-\,
(1-\alpha) (1+e^{\varphi (0)})\,\big\}  \;;  
\\
& \vphantom{\Big(}
\nabla \varphi  (1) \;=\; \frac{1}{B}\,
\big\{\, (1 + e^{-\varphi (1)})\, \beta 
\,-\, (1-\beta) \, (1+e^{\varphi (1)}) 
\,\big\} \;.
\end{aligned}
\right.
\end{equation}

By Lemma \ref{l01}(b), there exists a constant
$C_1=C_1(\alpha, \beta, A, B)\in (0,\infty)$ such that
\begin{equation}
\label{3-10b}
\frac{1}{C_1} \;\le\;  (\nabla \varphi) (x) \;\le\; C_1
\qquad \text{for all} \;  x \in [0,1] \;,\,
\gamma\in\ms M_{\rm ac}\;.
\end{equation}

Fix $T>0$ and a trajectory $u(t,\cdot)$, $0\le t\le T$, in
$C^{1,0}([0,T]\times[0,1])$ such that $0\le u(t,x) \le 1$ for all
$(t,x)$. Denote by $F(t,x)$ the function given by
$\color{bblue} F(t,x)= F(u(t,\cdot))(x)$. By Proposition \ref{p01},
$F$ belongs to $C^{0,2}([0,T]\times[0,1])$.  Next result asserts that
$F\in C^{1,2}([0,T]\times[0,1])$. Let
\begin{equation*}
\varphi (t,x) \;:=\; \Phi(F(t,\cdot)) (x) \;=\;
\log \frac {F(t,x)}{1-F(t,x)}\;\;,
\qquad (t,x)\in [0,T]\times [0,1]\;.
\end{equation*}

As $F$ belongs to $C^{0,2}([0,T]\times[0,1])$ and solves \eqref{3-04},
an elementary computation yields that
$\varphi\in C^{0,2}\left([0,T]\times [0,1] \right)$. Moreover, for
each $t\in [0,T]$, $\varphi(t)$ is the unique strictly increasing
(w.r.t.\ $x$) solution to the problem \eqref{3-09} with
$\gamma = u(t)$. By \eqref{3-10b}, there exists a constant
$C_1=C_1(\alpha, \beta, A, B)\in (0,\infty)$ such that
\begin{equation}
\label{3-10}
\frac{1}{C_1} \;\le\;  (\nabla \varphi) (t,x) \;\le\; C_1
\qquad \forall \, (t,x) \in [0,T]\times[0,1] \;.
\end{equation}

\begin{lemma}
\label{l04}
Fix $u\in C^{1,0} ([0,T]\times [0,1])$ and let $\varphi$ be the
corresponding solution to \eqref{3-09}.  Then
$\varphi \in C^{1,2}\left([0,T]\times [0,1] \right)$ and for each
$0\le t<T$, $\psi := \partial_t \varphi $ is the unique classical
solution to the linear boundary value problem
\begin{equation}
\label{3-11}
\left\{
\begin{aligned}
& \nabla\Big[ 
\frac{ \nabla \psi }{ \big(  \nabla \varphi \big)^2}
\Big]
- \frac{e^{\varphi }}{ \big( 1+ e^{\varphi }\big)^2} \, \psi
= \partial_t u \qquad  x\in (0,1) \\
& \nabla \psi \, = \,  \frac{1}{A} \Big\{ (1-\alpha) \, e^\varphi
\,+\, \alpha\, e^{-\varphi}\,\Big\}\, \psi \qquad x=0 \\
& \nabla \psi \, = \, -\, \frac{1}{B} \Big\{ (1-\beta) \, e^\varphi
\,+\, \beta\, e^{-\varphi}\,\Big\}\, \psi \qquad x=1\;.
\end{aligned}
\right.
\end{equation}
\end{lemma}

\begin{proof}
Fix $t\in [0,T]$. For $h\neq 0$ such that $t+h\in [0,T]$ let
$\psi_h(t,x) := h^{-1} \, [\varphi(t+h,x)-\varphi(t,x)]$. By Proposition
\ref{p01}, $\psi_h(t,\cdot)\in C^{2}\left([0,1]\right)$, and, by 
\eqref{3-09}, for $x \in (0,1)$, $\psi_h$ solves
\begin{equation}
\label{3-12}
\nabla\Big[\,
\frac{ \nabla \psi_h (t)}{ \nabla \varphi (t) \,  \nabla\varphi(t+h)} 
\, \Big]
\,-\, \frac{e^{\varphi(t)}}
{\big( 1+ e^{\varphi(t)}\big) \, \big( 1+ e^{\varphi(t+h)}\big)} 
\: \frac{e^{h\,\psi_h(t)}-1}{h} \;=\; u_h(t)
\end{equation}
where $u_h(t) = h^{-1}\, [u(t+h)-u(t)]$. At the boundary $x=0$,
\begin{equation*}
\nabla \psi_h \, = \, -\,
\frac{1}{A}\, \Big\{\, \alpha\, e^{- \varphi(t)}\,
\frac{e^{- h\,\psi_h(t)}-1}{h}  \;-\;
(1-\alpha) \, e^{\varphi(t)} \, \frac{e^{h\,\psi_h(t)}-1}{h}
\,\Big\} \;,
\end{equation*}
and at the boundary $x=1$,
\begin{equation*}
\nabla \psi_h \, = \, 
\frac{1}{B}\, \Big\{\, \beta\, e^{- \varphi(t)}\,
\frac{e^{- h\,\psi_h(t)}-1}{h}  \;-\;
(1-\beta) \, e^{\varphi(t)} \, \frac{e^{h\,\psi_h(t)}-1}{h}
\,\Big\} \;.
\end{equation*}

\smallskip\noindent{\it Claim 1: The sequence $\psi_h(t)$ is
relatively compact in $C([0,1])$.} \smallskip

To prove this claim, multiply equation \eqref{3-12} by $-\, \psi_h(t)$
and integrate by parts the first term.  Since
$\color{bblue} \Upsilon (a) := a\, (e^a-1)\ge 0$ for all $a\in\bb R$
and in view of the expression for $\nabla \psi_h$ at the boundary, the
boudary terms resulting from the integration by parts are positive.
Therefore,
\begin{equation*}
\begin{aligned}
& \int_0^1 \frac{ \nabla \psi_h (t)^2}
{ \nabla \varphi (t) \,  \nabla\varphi(t+h)} \, dx
\;+\;
\int_0^1 \frac{e^{\varphi(t)}}
{\big( 1+ e^{\varphi(t)}\big) \, \big( 1+ e^{\varphi(t+h)}\big)} 
\, \psi_h (t)\,  \frac{e^{h\,\psi_h(t)}-1}{h} \;dx \\
&\qquad \;\le \; -\, \int_0^1  \psi_h (t)\, u_h(t) \;dx\;.
\end{aligned}
\end{equation*}
As $\Upsilon (a) \ge 0$ for all $a\in\bb R$, by \eqref{3-10}, there
exists a finite constant $C_0$, which depends only on the parameters,
such that
\begin{equation}
\label{3-14}
\int_0^1 \nabla \psi_h (t)^2 \, dx
\;+\; \frac{1}{h^2}\,
\int_0^1 \Upsilon \big(\, h\, \psi_h (t)\,\big) \;dx 
\;\le \; C_0\, \Big|\,  \int_0^1  \psi_h (t)\, u_h(t) \;dx
\,\Big| \;.
\end{equation}

On the right-hand side, adding and subtracting
$\int_0^1 \psi_h (t)\, dx \int_0^1 u_h(t) \, dx$ inside the absolute
value, we estimate this term by
\begin{equation}
\label{3-13}
C_0 \, \Vert u_h(t) \Vert_\infty \, \Big\{\, 
\int_0^1  |\, \nabla \psi_h (t)\,|   \;dx
\;+\; \int_0^1  \big|\, \psi_h (t) \, \big| \;dx\,\Big\}
\end{equation}
for some finite constant $C_0$ which depends only on the parameters
and may change from line to line. We estimate each term separately. By
Young's inequality $2ab \le A a^2 + A^{-1} b^2$, $A>0$, the first one
is bounded by
\begin{equation*}
\frac{1}{2}\, \int_0^1  |\, \nabla \psi_h (t)\,|^2   \;dx \;+\;
C_0 \, \Vert u_h(t) \Vert_\infty^2\;.
\end{equation*}

To bound the second integral in \eqref{3-13}, let
$C_1 = C_0 \, \Vert u_h(t) \Vert_\infty$, and $\delta>0$ be such that
$\Upsilon(a) \ge \delta a^2$ for $|a|\le 1$, and
$\Upsilon(a) \ge \delta\, |a|$ for $|a|\ge 1$. Rewrite the second
integral as
\begin{equation*}
\frac{C_1}{|h|}\, \Big\{ \int_0^1  \big|\, h \psi_h (t) \, \big| \,
\chi_{|\, h\, \psi_h (t) \,| \le 1} \;dx \;+\;
\int_0^1  \big|\, h \psi_h (t) \, \big| \,
\chi_{|\, h\, \psi_h (t) \,| \ge 1} \;dx \, \Big\}\;,
\end{equation*}
where $\color{bblue} \chi_{\ms A}$ stands for the indicator of the set
$\ms A$. By Young's inequality and the definition of $\delta$, the
previous expression is bounded by
\begin{equation*}
\frac{C_1}{|h|}\, \Big\{ \frac{A}{2}\;+\;
\frac{1}{2A} \int_0^1  \big|\, h \psi_h (t) \,
\big|^2  \, \chi_{|\, h\, \psi_h (t) \,| \le 1} \;dx \;+\; \frac{1}{\delta}
\int_0^1  \Upsilon (h \psi_h (t))   \;dx \, \Big\}
\end{equation*}
for all $A>0$. By definition of $\delta$ and choosing
$A= |h| C_1 /\delta$, the previous expression is less than or equal to
\begin{equation*}
\frac{C_1^2}{2\delta} \;+\; \frac{1}{h^2}
\Big\{ \frac{1}{2} \,+\, \frac{C_1 \, |h|}{\delta}\,\Big\}
\int_0^1  \Upsilon (h \psi_h (t))   \;dx \;.
\end{equation*}
Therefore, \eqref{3-13} is bounded above by
\begin{equation*}
\frac{1}{2}\, \int_0^1  |\, \nabla \psi_h (t)\,|^2   \;dx \;+\;
C_0 \, \Vert u_h(t) \Vert_\infty^2 \;+\;
\frac{C_1^2}{2\delta} \;+\; \frac{1}{h^2}
\Big\{ \frac{1}{2} \,+\, \frac{C_1 \, |h|}{\delta}\,\Big\}
\int_0^1  \Upsilon (h \psi_h (t))   \,dx \,,
\end{equation*}
where $C_1 = C_0 \, \Vert u_h(t) \Vert_\infty$.

Reporting this estimate in \eqref{3-14} yields that
\begin{equation}
\label{3-16}
\frac{1}{2} \int_0^1 \nabla \psi_h (t)^2 \, dx
\;+\; \frac{1}{4 h^2}\,
\int_0^1 \Upsilon \big(\, h\, \psi_h (t)\,\big) \;dx 
\;\le \; C_0 \,  \Big( \, 1 \,+\, \frac{1}{2\delta}\, \Big)
\, \Vert u_h(t) \Vert_\infty^2
\end{equation}
for $|h|\le \delta/4C_1$.

This shows that the sequence $\psi_h(t)$ is uniformly Lispchitz
continuous, and thus relatively compact in $C([0,1])$, proving the
assertion. \smallskip

\smallskip\noindent{\it Claim 2: The sequence $\psi_h(t)$ converges in
$C([0,1])$ to the unique classical solution to \eqref{3-11}.}
\smallskip

Recall from Appendix \ref{sec05} the definition of the Sobolev space
$\mc H^1([0,1])$ and of the associated norm.  Fix a subsequence
$(\psi_{h(k)} : k\ge 1)$, still denoted by $\psi_{h}$, which converges
to a limit, represented by $\psi$. By \eqref{3-16}, $\psi$ belongs to
$\mc H^1([0,1])$ and $\nabla \psi_h (t)$ converges weakly in
$\ms L^2([0,1])$ to $\nabla \psi$.

Fix a function $v$ in $\mc H^1([0,1])$. Multiply both sides of
\eqref{3-12} by $v$ and integrate by parts to get that
\begin{equation*}
\begin{aligned}
& \mf a_h (1)\, v(1)\, \nabla \psi_h (t, 1) \;-\;
\mf a_h (0)\, v(0)\, \nabla \psi_h (t, 0) \\
& \quad - \, \int_0^1 \mf a_h\, \nabla v\, \nabla \psi_h (t)\, dx
\, + \, \int_0^1 \mf b_h\, v\, \frac{e^{h\,\psi_h(t)}-1}{h} \, dx
\;=\; \int_0^1 u_h(t) \, v\, dx\;,
\end{aligned}
\end{equation*}
where
$\mf a_h = [\, \nabla \varphi (t) \, \nabla\varphi(t+h)\,]^{-1}$,
$\mf b_h \,=\, -\, e^{\varphi(t)}/ (\,1+ e^{\varphi(t)}\,) \, (\, 1+
e^{\varphi(t+h)}\,)$. Replace in this equation $\nabla \psi_h (t, 0)$,
$\nabla \psi_h (t, 1)$ by the expressions appearing in the equations
below \eqref{3-12}. As $\mf a_h$, $\mf b_h$, $\psi_h$ converge in
$C([0,1])$, and since $\nabla \psi_h$ converges weakly to
$\nabla \psi$ in $\ms L^2([0,1])$, passing to the limit in the
previous equation yields that
\begin{equation*}
\begin{aligned}
& \mf a (1)\, v(1)\, \mf c_1\, \psi (t, 1) \;-\;
\mf a (0)\, v(0)\, \mf c_0\, \psi (t, 0)  \\
& \quad - \, \int_0^1 \mf a\, \nabla v\, \nabla \psi (t)\, dx
\, + \, \int_0^1 \mf b\, v\, \psi \, dx
\;=\; \int_0^1 (\partial_t u)(t) \, v\, dx\;,
\end{aligned}
\end{equation*}
where
$\mf c_1 \,=\, - \, B^{-1} \{ (1-\beta) \, e^{\varphi(t,1)} \,+\,
\beta\, e^{-\varphi(t,1)}\,\}$,
$\mf c_0 = A^{-1} \{ (1-\alpha) \, e^{\varphi(t,0)} \,+\, \alpha\,
e^{-\varphi (t,0)}\,\}$. Hence, according to \cite[IV, Section
1]{M83}, $\psi$ is a generalized solution to \eqref{3-11}. By
\cite[Theorem IV.1.2]{M83}, the generalized solution is unique, which
proves that $\psi_h$ converges in $C([0,1])$ to the unique generalized
solution to \eqref{3-11}. As $\partial_t u(t,\cdot) \in C([0,1])$, by
\cite[Theorem IV.2.1]{M83}, the generalized solution belongs to
$C^2([0,1])$ and is a classical solution to \eqref{3-11}. This proves
the claim. \smallskip

It remains to prove the continuity $t\mapsto \psi(t,\cdot)$. According
to \cite[Theorem IV.1.2]{M83}, there exists a constant $C_0$,
independent of $\partial_t u$, such that
$\Vert \psi (t) \Vert_{\mc H^1 ([0,1])} \le C_0\, \Vert (\partial_t
u)(t) \Vert_{ \ms L^2 ([0,1])}$. Since there exist a finite constant $C_0$
such that
$\Vert v \Vert_\infty \le C_0 \Vert v \Vert_{\mc H^1 ([0,1])}$ for all
$v \in \mc H^1 ([0,1])$,
\begin{equation*}
\Vert \, \psi (t+h)  - \psi (t) \,  \Vert_\infty \;\le\;
C_0 \, \big\Vert\,  (\partial_t u)(t+h) \,-\,
(\partial_t u)(t)  \, \big\Vert_{\ms L^2 ([0,1])}\;.
\end{equation*}
This proves that $\psi$ belongs to $C^{0,2}([0,T] \times [0, 1])$, and
therefore that $\varphi$ belongs to $C^{1,2}([0,T] \times [0, 1])$, as
claimed. 
\end{proof}

\begin{proof}[Proof of Theorem \ref{t02}]
For each $F\in \mc F$, $\ms G (\cdot, F)$ is a convex, lower
semi-continuous functional on $\ms M_{\rm ac}$.  The functional
$S_0(\cdot)$ inherits these properties. By choosing $F=\bar\rho$ in
\eqref{2-03} we obtain that for every $\gamma\in \ms M_{\rm ac}$,  
\begin{equation*}
\ms G (\gamma, \bar\rho) \;=\; S_{\rm eq} (\gamma) \;-\;
(1+A+B) \, \log (1+A+B)\;,
\end{equation*}
where $S_{\rm eq}\colon \ms M_{\rm ac} \to \bb R$ is the convex and
nonnegative functional 
\begin{equation*}
S_{\rm eq}(\gamma) \;=\; \int_{0}^1  \Big\{\, \gamma (x) \, \log
\frac {\gamma (x)}{\bar\rho (x)} + \big[1 - \gamma (x)\big]
\log \frac {1- \gamma (x)}{1- \bar\rho(x)} \,
\Big\} \; dx\;.
\end{equation*}
As $S_{\rm eq}$ is non-negative,
$S_0(\gamma) \, \ge \, -\, (1+A+B) \, \log (1+A+B)$. On the other hand,
as $a\mapsto \log a$ is concave, by Jensen's inequality and since
$\alpha \le F(x) \le \beta$, for every $F\in\mc F$,
$\gamma \in \ms M_{\rm ac}$,
\begin{equation*}
\ms G_{\rm bulk} (\gamma, F) \;\le\; \log \frac{1}{\alpha} \;+\;
\log \frac{1}{1-\beta} \;\cdot
\end{equation*}
Hence, there exists a finite constant $C_0 = C_0(\alpha, \beta, A, B)$
such that $S_0(\gamma) \, \le\, C_0$ for all
$\gamma \in \ms M_{\rm ac}$. This proves the first assertion of the
theorem. We turn to the second.

Recall from \eqref{3-17} the definition of
$\varphi =\Phi (F) \in \widetilde{\mc F}$, and that
$F= \Phi^{-1} (\varphi) = e^{\varphi }/[1+e^{\varphi}]$. Set
$\widetilde{\ms G}_{\rm bulk} (\gamma,\varphi) = \ms G_{\rm bulk}
(\gamma,\Phi^{-1}(\varphi))$ so that
\begin{equation}
\label{3-19}
\begin{aligned}
\widetilde{ \ms G}_{\rm bulk}  (\gamma ,\varphi) \;=\;
\int_{0}^1 \Big\{ \mf h (\gamma) \,+\, (1-\gamma) \, \varphi 
\,-\, \log \big[\, 1+e^{\varphi }  \,\big] \, \, 
+\; \log\frac{\varphi'}{\beta-\alpha} \Big\} \;dx\;, 
\end{aligned}
\end{equation}
where $\mf h (a) = a \log a \,+\, (1-a) \log (1-a)$.

To prove the second assertion of the theorem, we have to show that for
each $\gamma\in\ms M_{\rm ac}$, the supremum over the set
$\widetilde{\mc F}$ of
\begin{equation*}
\widetilde{ \ms G}  (\gamma ,\varphi) \;:=\;
\widetilde{ \ms G}_{\rm bulk}  (\gamma ,\varphi) 
\;+\; A\,
\ln \frac{F(0) - \alpha}{A(\beta-\alpha)}
\;+\; B\, \ln \frac{\beta - F(1)}{B (\beta-\alpha)}
\end{equation*}
is uniquely attained at
$\varphi= \Phi (F (\gamma))$, where, recall, $F(\gamma)$ represents
the unique solution to \eqref{3-03}. In the previous equation, $F(x)$
stands for $\exp\{\varphi(x)\}/ [1+\exp\{\varphi(x)\}]$, $x=0$, $1$.

Since the functions $a\mapsto \log a$, $a \mapsto -\log (1+e^{a})$,
$a \mapsto -\log \{\, [e^{a}/(1+ e^a)] -\alpha\}$,
$a \mapsto -\log \{\, \beta - [e^{a}/(1+ e^a)] \, \}$ are strictly
concave, the last two in the interval $(\varphi_-, \varphi_+)$ defined
above \eqref{3-18}, for each $\gamma\in\ms {M}_{\rm ac}$, the
functional ${\widetilde{\ms G }}(\gamma,\cdot)$ is strictly concave on
${\widetilde{\mc{F}}}$.  Moreover it is easy to show that
${\widetilde{\ms G }}(\gamma,\cdot)$ is Gateaux differentiable on
${\widetilde{\mc{F}}}$ with derivative given by
\begin{equation*}
\begin{aligned}
& \Big\langle  \frac{\delta {\widetilde{\ms G }}(\gamma,\varphi)}
{\delta \varphi} \,,\, g \Big\rangle
\; =\; \int_{0}^1 \Big\{ \, \frac{g'}{\varphi'} \,+\,
\Big[\, \frac{1}{1+e^{\varphi}} \,-\, \gamma \, \Big]
\, g\, \Big\}\; dx \\
&\qquad  +\; \frac{A \, g(0)}{(1-\alpha) (1+e^{\varphi(0)}) -
\alpha (1+e^{-\varphi(0)})} \;+\;
\frac{B \, g(1)}{(1-\beta) (1+e^{\varphi(1)}) -
\beta (1+e^{-\varphi(1)})} 
\end{aligned}
\end{equation*}
for all $g$ in $C^1([0,1])$.  By \eqref{3-09}, the right-hand side
vanishes for $\varphi = \Phi(F(\gamma))$.

By \cite[Proposition 1.5.4]{ET76} and since
${\widetilde{\ms G}} (\gamma, \,\cdot\,)$ is strictly concave, for any
$\psi \neq \varphi$ in ${\widetilde{\mc{F}}}$,
\begin{equation*}
{\widetilde{\ms G}}(\gamma,\psi) \;<\;
{\widetilde{\ms G}}(\gamma,\varphi)  \;+\;
\Big\langle \, \frac{\delta {\widetilde{\ms G}}(\gamma,\varphi)}
{\delta \varphi} \,,\, \psi-\varphi \,\Big\rangle \;.
\end{equation*}
Since
$\delta {\widetilde{\ms G}} (\gamma,\varphi) / {\delta \varphi} =0$
for $\varphi = \Phi(F(\gamma))$, the supremum on ${\widetilde{\mc{F}}}$
of ${\widetilde{\ms G}}(\gamma,\cdot)$ is uniquely attained when
$\varphi= \Phi (F (\gamma))$.
\end{proof}

\begin{remark}
\label{rm01}
Fix $\gamma\in\ms M_{\rm ac}$, and consider a sequence
$\gamma_n\in \ms M_{\rm ac}$ such that
\begin{itemize}
\item[(i)] For each $n\ge 1$, there exists $\delta_n>0$ such that
$0<\delta_n \le \gamma_n(x) \le 1-\delta_n$ for all 
$x\in [0,1]$;
\item[(ii)] $\gamma_n$ converges to $\gamma$ a.e.
\end{itemize}
Then, by the dominated convergence theorem and Proposition
\ref{p01}.(ii),
\begin{equation*}
\lim_{n\to\infty} S_0(\gamma_n) \;=\;
\lim_{n\to\infty} \ms G \big(\gamma_n,F(\gamma_n)\big) \;=\;
\ms G\big(\gamma,F(\gamma)\big) \;=\; S_0(\gamma)\;.
\end{equation*}

\end{remark}

\section{The quasi-potential}
\label{sec03}

Let $\delta_0>0$ be such that
$\color{bblue} \delta_0 \le \alpha < \beta \le 1- \delta_0$.  For
$\delta\in (0,\delta_0]$ and $T>0$, let
\begin{equation}
\label{4-01}
\begin{gathered}
\ms {M}_\delta \;:=\; \big\{ \gamma \in C^2 ([0,1])\, :\,
\delta \le \gamma(x) \le 1-\delta\, \big\}
\\
D_{T,\delta} \;:=\; \big\{\,  u\in C^{1,2} ( [0,T]\times[0,1])
\,:\: \delta \le u(t,x) \le 1-\delta \, \big \}\;.
\end{gathered}
\end{equation}
Unless otherwise stated, throughout this section, $T>0$ and
$0<\delta\le \delta_0$ are fixed.

\begin{lemma}
\label{l05} 
Fix $u$ in $D_{T,\delta}$ and denote by $F(t,x) = F(u(t,\cdot))\,(x)$
the solution to the boundary value problem \eqref{3-01} with $\gamma$
replaced by $u(t)$. Set
\begin{equation}
\label{4-02} 
\Gamma (t,x) \;=\; \log \frac{u(t,x)}{1- u(t,x)} \; -\; 
\log \frac{F(t,x)}{1-F(t,x)} \; \cdot
\end{equation}
Then, for each $T\geq0$,
\begin{equation}
\label{4-03} 
S_0\big(u(T)\big) - S_0\big(u(0)\big)  \;=\;
\int_0^T \langle \, \partial_t u(t) \,,\,
\Gamma(t) \, \rangle \; dt  \; .
\end{equation}
\end{lemma}

\begin{proof}
Recall that $F(t,\cdot)$ is strictly increasing for any $t\in[0,T]$.
By Lemma \ref{l04}, $F$ belongs to
$C^{1,2}\big([0,T]\times[0,1]\big)$.  By Theorem \ref{t02} and the
dominated convergence theorem,
\begin{equation*}
\begin{aligned}
& \frac {d}{dt} S_0(\, u(t)\,) \;=\;
\frac {d}{dt} \: \ms G (\, u(t), F(t) \,) \\
& \quad
\;=\; \big\langle \, \partial_t u(t) \,,\, \Gamma(t) \,\big\rangle
\;+\; \big\langle \, \partial_t F(t),
\frac{F(t)-u(t)}{F(t)[1-F(t)]} \, \big\rangle 
\;+\; \big\langle \, \frac{1}{\nabla F(t)}
\,,\,\partial_t \nabla F(t) \, \big \rangle \\
& \quad
\;+\; A\, \frac{(\partial_t F)(t,0)}{F(t,0) - \alpha}
\;+\; B\, \frac{(\partial_t F)(t,1)}{F(t,1) - \beta} \;\cdot
\end{aligned}
\end{equation*}
As $F$ belongs to $\mc B$ it satisfies mixed boundary conditions at
$x=0$, $x=1$. An integration by parts yields that the previous
expression is equal to
\begin{equation*}
\big\langle \, \partial_t u(t) \,,\, \Gamma(t) \,\big\rangle
\;+\; \big\langle \, \partial_t F(t),
\frac{F(t)-u(t)}{F(t)[1-F(t)]}
\,+\, \frac{\Delta F(t)}{\big(\nabla F(t)\big)^2}
\, \big\rangle \;.
\end{equation*}
To conclude the proof, it remains to recall Remark \ref{rm02}, which
asserts that $F$ solves \eqref{3-01} almost everywhere.
\end{proof}

Recall the definition of the Hamiltonian $\ms H$, given in
\eqref{1-09}, and the one of $\ms M_\delta$, introduced at the
beginning of this section.

\begin{lemma}
\label{l06} 
Fix $\gamma\in \ms {M}_\delta$, and let $F = F(\gamma)$ be the solution
of the boundary value problem \eqref{3-01}. Set
\begin{equation*}
\Gamma (x)  \;=\;  \log \frac{\gamma(x)}{1-\gamma(x)}  \;-\;
\log \frac{F(x)}{1-F(x)}\; \cdot
\end{equation*}
Then,
\begin{equation*}
\ms H(\,\gamma\,,\, \Gamma\,) \; =\;  0 \;.
\end{equation*}
\end{lemma}

\begin{proof}
By Corollary \ref{l02} and Proposition \ref{p01},
$F\in \ms {M}_{\delta}$. By definition of $\Gamma$ and since $F$
belongs to $\mc B$,
\begin{equation}
\label{4-04}
\begin{aligned}
& \mf b_{\alpha, A} \big(\, \gamma(0)\, ,\, \Gamma (0)\, \big)
\;+\; \mf b_{\beta, B} \big(\, \gamma(1)\, ,\, \Gamma (1)\, \big) \\
&\quad =\; [\, F(0) - \gamma (0)\,]\, 
\frac{(\nabla F)(0)}{F(0) \,[1-F(0)]}
\; -\; [\, F(1) - \gamma (1)\,] \,
\frac{(\nabla F)(1)}{F(1) \,[1-F(1)]} \;\cdot
\end{aligned}
\end{equation}
On the other hand, a straightforward computation yields that
\begin{equation*}
\begin{aligned}
& \big\langle \gamma (1-\gamma) \,,\, \big(\nabla \Gamma\big)^2  \big\rangle
\; -\; \big\langle \nabla \gamma \, ,\, \nabla \Gamma  \big\rangle \\
&\quad
=\; \Big\langle \gamma (1-\gamma) \,,\, \Big( 
\frac{\nabla F}{F(1-F)}
\Big)^2 \Big\rangle \; - \;    
\Big\langle \nabla \gamma \,,\, \frac{\nabla F}{F(1-F)} \Big\rangle \;.
\end{aligned}
\end{equation*}
Rewrite the second term as
\begin{equation*}
-\;  \Big\langle \nabla (\gamma-F) \,,\, \frac{\nabla
F}{F(1-F)}\Big\rangle
\;-\;  \Big\langle \nabla F \,,\, \frac{\nabla F}{F(1-F)} \Big\rangle 
\;,
\end{equation*}
and integrate by parts the first expression. The boundary terms cancel
with the ones appearing in \eqref{4-04}.

Up to this point, we proved that 
\begin{equation*}
\ms H(\,\gamma\,,\, \Gamma\,) \;=\; 
\Big\langle \, \gamma (1-\gamma) \,,\,
\Big( \frac{\nabla F}{F(1-F)} \Big)^2 \, \Big\rangle \;+\;
\Big\langle \gamma-F \,,\, \nabla \frac{\nabla
F}{F(1-F)}\Big\rangle
\;-\;  \Big\langle \nabla F \,,\, \frac{\nabla F}{F(1-F)} \Big\rangle \;.
\end{equation*}
Since $\gamma (1-\gamma) - F(1-F) = (\gamma - F)(1-\gamma - F)$, we
may rewrite this sum as
\begin{equation*}
\Big\langle \, \gamma -F  \,,\, (1-\gamma - F)
\Big( \frac{\nabla F}{F(1-F)} \Big)^2 \, \Big\rangle \;+\;
\Big\langle \gamma-F \,,\, \nabla \frac{\nabla
F}{F(1-F)}\Big\rangle \;.
\end{equation*}
On the other hand, as $\nabla \{\nabla F/ F(1-F)\} = \Delta F/ F(1-F)
- (1-2F) [\nabla F / F(1-F)]^2$, the previous expression is equal to
\begin{equation*}
\Big\langle \, \gamma -F  \,,\, (F-\gamma)
\Big( \frac{\nabla F}{F(1-F)} \Big)^2 \, \Big\rangle \;+\;
\Big\langle \gamma-F \,,\, \frac{\Delta F}{F(1-F)}\Big\rangle \;.
\end{equation*}
This sum vanishes because $F$ is the solution to \eqref{3-01}.
\end{proof}

\begin{remark}
\label{rm3}
Lemma \ref{l05} identifies $\Gamma$ as the functional derivative of
$S$,
\begin{equation*}
\Gamma \;=\; \frac{\delta S_0}{\delta \gamma}
\;=\; \frac{\delta S}{\delta \gamma}\;,
\end{equation*}
and Lemma \ref{l06} states that this derivative
$\Gamma = \delta S/\delta \gamma$ satisfies the Hamilton--Jacobi
equation.
\end{remark}

\subsection{Lower bound for the quasi-potential}
In this subsection, we prove that $V \ge S$.

\begin{lemma}
\label{l07} 
For each $\gamma\in \ms M_{\rm ac}$, $V(\gamma)\ge S(\gamma)$.
\end{lemma}

\begin{proof}
In view of the variational definition of $V$, we have to show that
$S(\gamma) \le I_{[0,T]}(u |\bar\rho)$ for any $T>0$ and any path
$u\in D\big([0,T];\ms {M}\big)$ which connects the stationary profile
$\bar\rho$ to $\gamma$ in the time interval $[0,T]$: $u(0)=\bar\rho$,
$u(T)=\gamma$.
  
Fix such a path $u$ and assume first that $u$ belongs to
$D_{T,\delta}$ for some $\delta>0$.  For $0\le t\le T$, let
$F(t) = F(u(t))$ be the solution to the elliptic problem \eqref{3-01}
with $u(t)$ in place of $\gamma$.  In view of the variational
definition of $I_{[0,T]}(u |\bar\rho)$ given in \eqref{1-02}, to prove
that $S(\gamma) \le I_{[0,T]}(u |\bar\rho)$ it is enough to exhibit
some function $H \in C^{1,2} ([0,T] \times [0,1])$ for which
$S(\gamma) \le J_{T,H} (u)$.  We claim that $\Gamma$ given in
\eqref{4-02} fulfills these conditions.

Note that $\Gamma$ belongs to $C^{1,2} ([0,T] \times [0,1])$ because,
on the one hand, $u$ belongs to this set as it is assumed to be
in $D_{T,\delta}$. On the other hand, by Lemma \ref{l04},
$F\in C^{1,2} ([0,T] \times [0,1])$.

Recall the definition of the Hamiltonian $\ms H$ introduced in
\eqref{1-09}.  By \eqref{1-01b}, integrating by parts in time yields
that
\begin{equation*}
J_{T,\Gamma} \big (\, u \, \big) \; =\;
\int_0^T \big\{\, \big\langle \, \partial_t u (t) \,,\,
\Gamma(t) \, \big\rangle \,-\,
\ms H (\, u(t) \,,\, \Gamma(t) \, \big) \, \big\} 
\;dt  \;.
\end{equation*}
By Lemmata \ref{l05} and \ref{l06},
$J_{T,\Gamma} (u) = S_0(u(T)) \,-\, S_0(u(0)) = S_0(\gamma) -
S_0(\bar\rho) = S(\gamma)$.

Up to this point we have shown that $S(\gamma) \le I_{[0,T]}(u |\bar\rho)$
for smooth paths $u$ bounded away from $0$ and $1$. We extend this
result to arbitrary  paths

Fix a path $u$ with finite rate function:
$I_{[0,T]}(u|\bar\rho) < \infty$.  Since
$\alpha \le \bar\rho \le \beta$, by Theorem \ref{mt3b}, there exists a
sequence $\{u^n ,\, n\ge 1\}$, $u^n \in D_{T,\delta_n}$ for some
$\delta_n>0$, such that $u_n$ converges to $u$ and
$I_{[0,T]}(u_n|\bar\rho)$ converges to $I_{[0,T]}(u|\bar\rho)$.
Therefore, by the result on smooth paths and the lower semi continuity
of $S$,
\begin{equation*}
I_{[0,T]}(u|\bar\rho) \; =\; \lim_{n\to\infty} I_{[0,T]}(u_n|\bar\rho) \ge 
\liminf_{n\to\infty} S\big(u_n(T) \big) \ge S(u(T)) \;,
\end{equation*}
which concludes the proof of the lemma. 
\end{proof}

\subsection{Upper bound, the adjoint hydrodynamic equation}

The following lemma explains which is the right candidate for the 
optimal path for the variational problem \eqref{1-04}. 
For $0< \varrho <1$, $D>0$, $0<a<1$, $M\in \bb R$, let 
\begin{equation}
\label{6-04}
\mf q_{\varrho, D} (a, M) \;=\; \frac{1}{D}\, \Big\{ [1-a] \, \varrho \,
\big[\, e^{M} \,-\, M \, -\, 1\,\big] \;+\;
a \, [1-\varrho] \,  \big[\, e^{-M} \,+\, M \, -\, 1\,\big]\, \Big\} \;.
\end{equation}
Note that $\mf q_{\varrho, D} (a, \cdot\, )$ is a nonnegative, convex
function which vanishes at the origin. Recall the definition of
$\mf p_{\varrho, D} (a,M)$, introduced in \eqref{5-02}.

\begin{lemma}
\label{l08} 
Fix a profile $\psi \in \ms M_\delta$, and a path $u \in D_{T,\delta}$
with finite rate function, $I_{[0,T]}(u|\psi)<\infty$.  For
$0\le t\le T$, denote by $F(t)= F(u(t))$ the unique solution to the
boundary value problem \eqref{3-01} with $\phi$ replaced by
$u(t)$. Then, there exists a function $K\in\mc H^1 (\Omega_T)$
such that $u$ is the weak solution to
\begin{equation}
\label{4-06} 
\left\{
\begin{aligned}
&\partial_t u \;=\; - \,  \Delta u
\;+\; 2\,  \nabla \Big( \sigma(u)   \,
\nabla  \big[ \log \frac{ F}{1-F} + K \big] \Big) \;,  
\quad  (t,x)\in [0,T]\times (0,1) \\
& \nabla u_t (1) \,-\, 2\, \sigma(u_t(1)) \, \nabla G_t(1) \,=\,
\mf p_{\beta, B} \big(\, u_t(1)\,,\, G_t(1)\, \big) \;, \\
& \nabla u_t (0) \,-\, 2\, \sigma(u_t(0)) \, \nabla G_t(0) \,=\,
-\, \mf p_{\alpha, A} \big(\, u_t(0)\,,\, G_t(0)\, \big) \;, \\
& u(0,x) \;=\; \psi (x) \;, \quad x\in [0,1]\;.
\end{aligned}
\right.
\end{equation}
where $G_t = \Gamma_t - K_t$ and $\Gamma_t$ is given by \eqref{4-02}.
Moreover, 
\begin{equation}
\label{4-07} 
\begin{aligned}
& I_{[0,T]}(u|\psi) \;=\;  S_0(u(T)) \,-\,  S_0(\psi)  \,+\,
\int_0^T \big\langle\, \sigma(u(t))\,,\, [ \nabla K(t)]^2 \,
\big\rangle \\
& \; +\; \int_0^T e^{G_t(1)}\,
\mf q_{\beta, B} (u_t(1), K_t(1)) \; dt
\;+\; \int_0^T e^{G_t(0)}\, \mf q_{\alpha, A} (u_t(0), K_t(0)) \; dt \;.
\end{aligned}
\end{equation}
\end{lemma}

\noindent{\bf Strategy of the proof of the upper bound:} We present
below the main steps of the proof in light of Lemma \ref{l08}. Fix a
density profile $\gamma \in \ms M_{\rm ac}$ and denote by
$(v^{(\gamma)}, F^{(\gamma)})$ the solution to the time-reversed
equation \eqref{4-06} with $K=0$ and starting from $\gamma$:
\begin{align}
\label{4-08}
&
\left\{
\begin{aligned}
& \partial_t v \,=\,  \Delta v  \,-\, 2\,  \nabla \big( \, \sigma (v) 
\, \nabla \log \frac{ F}{1-F} \, \big)   \quad (t,x) \in
(0,\infty)\times (0,1) \;, \\
& \nabla v_t (1) \,-\, 2\, \sigma(v_t(1)) \, \nabla H_t(1) \,=\,
\mf p_{\beta, B} \big(\, v_t(1)\,,\, H_t(1)\, \big) \;, \\
& \nabla v_t (0) \,-\, 2\, \sigma(v_t(0)) \, \nabla H_t(0) \,=\,
-\, \mf p_{\alpha, A} \big(\, v_t(0)\,,\, H_t(0)\, \big) \;, \\
& v(0,\cdot) = \gamma (\,\cdot\,) \;, \quad  x\in [0,1]\;, 
\end{aligned}
\right.
\\
\label{4-08b}
&
\left\{
\begin{aligned}
& \Delta F_t = 
\big( v_t - F_t \big) \frac{\big( \nabla F_t \big)^2}{F_t(1-F_t)}
\quad (t,x) \in (0,\infty)\times (0,1) \; , \\
& 
\nabla F_t(0) = A^{-1} [F_t(0)-\alpha] \;, \quad
\nabla F_t(1) = B^{-1} [\beta - F_t(1)] \; . 
\end{aligned}
\right.
\end{align}
In equation \eqref{4-08},
$H_t = \log [v_t/(1-v_t)] - \log [F_t/(1-F_t)]$ and
$\mf p_{\varrho, D}$ has been introduced in \eqref{5-02}.  Equation
\eqref{4-08} will be shown to be equivalent to \eqref{4-11c}.
Proposition \ref{l10} provides a precise meaning to the coupled
equation \eqref{4-08}--\eqref{4-08b}.

The second step consists in proving that the solution $v^{(\gamma)}_t$
of equation \eqref{4-08} converges to $\bar\rho$ as $t\to
\infty$. This is the content of Lemma \ref{l12}, where we prove that
this convergence takes place in $\ms L^\infty$.

Fix $T_1$ large enough for $v^{(\gamma)}(T_1)$ to be close to
$\bar\rho$ in $\ms L^\infty$. Reverse in time the path $v^{(\gamma)}$
by setting $w^{(1)}(t) = v^{(\gamma)}(T_1-t)$, $0\le t\le T_1$. The
path $w^{(1)}$ satisfies equation \eqref{4-06} with $K=0$ and
$\psi = v^{(\gamma)} (T_1)$. Therefore, by Lemma \ref{l08},
\begin{equation*}
I_{[0,T_1]}(\, w^{(1)} \,|\, v^{(\gamma)} (T_1)\,)
\;=\;  S_0(\gamma) \,-\,  S_0(v^{(\gamma)} (T_1)) \;.
\end{equation*}

It remains to replace in the previous formula $v^{(\gamma)} (T_1)$ by
$\bar\rho$, keeping in mind that $v^{(\gamma)} (T_1)$ is close to
$\bar\rho$ in the $\ms L^\infty$ norm. This is done in Lemma
\ref{l13}, where we show that if $\phi$ is close to $\bar\rho$ in
$\ms L^\infty$, then there exists a path $w^{(2)}_t$, $0\le t\le 1$,
which connects $\bar\rho$ to $\phi$ and such that
$I_{[0,1]}(\, w^{(2)} \,|\, \bar\rho\,) \le C_0 \,\Vert \phi -
\bar\rho\Vert^2_2$.

Define the path $w(t)$, $0\le t\le T_1+1$, by $w(t) = w^{(2)}(t)$ for
$0\le t\le 1$, $w(t) = w^{(1)}(t-1)$, $1\le t\le T_1+1$. By definition
$w(0) = \bar\rho$ and $w(T_1+1) = \gamma$. Moreover, by the previous
bounds of the rate functional, and since $w(1) = v^{(\gamma)} (T_1)$,
\begin{equation*}
\begin{aligned}
I_{[0,T_1+1]}(\, w \,|\, \bar\rho\,) \; & =\;
I_{[0,1]}(\, w^{(2)} \,|\, \bar\rho\,) \;+\;
I_{[1,T_1+1]}(\, w^{(1)} (1+\,\cdot\,) \,|\, v^{(\gamma)} (T_1) \,) \\
\; & \le\; S_0(\gamma) \,-\,  S_0(v^{(\gamma)} (T_1)) \;+\;
C_0\, \,\Vert  v^{(\gamma)} (T_1) - \bar\rho\Vert^2_2\;.
\end{aligned}
\end{equation*}
The first identity says that the cost of a path in the time-interval
$[0,T_1+1]$ is equal to its cost in the interval $[0,1]$ plus its cost
in the interval $[1,T_1+1]$. Equation \eqref{4-02b} states that the
inequality holds, which is enough for the argument.  By lower
semicontinuity of $S_0$, and since $v^{(\gamma)} (T_1) \to \bar\rho$
as $T_1\to\infty$,
$S_0(\bar\rho) \le \liminf_{T_1\to\infty} S_0(v^{(\gamma)} (T_1))$.
Hence, for all $\epsilon >0$, there exists $T_1$ large enough such
that
\begin{equation*}
I_{[0,T_1+1]}(\, w \,|\, \bar\rho\,) \;
\le\; S_0(\gamma) \,-\,  S_0(\bar\rho) \;+\; \epsilon \;.
\end{equation*}
This proves that $V(\gamma) \le S(\gamma)$, as claimed.

\begin{proof}[Proof of Lemma \ref{l08}]
Denote by $H$ the function in $\mc H^1(\Omega_T)$ introduced in Lemma
\ref{l09}, and recall from (\ref{4-02}) the definition of $\Gamma$.
Set $K:=\Gamma-H$, so that $G=H$. The function $K$ belongs to
$\mc H^1(\Omega_T)$ because, by hypothesis, $u\in D_{T,\delta}$ and,
by Lemma \ref{l04}, $F\in C^{1,2} (\, [0,T]\times [0,1]\,)$. Then
\eqref{4-06} follows easily from \eqref{5-01}.

We turn to the identity \eqref{4-07}, Note that
$\partial_t u \;=\; \Delta u \;-\; 2\, \nabla (\, \sigma(u) \,
\nabla \, [ \, \Gamma \,-\, K \,] \,)$. In \eqref{4-03}, replace
$\partial_t u(t)$ by the right-hand side of this identity and
integrate by parts to get that
\begin{equation*}
\begin{aligned}
& S_0(u(T)) \,-\, S_0(\psi) \; =\;
-\, \int_0^T  \big\langle \nabla u(t) \,,\, \nabla \Gamma (t)
\big\rangle \; dt \\
& \qquad +\; 2\, \int_0^T  \big\langle \, \sigma(u(t))\, 
\nabla (\, \Gamma (t) - K(t)\,) \, ,\,
\nabla \Gamma (t) \, \big\rangle \; dt 
\\
&\qquad +\; \int_0^T  \mf p_{\beta, B} \big(\, u_t(1)\,,\, H_t(1) \,\big)
\, \Gamma_t (1) \, dt \;+\;
\int_0^T \mf p_{\alpha, A} \big(\, u_t(0)\,,\, H_t(0) \,\big)
\, \Gamma_t (0)\; dt \;.
\end{aligned}
\end{equation*}
By Lemma \ref{l06}, the previous expression is equal to
\begin{equation*}
\begin{aligned}
& - \, \int_0^T  \big\langle \, \sigma(u(t))\, 
(\, \nabla  \Gamma (t) \,) ^2 \, \big\rangle \; dt 
\; +\; 2\, \int_0^T  \big\langle \, \sigma(u(t))\, 
\nabla (\, \Gamma (t) - K(t)\,) \, ,\,
\nabla \Gamma (t) \, \big\rangle \; dt 
\\
&  -\; \int_0^T  \mf b_{\beta, B} \big(\, u_t(1)\, ,\,
\Gamma_t (1)\, \big) \; dt
\;-\; \int_0^T \mf b_{\alpha, A} \big(\, u_t(0)\, ,\, \Gamma_t (0)\,
\big)\; dt
\\
& + \; \int_0^T  \mf p_{\beta, B} \big(\, u_t(1)\,,\, H_t(1) \,\big)
\, \Gamma_t (1) \; dt \;+\;
\int_0^T \mf p_{\alpha, A} \big(\, u_t(0)\,,\, H_t(0) \,\big)
\, \Gamma_t (0)\; dt \;.
\end{aligned}
\end{equation*}

Up to this point, we expressed the difference
$S_0(u(T)) \,-\, S_0(\psi)$ as a sum of many terms. Add on both sides of
this identity $\int_0^T \big\langle \, \sigma (u(t))\, (\, \nabla K (t) \,) ^2 \,
\big\rangle \; dt$. Add and subtract on the right-hand side
\begin{equation*}
\int_0^T \mf c_{\beta, B} \big(\, u_t(1)\,,\, H_t(1)\, \big)\; dt 
\; +\; \int_0^T
\mf c_{\alpha, A} \big(\, u_t(0)\,,\, H_t(0)\, \big)\; dt\;.
\end{equation*}
Recall that $H = \Gamma - K = G$ and, from \eqref{5-03}, the identity
satisfied by $I_{[0,T]}(u|\psi)$, to get after these summations that
\begin{equation*}
\begin{aligned}
& S_0(u(T)) \;-\;  S_0(\psi)  \;+\;
\int_0^T \big\langle \, \sigma (u(t))\, (\, \nabla K (t) \,) ^2 \,
\big\rangle \; dt \\
&\; =\; I_{[0,T]}(u|\psi) \;-\; \int_0^T e^{H_t(1)}\,
\mf q_{\beta, B} (u_t(1), K_t(1)) \; dt
\;-\; \int_0^T e^{H_t(0)}\, \mf q_{\alpha, A} (u_t(0), K_t(0)) \; dt \;,
\end{aligned}
\end{equation*}
as claimed [because $H=G$].
\end{proof}

We turn to the proof of the upper bound for the quasi-potential, as
described in the Strategy of the proof.  We first simplify the
boundary conditions in equations \eqref{4-08}--\eqref{4-08b}.

As $H_t = \log [v_t/(1-v_t)] - \log [F_t/(1-F_t)]$, the boundary terms
are given by
\begin{equation}
\label{4-09}
\begin{gathered}
- \, \nabla v_t \,+\, 2\, \frac{\sigma(v_t)}{\sigma(F_t)} \, \nabla F_t \,=\,
\frac{1}{B\, \sigma(F_t)} \, \big\{\, \beta\, v_t\,
[\, 1- F_t\,]^2 \,-\, (1-\beta) \, (1-v_t) \, F_t^2 \,\big\}
\;, \\
- \, \nabla v_t \,+\, 2\, \frac{\sigma(v_t)}{\sigma(F_t)} \, \nabla F_t \,=\,
\frac{-\, 1}{A\, \sigma(F_t)} \, \big\{\, \alpha\, v_t\,
[\, 1- F_t\,]^2 \,-\, (1-\alpha) \, (1-v_t) \, F_t^2 \,\big\}
\;, 
\end{gathered}
\end{equation}
for $x=1$, $0$, respectively. Let $\color{bblue} R = \log [ F/(1-F)]$.
With this notation, equation \eqref{4-08} can be written as
\begin{equation}
\label{4-11}
\left\{
\begin{aligned}
& \partial_t v \,=\,  \Delta v  \,-\, 2\,  \nabla \big( \, \sigma (v) 
\, \nabla R \, \big)   \quad (t,x) \in
(0,\infty)\times (0,1) \;, \\
& \nabla v_t (1) \,-\, 2\, \sigma(v_t(1)) \, \nabla R_t(1) \,=\,
\mf p_{1-\beta, B} \big(\, v_t(1)\,,\, R_t(1)\, \big) \;, \\
& \nabla v_t (0) \,-\, 2\, \sigma(v_t(0)) \, \nabla R_t(0) \,=\,
-\, \mf p_{1-\alpha, A} \big(\, v_t(0)\,,\, R_t(0)\, \big) \;, \\
& v(0,\cdot) = \gamma \;, \quad  x\in [0,1]\;,
\end{aligned}
\right. 
\end{equation}
where $F_t$ is the solution to \eqref{4-08b}. Note that this equation
corresponds to equation \eqref{4-11c}.

Proposition \ref{l10} provides a precise meaning for the system of
equations \eqref{4-08}--\eqref{4-08b}, or, equivalently,
\eqref{4-11}--\eqref{4-08b}. This proposition requires an estimate on
the solutions to equation \eqref{1-06}. Denote by $u^{(\gamma)}$,
$\gamma\in\ms M_{\rm ac}$, the solution to \eqref{1-06} with initial
conditions $\gamma$, and by $F = F(\gamma)$ the solution to
\eqref{3-01}. Recall the definition of the constant $p$ and $q$
introduced in \eqref{3-15}.

\begin{lemma}
\label{l16}
For every $\gamma\in\ms M_{\rm ac}$, $(t,x) \in \bb R_+ \times [0,1]$,
$\alpha \,+\, Ap \, \le\, u^{(F(\gamma))}(t,x) \,\le\, \beta \,-\,
Bp$. Moreover, for every $T>0$, there exists a constant
$c_1 = c_1(A, B, \alpha, \beta, T)>0$ such that
$c_1 \le \nabla u^{(F(\gamma))} (t, x) \le c^{-1}_1$ for all
$\gamma\in\ms {M}_{\rm ac}$, $(t,x)\in [0,T]\times [0,1]$.
\end{lemma}

\begin{proof}
By \eqref{3-07} and Corollary \ref{l02},
$\alpha \,+\, Ap \, \le\, F(\gamma) (x) \,\le\, \beta \,-\, Bp$ for
all $\gamma\in\ms M_{\rm ac}$, $x \in [0,1]$.  The first assertion of
the lemma follows from Theorem \ref{mt4}.

By Corollary \ref{l02}, $F = F(\gamma)$ belongs to $\mc B$. Therefore,
$p\le \nabla F (x) \le q$ for all $0\le x\le 1$,
$\gamma\in\ms M_{\rm ac}$. Let $v = \nabla u^{(F(\gamma))}$. Then, $v$
solves the equation
\begin{equation}
\label{1-06d} 
\begin{cases}
\partial_t v \,=\, \Delta v\\
v (t,0) \,=\, A^{-1} \,[\, u^{(F(\gamma))}(t,0) - \alpha\,] \\
v (t,1) \,=\, B^{-1}\, [\beta \,-\, u^{(F(\gamma))}(t,1)\,] \\
v (0,\cdot) \,=\, \nabla F (\cdot)\;.
\end{cases}
\end{equation}
The maximum principle, Theorem 2 of \cite[Chapter 3]{PW84}, states
that the maximum and the minimum of $v$ are attained at the boundary.
The assertion of the lemma follows from the bounds on
$u^{(F(\gamma))}$ and $\nabla F$ obtained above. These estimates are
uniform over $\gamma\in\ms M_{\rm ac}$.
\end{proof}

\begin{proposition}
\label{l10} 
Fix $\gamma\in\ms {M}_{\rm ac}$, and denote by
$F^{(\gamma)}(t) = u^{(F(\gamma))} (t)$ the solution to the heat
equation \eqref{1-06} with initial condition
$F^{(\gamma)} (0, \,\cdot\,) = F(\gamma) (\,\cdot\,)$. Define
$v^{(\gamma)}=v^{(\gamma)}(t,x)$ by \eqref{2-06}. Then,
$v^{(\gamma)} (0, \,\cdot\,) = \gamma(\,\cdot\,)$, $v^{(\gamma)}$ is
smooth in $(0,\infty)\times[0,1]$ and $(v^{(\gamma)}, F^{(\gamma)})$
satisfies \eqref{4-08b}, \eqref{4-11} in $(0,\infty)\times[0,1]$.
\end{proposition}

\begin{proof}
Fix $\gamma\in\ms {M}_{\rm ac}$, and let $F(\gamma)$ be the solution
of \eqref{3-01}. By Corollary \ref{l02}, $F(\gamma)$ belongs to
$C^1 ([0,1])$ and there is a constant $c_0\in(0,\infty)$, depending
only on the parameters, such that
$c_0 \le [\nabla F(\gamma)]\, (x) \le c^{-1}_0$ for all $x\in [0,1]$.

Let $F^{(\gamma)}(t)$ be the solution to \eqref{1-06} with initial
condition $F^{(\gamma)}(0) = F(\gamma)$. By Theorem \ref{mt4},
$F^{(\gamma)}$ is smooth in $(0,\infty) \times [0,1]$. By Lemma
\ref{l16}, for every $T>0$, there exists
$c_1 = c_1(A, B, \alpha, \beta, T)\in (0,\infty)$ such that
$c_1 \le ( \nabla F^{(\gamma)})(t, x) \le c^{-1}_1$ for all
$(t,x)\in [0,T]\times [0,1]$. 

Define $v^{(\gamma)}_t$ by equation \eqref{2-06} which we reproduce
here:
\begin{equation}
\label{4-10}
v^{(\gamma)}_t \;=\; F^{(\gamma)}_t \; +\;  \sigma(\, F^{(\gamma)}_t \,)\;
\frac {\Delta F^{(\gamma)}_t}{(\nabla F^{(\gamma)}_t)^2} \;, \quad t\,\ge \, 0\;.
\end{equation}
For $t=0$, we may replace on the right-hand side $F^{(\gamma)}_0$ by
$F (\gamma)$ to get that $v^{(\gamma)}_0 = \gamma$, as claimed.

In view of the regularity and the bounds obtained for $F^{(\gamma)}$
in the previous paragraph, $v^{(\gamma)}$ is smooth in
$(0,\infty)\times [0,1]$. Moreover, by \eqref{4-10}, the pair
$(v^{(\gamma)}, F^{(\gamma)})$ satisfies \eqref{4-08b}.

It remains to show that $v^{(\gamma)}$ fullfils \eqref{4-11}.  

\smallskip\noindent{\it Claim 1: The function $v^{(\gamma)}_t$
complies with the boundary conditions of \eqref{4-11}}. \smallskip

We prove this assertion for $x=1$, the other one being similar. By
\eqref{4-10}, at the boundary, as
$\nabla F^{(\gamma)}_t(1) = [\beta - F^{(\gamma)}_t(1)]/B$, taking a
time-derivative on both sides of the identity yields that
$(\nabla^3 F^{(\gamma)}_t)(1) = - \, (1/B) \, (\Delta F^{(\gamma)}_t)
(1)$ because $\partial_t F^{(\gamma)} = \Delta F^{(\gamma)}$.

Compute, separately, the right and left-hand sides of the right
boundary condition in \eqref{4-11}. We start with the left-hand side
of the identity. Since
$(\nabla^3 F^{(\gamma)}_t)(1) = - (1/B) (\Delta F^{(\gamma)}_t) (1)$,
taking a space derivative in \eqref{4-10} yields that, at $x=1$,
\begin{equation*}
-\, \nabla v^{(\gamma)} \; =\; -\, \nabla F^{(\gamma)} \, - \, (1-2F^{(\gamma)})\,
\frac {\Delta F^{(\gamma)}}{ \nabla F^{(\gamma)}} \;+\;
\frac {\sigma(F^{(\gamma)})}{B} \, \frac {\Delta F^{(\gamma)}}{ (\nabla F^{(\gamma)})^2} 
\;+\; 2\, \sigma(F^{(\gamma)})\, \frac {(\Delta F^{(\gamma)})^2}{ (\nabla F^{(\gamma)})^3}\;\cdot
\end{equation*}
On the other hand, by \eqref{4-10} and a straightforward computation,
\begin{equation*}
2\, \frac{\sigma(v^{(\gamma)})}{\sigma(F^{(\gamma)})} \, \nabla F^{(\gamma)}\;=\;
2\, \nabla F^{(\gamma)} \, \Big\{\, 1\,+\, (1-2F^{(\gamma)}) \,
\frac {\Delta F^{(\gamma)}}{ (\nabla F^{(\gamma)})^2} \,-\,
\sigma(F^{(\gamma)}) \, \frac {(\Delta F^{(\gamma)})^2}{ (\nabla F^{(\gamma)})^4} \,\Big\}\;.
\end{equation*}
Summing the previous two identities yields that, at $x=1$,
\begin{equation*}
-\, \nabla v^{(\gamma)} \; + \;
2\, \frac{\sigma(v^{(\gamma)})}{\sigma(F^{(\gamma)})} \, \nabla F^{(\gamma)}
\;=\;
\nabla F^{(\gamma)} \;+\; \{\, \beta \,-\, 2 \beta F^{(\gamma)}
\, +\, (F^{(\gamma)})^2 \, \}\,
\frac{\Delta F^{(\gamma)}}{B\, (\nabla F^{(\gamma)})^2}\;\cdot
\end{equation*}
A simple calculation gives that the right-hand side of the right
boundary condition in \eqref{4-11} is also equal to this
quantity. This proves Claim 1.

\smallskip\noindent{\it Claim 2: The function $v^{(\gamma)}_t$
fullfils \eqref{4-11} in the interior}. \smallskip

The proof of this claim is
identical to the one presented in \cite[Appendix B]{BDGJL} and in
\cite[Lemma 5.5]{BDGJL2003}. We reproduce it here in sake of
completeness.

From \eqref{4-10},
\begin{equation*}
\frac{ v^{(\gamma)}(1-v^{(\gamma)})}{ F^{(\gamma)}(1-F^{(\gamma)})} = 
1 + (1-2F^{(\gamma)}) \frac{\Delta F^{(\gamma)}}{ \big(\nabla F^{(\gamma)}\big)^2} 
- F^{(\gamma)}(1-F^{(\gamma)}) \frac{\big(\Delta F^{(\gamma)}\big)^2}{\big(\nabla
F^{(\gamma)}\big)^4}\;\cdot 
\end{equation*}
As $F^{(\gamma)}$ solves the heat equation \eqref{1-06}, a long computation
yields that 
\begin{equation*}
\big(\,  \partial_t \,-\, \Delta \,\big) 
\, \Big(\, \sigma( F^{(\gamma)}) \, \frac {\Delta F^{(\gamma)}}
{\big(\nabla F^{(\gamma)}\big)^2} \,\Big) 
\;=\; - \, 2 \, \nabla \Big( \frac{\sigma(v^{(\gamma)})}{ \sigma(F^{(\gamma)})}
\, \nabla F^{(\gamma)} \Big)\;.
\end{equation*}
By \eqref{4-10}, $v^{(\gamma)}$ satisfies the differential equation in
\eqref{4-11}, as claimed.
\end{proof}

\begin{lemma}
\label{l10b}
Under the hypotheses of Proposition \ref{l10}, assume that $\gamma$
belongs to $C^2([0,1])$.  Then, $v^{(\gamma)}$ belongs to
$C^{1,2} ([0,\infty)\times[0,1]) \cap C ([0,\infty);\ms M_{\rm ac})$.
\end{lemma}

\begin{proof}
Assume that $\gamma$ belongs to $C^2([0,1])$. By Remark \ref{rm02} and
\eqref{3-01}, $\Delta F (\gamma)$ belongs to $C^2([0,1])$.  Taking
time derivatives in \eqref{1-06} yields that $\Delta F^{(\gamma)} (t)$
is the solution to \eqref{1-06h} with initial condition
$\Delta F (\gamma)$. By Theorem \ref{mt6} (a),
$(t,x) \mapsto \Delta F^{(\gamma)} (t,x)$ belongs to
$C^{1,2} ([0,\infty)\times[0,1])$.  Therefore, by \eqref{4-10},
$v^{(\gamma)}$ belongs to $C^{1,2} ([0,\infty)\times[0,1])$

\smallskip\noindent{\it Claim 1: The function $v^{(\gamma)}_t$ belongs to $C(\bb
R_+, \ms M_{\rm ac})$}. \smallskip

We have to show that $0\le v^{(\gamma)}(t,x) \le 1$ for all
$(t,x)$. In view of \eqref{4-11d}, equation \eqref{4-11} describes the
macroscopic evolution of the density for weakly asymmetric boundary
driven exclusion processes with weak boundary interaction. The drift
is given by $\nabla R_t$ where
$R_t = \log [ F^{(\gamma)}_t/(1-F^{(\gamma)}_t)]$. By the first part
of the proof and Lemma \ref{l16}, $R_t$ belongs to
$C^{1,2} ([0,\infty)\times[0,1])$.

As $v^{(\gamma)}_0 = \gamma$ and $0\le \gamma \le 1$, by the
hydrodynamic limit of theses systems, derived in \cite{FGLN2021}, this
equation has a weak solution taking values in the interval $[0,1]$,
cf. Theorem \ref{mt5}. Since $v^{(\gamma)}$ belongs to
$C^{1,2} ([0,\infty)\times[0,1])$ and solves \eqref{4-11} pointwisely,
it is a weak solution to equation \eqref{4-11} in the sense of
Definition \ref{d03}. Therefore, by the uniqueness of weak solutions
of \eqref{4-11}, Theorem \ref{mt5}, $v^{(\gamma)}$ coincides with
solution obtained in the proof of the hydrodynamic limit which takes
values in $[0,1]$. This completes the proof of the lemma.
\end{proof}


\begin{lemma}
\label{l11}
Under the hypotheses of Proposition \ref{l10},
assume that $\delta \le \gamma(x) \le 1-\delta$ a.e.\ for some
$\delta>0$. Then, there exists
$\delta'=\delta'(A, B, \alpha, \beta,\delta) \in (0,1)$, such that
$\delta' \le v^{(\gamma)} (t,x) \le 1-\delta'$ for all
$(t,x) \in (0,\infty) \times [0,1]$.
\end{lemma}

\begin{proof}
The proof is divided in several assertions. Fix $t>0$.

\smallskip\noindent {\it Claim 1:} If $v^{(\gamma)}(t, \cdot)$ has a
local maximum at $x_0\in (0,1)$ and $v^{(\gamma)}(t,x_0) > 1-\alpha$,
then $\partial_t v^{(\gamma)} (t,x_0)<0$. \smallskip

Assume that $v^{(\gamma)}(t, \cdot)$ has a local maximum at
$x_0\in (0,1)$. Since $v^{(\gamma)}$ is a smooth solution to
\eqref{4-08}, $(\nabla v^{(\gamma)})(t,x_0)=0$. By \eqref{4-10} and a
straightforward computation,
$\Delta \log \{F^{(\gamma)} /(1-F^{(\gamma)})\} = (v^{(\gamma)} +
F^{(\gamma)} -1) \, (\nabla
F^{(\gamma)})^2/\sigma(F^{(\gamma)})^2$. Therefore, by \eqref{4-11},
at the point $(t,x_0)$,
\begin{equation*}
\partial_t v^{(\gamma)} \;=\; \Delta v^{(\gamma)} 
\;-\; 2\, \sigma (v^{(\gamma)}) \, (v^{(\gamma)} + F^{(\gamma)} -1) \,
\frac{ \, (\nabla F^{(\gamma)})^2}{\sigma(F^{(\gamma)})^2}  \;\cdot
\end{equation*}
As $x_0$ is a local maximum, $\Delta v^{(\gamma)} \le 0$. On the other
hand, since $v^{(\gamma)}(t,x_0) > 1-\alpha$, and, by Lemma \ref{l16},
$\alpha \le F^{(\gamma)}$, $v^{(\gamma)} + F^{(\gamma)} -1 >0$ so that
$\partial_t v^{(\gamma)} <0$, which proves the claim.

The same argument shows that $(\partial_t v^{(\gamma)})(t,x_1) >0$ if
$x_1\in (0,1)$ is a minimum of $v^{(\gamma)}(t,\cdot)$ and
$ v^{(\gamma)} (t,x_1) < 1-\beta$.

We turn to the possibility that the maximum is attained at the
boundary.  By Lemma \ref{l16}, $F^{(\gamma)}$ takes value in the
interval $[\alpha + pA , \beta -pB]$.  Let
\begin{equation*}
\nu \;: =\; \max_{\alpha + pA \le \varphi \le \beta -pB}
\frac{(1-\beta) \varphi^2}{\beta - 2 \beta \varphi + \varphi^2} \;<\;
1\;. 
\end{equation*}
As $\nu < 1$, there exists $a>0$ such that $\nu/(1-2a)<1$ and $aBp<1$.

\smallskip\noindent {\it Claim 2:} If
$v^{(\gamma)}_t(1) \,>\, {\color{bblue} \nu_r \, :=\, \max\{
\nu/(1-2a) \,,\, 1 - a B p\}}$, then $\nabla
v^{(\gamma)}_t(1)<0$. \smallskip

Since $v^{(\gamma)}$ solves equation \eqref{4-11}, in view of the
definition of $R$, the first equation in \eqref{4-09} holds with
$v^{(\gamma)}$, $F^{(\gamma)}$ in place of $v$, $F$, respectively. At
the boundary,
$\nabla F^{(\gamma)}_t(1) = (1/B) [ \beta -
F^{(\gamma)}_t(1)]$. Therefore,
\begin{equation*}
\begin{aligned}
\nabla v^{(\gamma)}_t \, &= \, \frac{1}{B\, \sigma(F^{(\gamma)}_t)} \,
\Big\{  2\, \sigma(v^{(\gamma)}_t) \, [\beta - F^{(\gamma)}_t]
\,+\, (1-\beta) \, (1-v^{(\gamma)}_t) \, [F^{(\gamma)}_t]^2
\,-\, \beta\, v^{(\gamma)}_t\, [\, 1- F^{(\gamma)}_t\,]^2  \, \Big\} \\
& = \, \frac{- 1}{B\, \sigma(F^{(\gamma)}_t)} \, \Big\{\,
v^{(\gamma)}_t \, \Big[\, \{\, \beta - 2\beta F^{(\gamma)}_t + [F^{(\gamma)}_t]^2\,\}
\,-\, 2\, (1-v^{(\gamma)}_t)\, (\beta-F^{(\gamma)}_t)\, \Big]\, -\,
(1-\beta) \, [F^{(\gamma)}_t]^2\, \Big\} \;.
\end{aligned}
\end{equation*}
In these formulas, we wrote $v^{(\gamma)}_t$, $F^{(\gamma)}_t$ for
$v^{(\gamma)}_t(1)$, $F^{(\gamma)}_t(1)$, respectively. As
$v^{(\gamma)}_t(1) > 1- a\, B\, p$ and $Bp \le \beta-F^{(\gamma)}_t$ ,
the last term is bounded by
\begin{equation*}
\begin{aligned}
& \frac{- 1}{B\, \sigma(F^{(\gamma)}_t)} \, \Big\{\,
v^{(\gamma)}_t \, \Big[\, \{\, \beta - 2\beta F^{(\gamma)}_t
+ [F^{(\gamma)}_t]^2\,\}
\,-\, 2\, a \, (\beta-F^{(\gamma)}_t)^2\, \Big]\, -\,
(1-\beta) \, [F^{(\gamma)}_t]^2\, \Big\} \\
& \quad \le\;
\frac{- 1}{B\, \sigma(F^{(\gamma)}_t)} \, \Big\{\,
(1-2a) \, v^{(\gamma)}_t\, \{\, \beta - 2\beta F^{(\gamma)}_t
+ [F^{(\gamma)}_t]^2 \, \} \, -\,
(1-\beta) \, [F^{(\gamma)}_t]^2 \, \Big\}\;.
\end{aligned}
\end{equation*}
This expression is negative by definition of $\nu$ and because we
assumed that $v^{(\gamma)}_t(1) > \nu/(1-2a)$. This proves Claim
2. \smallskip

Let
\begin{equation*}
\mu_r \;: =\; \min_{\alpha + pA \le \varphi \le \beta -pB}
\frac{(1-\beta) \varphi^2}{\beta - 2 \beta \varphi + \varphi^2} \;>\;
0\;. 
\end{equation*}

\smallskip\noindent {\it Claim 3:} If $v^{(\gamma)}_t(1) \,<\, \mu_r$,
then $\nabla v^{(\gamma)}_t(1)>0$. \smallskip

Recall the formula for $\nabla v^{(\gamma)}_t(1)$ presented at the
beginning of the proof of Claim 2. Since
$2(1-v^{(\gamma)}_t)\, (\beta-F^{(\gamma)}_t)\,>\,0$ and
$v^{(\gamma)}_t(1) \,<\, \mu_r$,
\begin{equation*}
\nabla v^{(\gamma)}_t \, > \,
\frac{1}{B\, \sigma(F^{(\gamma)}_t)} \, \Big\{\,
(1-\beta) \, [F^{(\gamma)}_t]^2 \, -\,
\mu_r \, \{\, \beta - 2\beta F^{(\gamma)}_t + [F^{(\gamma)}_t]^2\,\}
\, \Big\} \;.
\end{equation*}
As $F^{(\gamma)}_t$ takes values in the interval
$[\alpha + pA , \beta -pB]$, the right-hand side is non-negative by
definition of $\mu_r$. This proves Claim 3. \smallskip

Similarly, one can be prove the existence of $\mu_l >0$ and $\nu_l<1$
such that $\nabla v^{(\gamma)}_t(0)>0$ if
$v^{(\gamma)}_t(0) \,>\, \nu_l$; $\nabla v^{(\gamma)}_t(0)<0$ if
$v^{(\gamma)}_t(0) \,<\, \mu_l$.

This result together with Claims 1, 2, 3, yield that
$\min\{ \mu_l \,,\, \mu_r \,,\, \delta \,,\, 1-\beta \,\} \,\le\,
v^{(\gamma)} (t,x) \,\le\, \max\{ \nu_l \,,\, \nu_r \,,\, 1-\delta
\,,\, 1-\alpha \,\}$ for all $(t,x) \in \bb R_+ \times [0,1]$, which
concludes the proof of the lemma.
\end{proof}

Next result states that the solution to \eqref{4-08}, as constructed
in Proposition \ref{l10}, converges to $\bar\rho$, as $t\to\infty$,
uniformly with respect to the initial condition $\gamma$.

\begin{lemma}
\label{l12}
Let $v^{(\gamma)}$, $\gamma\in\ms M_{\rm ac}$, be given by
\eqref{4-10}.  Then,
\begin{equation*}
\lim_{t\rightarrow \infty}  \sup_{\gamma \in\ms {M}_{\rm ac}} \; 
\big\Vert \, v^{(\gamma)} (t) -\bar\rho \,\big\Vert_\infty \; =\; 0 \; .
\end{equation*}
\end{lemma}

\begin{proof}
Write the solution $F^{(\gamma)}(t)$ of \eqref{1-06} as
$F^{(\gamma)}(t,x) = \bar\rho(x) + \Psi^{(\gamma)} (t,x)$. Then,
$\Psi^{(\gamma)}$ solves the equation \eqref{1-06h} with
$\phi = F( \gamma ) \,-\, \bar\rho$.  In particular, by Theorem
\ref{mt6}, $\Psi^{(\gamma)}(t)$ can be represented as
$\Psi^{(\gamma)}(t) = P^{(R)}_t \Psi^{(\gamma)}(0)$. Since
$\Psi^{(\gamma)}(0)= F(\gamma) -\bar\rho$ and the solution $F(\gamma)$
of \eqref{3-01} as well as $\bar \rho$ are contained in the interval
$[\alpha, \beta]$, we have that
$\|\Psi^{(\gamma)}(0)\|_\infty \le \beta -\alpha <1$, uniformly over
$\gamma\in\ms M_{\rm ac}$. Therefore, expressing
$\Psi^{(\gamma)}(t) = P^{(R)}_t \Psi^{(\gamma)}(0)$ in the basis
$(f_j : j\ge 1)$ by standard arguments (see \cite[Corollary 2 of
Section 4.4]{zeid108}),
\begin{equation*}
\lim_{t\to\infty} \sup_{\gamma \in\ms M_{\rm ac}} \; 
\| \Psi^{(\gamma)}(t) \|_\infty  \;=\; 0 \;.
\end{equation*}

Taking a time derivative at the boundary yields that
$\Delta \Psi^{(\gamma)} (t)$ also solves equation \eqref{1-06h}. By
\eqref{3-01}, as $F(\gamma)$ belongs to $\mc B$,
$\Vert \Delta \Psi^{(\gamma)}(0)\Vert_\infty \le
C_0(A,B,\alpha,\beta)$. Thus, as $\Delta \bar\rho=0$, by the same
argument,
\begin{equation*}
\lim_{t\to\infty} \sup_{\gamma \in\ms M_{\rm ac}} \; 
\| \Delta \Psi^{(\gamma)}(t) \|_\infty  \;=\; 0 \;.
\end{equation*}

Expressing the derivative $(\nabla \Psi^{(\gamma)}) (t,x)$ as
$(\nabla \Psi^{(\gamma)}) (t,0) + \int_0^x (\Delta
\Psi^{(\gamma)})(t,y)\, dy$, and since
$(\nabla \Psi^{(\gamma)}) (t,0) = A^{-1} \, \Psi^{(\gamma)}(t,0)$, we
deduce from the two previous results that
\begin{equation*}
\lim_{t\to\infty} \sup_{\gamma \in\ms M_{\rm ac}} \; 
\| (\nabla \Psi^{(\gamma)}) (t) \|_\infty  \;=\; 0 \;.
\end{equation*}

By \eqref{4-10}, the assertion of the lemma follows from the previous
estimates.
\end{proof}

Lemma \ref{l12} shows that we may join a profile $\gamma$ in
$\ms M_{\rm ac}$ to a neighborhood of the stationary profile by using
the equation \eqref{4-08} for a time interval $[0,T_1]$ which at the
same time regularizes the profile. As described in the Strategy of the
proof, it remains to connect $v^{(\gamma)}(T_1)$ to $\bar\rho$. In the
next lemma we show that this can be done by paying only a small price.
Denote by $\Vert\,\cdot\, \Vert_2$ the norm in $\ms L^2([0,1])$, and
recall the definition of $\delta_0$, given at the beginning of this
section, and of the set $D_{T,\delta}$, introduced in \eqref{4-01}. In
the lemma below, $\lambda_1$ represents the smallest eigenvalue of the
Robin Laplacian (cf. Appendix \ref{sec04}).

\begin{lemma}
\label{l13} 
Let $\gamma \in \ms {M}_{\rm ac}$ be a smooth profile such that
$\|\gamma - \bar\rho\|_\infty \le \delta_0 \, \min\{(1/4)$,
$(1/\Lambda)\}$, where $\Lambda = 16 \sqrt{A/\lambda_1}$. Then, there
exist a smooth path $w(t)$, $t\in [0,1]$, with
$\delta_0/2 \le w(t) \le 1- \delta_0/2$, $w(0)=\bar\rho$, $w(1)=\gamma$
and a finite constant $C_0=C_0(\delta_0)$ such that
\begin{equation*}
I_{[0,1]}(w|\bar\rho) \;\le\; C_0\,  \| \gamma -\bar\rho\|_2^2
\;. 
\end{equation*}
In particular, for profiles $\gamma$ satisfying the hypotheses of this
lemma, $V(\gamma) \,\le\,  C_0 \, \| \gamma -\bar\rho\|_2^2$.
\end{lemma}

The ``straight path'' $w (t) = \bar\rho\, (1-t) + \gamma \,t$ yields a
bound in terms of the $\mc H_{1}([0,1])$ norm of $\gamma-\bar\rho$. In
contrast, the path below, similar to the one proposed in
\cite{BDGJL2003}, provides a bound in terms of the $\ms L^2$ norm.  We
assume in the proof below that the reader is familiar with the
notation and results presented in Appendix \ref{sec04}.

\smallskip
\begin{proof}
Recall that we denote by $\{f_j : j\ge 1\}$ the orthonormal
eigenfunctions of the Robin Laplacian and by $\lambda_j$ the
associated eigenvalues.

We claim that the path
$w(t)=w(t,x)$, $(t,x)\in [0,1]\times [0,1]$ given by
\begin{equation}
\label{6-06}
w(t) \;=\;   \bar\rho \;+\; \sum_{k\ge 1}
\frac{e^{\lambda_k t} -1}{e^{\lambda_k} -1} \,
\langle \gamma -\bar\rho \,,\, f_k\rangle \, f_k
\end{equation}
fulfills the conditions stated in the lemma. By \cite[Corollary 2 of
Section 4.4]{zeid108}, this sum is absolutely convergent, uniformly in
$x\in [0,1]$.

Clearly, $w(0)=\bar\rho$, $w(1)=\gamma$ and $w$
satisfies the boundary conditions
\begin{equation}
\label{6-03}
(\nabla w) (t,0) \,=\, A^{-1} \, [w (t,0) - \alpha ]\;, \quad
(\nabla w)(t,1) \,=\, B^{-1} [\beta - w (t,1)]\;.
\end{equation}
By the smoothness assumption on $\gamma$,
$w\in C^{1,2}([0,1]\times [0,1])$.

In order to show that $\delta_0/2 \le w\le 1- \delta_0/2$, write
$w(t)$ as $w(t)=\bar\rho + q(-t)$. Clearly, $q(t)=q(t,x)$,
$(t,x)\in [-1,0]\times[0,1]$ solves the equation
\begin{equation}
\label{5-04}
\left\{
\begin{aligned}
& \partial_t q (t) \;=\;  \Delta q(t) \,-\,  g \\
& (\nabla q) (t,0) \,=\, A^{-1} \, q(t,0) \;,\\
& (\nabla q)(t,1) \,=\, -\, B^{-1} q (t,1) \;,\\
& q(-1) = \gamma -\bar\rho \;, 
\end{aligned}
\right.
\end{equation}
where $g=g(x)$ is given by
\begin{equation*}
g \;=\; \sum_{k\ge 1} \frac{\lambda_k}{e^{\lambda_k} -1} 
\langle \gamma -\bar\rho \,,\, f_k\rangle \, f_k\;.
\end{equation*}

Recall the definition of the $\mc H_R$-norm $\| g\|_{\mc H_R}$ induced 
by the Robin Laplacian. By definition of $g$ and since
$a^2 \le 2( e^a -1)$, $a>0$,
$\| \gamma -\bar\rho \|_\infty \le \delta_0/\Lambda$,
\begin{equation*}
\begin{aligned}
\| g\|_{\mc H_R}^2 \;& =\; \sum_{k\ge 1} \lambda_k \,
\Big( \frac{\lambda_k}{e^{\lambda_k} -1} \Big)^2 \, \langle \gamma
-\bar\rho, f_k\rangle^2 \;\le\; \frac{4}{\lambda_1} \sum_{k\ge 1}
\langle \gamma -\bar\rho, f_k\rangle^2 \\
& =\; \frac{4}{\lambda_1} \, \| \gamma -\bar\rho \|_2^2 \;\le\;
\frac{4}{\lambda_1} \frac{\delta_0^2}{\Lambda^2} \;\cdot
\end{aligned}
\end{equation*}

The solution $q(t)$ of \eqref{5-04} can be expressed as
$P_{t+1}^{(R)} (\gamma -\bar\rho) + \int_{-1}^t  P^{(R)}_{t-s} g\,
ds$. Therefore, since, by hypothesis, $\|\, \gamma -\bar\rho \,
\|_\infty \le \delta_0/4$, by \eqref{5-05}, \eqref{6-02} and the
previous bound,  
\begin{equation*}
\sup_{t\in[-1,0]} \|q(t)\|_\infty \;\le\; \|\, \gamma -\bar\rho \,
\|_\infty \;+\;
\Vert\, g \, \|_\infty \;\le\;
\frac{\delta_0}{4} \;+\;
\sqrt{2(A\vee 1)}\, \frac{2}{\sqrt{\lambda_1}}
\, \frac{\delta_0}{\Lambda} \; \cdot
\end{equation*}
In particular, by definition of $\Lambda$, $w$ belongs to
$D_{1,\delta_0/2}$.

\smallskip

We turn to the cost of the path $w$. Since $w$ is a smooth path such that
$w(0) = \bar\rho$, in formula \eqref{1-01}, integrate by parts twice
in space and once in time to get that 
\begin{equation*}
\begin{aligned}
J_{1,H} ( w ) \; & =\;
\int_0^{1} \big\langle  \partial_t w_t \,-\, \Delta w_t
\,,\, H_t  \big\rangle \; dt
\;-\; \int_0^{1}  \big\langle \sigma( w_t ), 
\big( \nabla H_t \big)^2 \big\rangle \; dt
\\
& -\;  \int_0^{1} \nabla w_t(1)\,   H_t (1) \; dt 
\; +\;   \int_0^{1} \nabla w_t(0)\, H_t(0) \; dt  \\
& -\; \int_0^1 \Big\{\, 
\mf b_{\alpha, A} \big(\, w_t(0)\, ,\, H_t (0)\, \big)
\;+\; \mf b_{\beta, B} \big(\, w_t(1)\, ,\, H_t(1)\, \big)\,
\Big\}\, dt \; .
\end{aligned}
\end{equation*}
for all $H$ in $C^{1,2}([0,1]\times[0,1])$. Recall the definition of
$\mf q_{\varrho, D} (a, M)$, introduced in \eqref{6-04}. As $w$
satisfies the boundary conditions \eqref{6-03}, we may rewrite the
previous identity as
\begin{equation*}
\begin{aligned}
J_{1,H} ( w ) \; & =\;
\int_0^{1} \big\langle  \partial_t w_t \,-\, \Delta w_t
\,,\, H_t  \big\rangle \; dt
\;-\; \int_0^{1}  \big\langle \sigma( w_t ), 
\big( \nabla H_t \big)^2 \big\rangle \; dt
\\
& -\; \int_0^1 \Big\{\, 
\mf q_{\alpha, A} \big(\, w_t(0)\, ,\, H_t (0)\, \big)
\;+\; \mf q_{\beta, B} \big(\, w_t(1)\, ,\, H_t(1)\, \big)\,
\Big\}\, dt \; .
\end{aligned}
\end{equation*}
As $\delta_0/2 \le w\le 1- \delta_0/2$,
$\sigma( w_t ) \ge (\delta_0/2)^2$. On the other hand, for
$\delta_0/2 \le a, \varrho \le 1- \delta_0/2$, $\mf q_{\varrho, D} (a,
M) \ge (2/D) \, (\delta_0/2)^2 \, [\cosh M -1] \ge (1/D) \,
(\delta_0/2)^2 \, M^2$. Therefore,
\begin{equation*}
\begin{aligned}
J_{1,H} ( w ) \; & \le \;
\int_0^{1} \big\langle  \partial_t w_t \,-\, \Delta w_t
\,,\, H_t  \big\rangle \; dt
\;-\; \Big( \frac{\delta_0}{2}\Big)^2 \, \int_0^{1}  \big\langle 
\big( \nabla H_t \big)^2 \big\rangle \; dt
\\
& -\; \Big( \frac{\delta_0}{2}\Big)^2 \int_0^1 \Big\{\, 
\frac{1}{A} H_t(0)^2 \;+\; \frac{1}{B} H_t(1)^2
\Big\}\, dt \; .
\end{aligned}
\end{equation*}
By \eqref{6-05}, we may rewrite this inequality as
\begin{equation*}
J_{1,H} ( w ) \; \le \;
\int_0^{1} \big\langle  \partial_t w_t \,-\, \Delta w_t
\,,\, H_t  \big\rangle \; dt
\;-\; \Big( \frac{\delta_0}{2}\Big)^2 \, \int_0^{1}  
\big\Vert\, H_t \, \big\Vert_{\mc H_R}^2  \; dt\;.
\end{equation*}

Since $Q_{[0,1]} (w) < \infty$, the rate functional
$I_{[0,1]}(w|\bar\rho)$ is given by the variational formula
\eqref{1-02}. Therefore, maximizing over $H$ on both sides of the
previous displayed equation yields that
\begin{equation}
\label{6-08}
I_{[0,1]}(w|\bar\rho) \; \le \;
\int_0^{1} \Big( \frac{2}{\delta_0}\Big)^2 \sup_{G\in C^2([0,1])}
\Big\{\,
\big\langle h(t) \,,\, G  \big\rangle  \;-\;  
\big\Vert\, G \, \big\Vert_{\mc H_R}^2  \, \Big\} \; dt\;,
\end{equation}
where $h(t) = \partial_t w_t \,-\, \Delta w_t$.

By \eqref{6-06},
\begin{equation*}
h(t) \;=\; \sum_{k\ge 1} \lambda_k \,
\frac{2 e^{\lambda_k t} -1}{e^{\lambda_k}-1} \,
\langle \gamma -\bar\rho \,,\, f_k\rangle \, f_k 
\end{equation*}
Hence, by Young's inequality and \eqref{6-07}, for all $G\in
C^2([0,1])$,
\begin{equation*}
\langle h(t) \,,\, G  \rangle \;=\;
\sum_{k\ge 1} \langle h(t) \,,\, f_k  \big\rangle\,
\langle G \,,\, f_k  \big\rangle \;\le\;
\sum_{k\ge 1} \frac{1}{4\, \lambda_k}
\langle h(t) \,,\, f_k  \big\rangle^2 \;+\;
\big\Vert\, G \, \big\Vert_{\mc H_R}^2\;.
\end{equation*}
By the formula for $h(t)$, the last sum is equal to
\begin{equation*}
\frac{1}{4} \sum_{k\ge 1} \lambda_k \,
\Big( \frac{2 e^{\lambda_k t} -1}{e^{\lambda_k}-1}\Big)^2 \,
\langle \gamma -\bar\rho \,,\, f_k\rangle^2
\;\le\; J^2 \sum_{k\ge 1} \lambda_k \,
e^{2 \lambda_k (t-1)} \,
\langle \gamma -\bar\rho \,,\, f_k\rangle^2 \;,
\end{equation*}
where $J= (1-e^{-\lambda_1})^{-1}$.

Reporting the previous estimate to \eqref{6-08} yields that
\begin{equation*}
\begin{aligned}
I_{[0,1]}(w|\bar\rho) \; & \le \;
\Big( \frac{2J}{\delta_0}\Big)^2 \int_0^{1} 
\sum_{k\ge 1} \lambda_k \, e^{2 \lambda_k (t-1)}
\, \langle \gamma -\bar\rho \,,\, f_k\rangle^2  \; dt \\
& \le \; \Big( \frac{2J}{\delta_0}\Big)^2 
\sum_{k\ge 1} \, \langle \gamma -\bar\rho \,,\, f_k\rangle^2
\;=\; \Big( \frac{2J}{\delta_0}\Big)^2  \, \Vert \, \gamma -\bar\rho
\, \Vert^2_2\;,
\end{aligned}
\end{equation*}
which concludes the proof of the lemma.
\end{proof}

We can now prove the upper bound for the quasi-potential and conclude
the proof of Theorem \ref{qp=s}.

\begin{lemma}
\label{l14}  
For each $\gamma \in \ms M_{\rm ac}$, we have  $V(\gamma) \le S(\gamma)$.
\end{lemma}

\begin{proof}
Fix $0<\varepsilon < (\delta_0/2) \, \min\{(1/4) , (1/\Lambda)\}$,
where $\Lambda$ has been introduced in the statement of Lemma
\ref{l13}. Let $\gamma \in \ms M_{\rm ac}$, and recall that we denote
by $v^{(\gamma )}(t,x)$ the solution to \eqref{4-08} with initial
condition $\gamma$. By Lemma \ref{l12}, there exists
$T_1=T_1(\varepsilon)$ such that {\color{bblue}
$\| v^{(\gamma )}(t) -\bar\rho \|_\infty < \varepsilon$ for any
$t\ge T_1$}.  Since $v^{(\gamma )}(T_1)$ fullfils the hypotheses of
Lemma \ref{l13}, let $w$ be the path which connects $\bar\rho$ to
$v^{(\gamma )}(T_1)$ in the interval $[0,1]$ constructed in that
lemma.

Let $T:=T_1+1$ and $w^*(t)$, $t\in [0,T]$, be the path
\begin{equation}
\label{pi*}
w^* (t) \; =\; 
\left\{ 
  \begin{array}{ll}
w (t) & \text{for $0\le t\le 1$} \\
v^{(\gamma )} (T - t)  & \text{for $1\le t\le T$}
  \end{array}
\right.
\end{equation}

Recall the definition of $\ms M_\delta$ given in \eqref{4-01}.  Let
$(\gamma_n ,\, n\ge 1)$ be a sequence such that
$\gamma_n\in \ms M_{\delta_n}$ for some $\delta_n>0$ and which
converges to $\gamma$ a. s. Denote by $v^{(\gamma_n)}$ the solution to
\eqref{4-08} with initial condition $\gamma_n$.

\smallskip\noindent{\it Claim 1:} $v^{(\gamma_n)}(T_1)$ converges to
$v^{(\gamma )}(T_1)$ in $C([0,1])$. \smallskip

To prove this claim, let $F=F(\gamma)$, $F_n=F(\gamma_n)$, and denote
by $F^{(\gamma)}$, $F^{(\gamma_n)}$ the solutions to \eqref{1-06} with initial
conditions $F$, $F_n$, respectively.

By Proposition \ref{p01}, $F_n$ converges to $F$ in $C^1([0,1])$.
Hence, by Lemma \ref{l19}, $F^{(\gamma_n)}(T_1)$ converges to
$F^{(\gamma)}(T_1)$ in $C^2([0,1])$. On the other hand, by Lemma
\ref{l16}, there exists a constant $c_1>0$ such that
$c_1 \le \nabla F^{(\gamma_n)}(T_1) \le c^{-1}_1$,
$c_1 \le \nabla F^{(\gamma)} (T_1) \le c^{-1}_1$ for all $n\ge 1$.
Hence, by \eqref{4-10}, $v^{(\gamma_n)}(T_1)$ converges to
$v^{(\gamma )}(T_1)$ in $C([0,1])$, as claimed. \smallskip

Since $\| v^{(\gamma )}(T_1) -\bar\rho \|_\infty < \varepsilon$, by
Claim 1, there exists $n_0$ such that
$\| v^{(\gamma_n)}(T_1) -\bar\rho \|_\infty < 2\, \varepsilon$ for all
$n\ge n_0$. Fix such $n$, and let $w^{n} (t)$ be the path joining
$\bar\rho$ to $v^{(\gamma_n)}(T_1)$ in the time interval $[0,1]$
constructed in Lemma \ref{l13}. Define the path $w^{n,*} (t)$,
$0\le t\le T_1+1$ as
\begin{equation}
\label{pi*n}
w^{n,*} (t) \; =\; 
\left\{ 
\begin{array}{ll}
w^{n} (t) & \text{for $0\le t\le 1$} \\
v^{(\gamma_n)} (T - t)  & \text{for $1\le t\le T$}
  \end{array}
\right.
\end{equation}

\smallskip\noindent{\it Claim 2:} The path $w^{n,*}$ converges in
$D ([0,T], \ms M )$ to $w^*$. \smallskip

Before proving this claim, we conclude the proof of the lemma.
By the lower semi-continuity of the functional
$I_{[0,T]}(\,\cdot\, |\, \bar\rho)$,
\begin{equation}\label{4}
I_{[0,T]} (w^* | \bar \rho) 
\;\le\; \liminf_{n} I_{[0,T]} (w^{n,*} | \bar \rho)\;.
\end{equation}
On the other hand, by definition of the rate function and
\eqref{4-02b},
\begin{equation}
\label{eq:5}
I_{[0,T]} \big( w^{n,*} \,|\, \bar \rho \big) \;\le\;
I_{[0,1]} \big(w^{n} \,|\, \bar \rho \big) \;+\;
I_{[0,T_1]} \big( v^{(\gamma_n)}(T_1-\cdot) \,|\,  v^{(\gamma_n)}(T_1)\big)\;.
\end{equation}
By Lemma \ref{l13}, for $n\ge n_0$
\begin{equation}
\label{6}
I_{[0,1]} \big(w^{n} \big| \bar \rho \big) \,\le\; C_0\,
\Vert v^{(\gamma_n)}(T_1) - \bar\rho \Vert_2^2 
\end{equation}
for some constant $C_0=C_0(\delta_0)$.

By Proposition \ref{l10} and Lemmata \ref{l10b}, \ref{l11},
$(x,t) \mapsto v^{(\gamma_n)} (T_1-t, x)$ belongs to
$C^{1,2}([0,T_1] \times [0,1])$ and is bounded away from $0$ and $1$,
namely it belongs to $D_{T_1,\delta_n}$ for some $\delta_n>0$. Hence,
by Lemma \ref{l08}, as $v^{(\gamma_n)}(T_1-\,\cdot\,)$ solves
\eqref{4-06} with $K=0$,
\begin{equation}
\label{7}
I_{[0,T_1]} \big( v^{(\gamma_n)}(T_1-\cdot) \,|\, v^{(\gamma_n)}(T_1)\big)
\;=\; S_0(\gamma_n) \;-\; S_0(v^{(\gamma_n)}(T_1)) \;.
\end{equation}

By equations \eqref{4}--\eqref{7},
\begin{equation*}
I_{[0,T]} (w^* \,|\, \bar \rho) \;\le \;
\liminf_{n} \big\{\,  S_0(\gamma_n) \;-\; S_0(v^{(\gamma_n)}(T_1))
\;+\; C_0\,  \Vert v^{(\gamma_n)}(T_1) - \bar\rho \Vert_2^2 \,\big\}\;.
\end{equation*}
By Remark \ref{rm01}, $S_0(\gamma_n)$ converges to $S_0(\gamma)$.
Thus, by the convergence of $v^{(\gamma_n)}(T_1)$ to $v^{(\gamma)}(T_1)$
in $C([0,1])$, the lower semicontinuity of $S_0$, and the bound
$\Vert v^{(\gamma)}(T_1) -\bar\rho \Vert_\infty < \varepsilon$, the
right-hand side is less than or equal to
\begin{equation*}
S_0(\gamma) \;-\; S_0(\bar\rho) \;+\; C_0\, \Vert v^{(\gamma)} (T_1) -
\bar\rho \Vert_2^2 \;\le\; S(\gamma) \;+\; C_0\, \varepsilon^2 \;.
\end{equation*}
To completes the proof of the lemma, it remains to show that Claim 2
is in force.

\smallskip\noindent{\it Proof of Claim 2.}  It is enough to show that
$w^{n,*}$ converges to $w^*$ in $C([0,T], \ms {M})$. Equivalently, to
show that $v^{(\gamma_n)}$ converges to $v^{(\gamma)}$ in
$C([0,T_1], \ms {M})$ and $w^{n}$ converges to $w$ in
$C([0,1], \ms {M})$.

Fix $0< t_0 < T_1$ small. By the arguments presented in the proof of
Claim 1 and since the convergence in Lemma \ref{l19} is uniform for
$t\ge t_0$, $v^{(\gamma_n)}$ converges to $v^{(\gamma)}$ in
$C([t_0, T_1] \times [0,1])$.

On the other hand, by Lemma \ref{l01}, $\nabla F^{(\gamma_n)}(t)$ and
$\nabla F^{(\gamma)} (t)$ are uniformly bounded, and, by Lemma \ref{l16},
$\nabla v^{(\gamma_n)}(t)$ and $\nabla v^{(\gamma)}(t)$ are uniformly
bounded in $[0,T_1]$.  As $v^{(\gamma_n)}$, $v^{(\gamma)}$ are weak
solutions to \eqref{4-11}, for each $G\in C([0,1])$,
\begin{equation*}
\lim_{t_0\downarrow 0} \; \limsup_n \;
\sup_{t\in[0, t_0]} 
\big|\, \langle v^{(\gamma_n)} (t),G\rangle
\,-\, \langle v^{(\gamma)} (t),G\rangle \, \big|
\;=\; 0\;.
\end{equation*}
This concludes the proof that $v^{(\gamma_n)}$ converges to $v^{(\gamma)}$
in $C([0,T], \ms {M})$.

By \eqref{4-10} and Remark \ref{rm06}, we can extend the result
obtained in Claim 1 and show that $v^{(\gamma_n)}(T_1)$ converges to
$v^{(\gamma)} (T_1)$ in $C^2( [0,1])$. Hence, by the explicit formula
\eqref{6-06}, $w^{n}$ converges to $w$ in $C( [0,1]\times[0,1])$.
This completes the proof of Claim 2 and of the lemma.
\end{proof}


\appendix

\section{The Robin Laplacian}
\label{sec04}

We present in this section some results on the Robin Laplacian needed
in the previous section. Denote by $\color{bblue} \Delta_R$ the
Laplacian on $[0,1]$ with Robin boundary conditions, sometimes called
the Robin Laplacian \cite[Section 4.3]{S08}.

Consider the eigenvalue problem
\begin{equation}
\label{6-01}
\left\{
\begin{aligned}
& -\, \Delta f =  \lambda \, f \; , \\
& (\nabla f ) (0) \,=\, A^{-1} \, f(0) \;,\\
& (\nabla f)(1) \,=\, -\, B^{-1} f(1) \;.
\end{aligned}
\right.
\end{equation}
The equation $-\, \Delta f = \lambda \, f$ can be turned into a
two-dimensional ODE which yields that the solutions to \eqref{6-01}
are given by
$f(x) = a \, [ \, \cos (\sqrt{\lambda} x) + b \, \sin (\sqrt{\lambda}
x)\,]$ for some $a$, $b\in \bb R$. The boundary conditions are
satisfied if and only if
\begin{equation}
\label{6-09}
\tan \sqrt{\lambda} \;=\; (A+B)\,
\frac{\sqrt{\lambda}}{\lambda AB -1}\;,
\end{equation}
in which case $b= (A\sqrt{\lambda})^{-1}$. This identity excludes
$\lambda=0$ from the set of eigenvalues of the Robin Laplacian. 

An analysis of \eqref{6-09} shows that it has a countable set of
solutions $\color{bblue} \{\lambda_j : j\ge 1\}$, where $0<\lambda_1$,
$\lambda_j < \lambda_{j+1}$ and $\lambda_j \sim j^2$ in the sense that
there exists $0<c_0<c_1<\infty$ such that
\begin{equation}
\label{6-12}
c_0 \, j^2 \;\le\; \lambda_j \;\le\; c_1\, j^2\;.
\end{equation}
Denote by $\color{bblue} \{f_j : j\ge 1\}$ the associated orthonormal
eigenvectors, which form a basis of $\ms L^2([0,1])$. By the previous
analysis,
\begin{equation}
\label{6-13}
f_j(x) \;=\; a_j \,\big\{\, \cos (\sqrt{\lambda_j} x)
\;+\; \frac{1}{A\sqrt{\lambda_j}}  \, \sin (\sqrt{\lambda_j}
x)\, \big\}\;,
\end{equation}
where $a_j$ is chosen for $f_j$ to have $\ms L^2$-norm equal to $1$. It
can be shown that $|a_j|\le C_0$ for all $j\ge 1$, where $C_0$ is a
finite constant depending only on $A$ and $B$. Therefore, by
\eqref{6-12},
\begin{equation}
\label{6-11}
\Vert \, f_j \, \Vert_\infty \;\le\; C_0 \;,
\quad \Vert \, \nabla^n f_j \, \Vert_\infty
\;\le\; C_0 \, (\lambda_j)^{n/2} \;\le\; C_0 \, j^n
\end{equation}
for all $j\ge 1$, $n\ge 1$.

A straightforward computation provides a formula for the Green
function of the Robin Laplacian: Let
$K_R: [0,1] \times [0,1] \to \bb R_+$ be given by
\begin{equation}
\label{6-10}
K_R(x,y) \;=\; \frac{1}{1+A+B}\,
\left\{
\begin{aligned}
& (B+1-x)\, (A+y)\;, \quad 0\le y\le x \le 1\;,  \\
& (B+1-y)\, (A+x)\;, \quad 0\le x\le y \le 1 \;.
\end{aligned}
\right.
\end{equation}
Denote by $K_R$ the integral operator defined by
\begin{equation*}
(K_R f)(x) \;=\; \int_0^1 K_R(x,y)\, f(y)\; dy\;.
\end{equation*}
Then, $K_R = (-\Delta_R)^{-1}$.

Denote by $\color{bblue} \mc H_R$ the Hilbert space obtained by
completing the space
$\color{bblue} C^2_{A,B}([0,1]) = \{ f \in C^2([0,1]) : (\nabla f )
(0) \,=\, A^{-1} \, f(0) \;,\, (\nabla f)(1) \,=\, -\, B^{-1} f(1) \}$
endowed with the scalar product $\<\,\cdot\,,\,\cdot\,\>_{\mc H_R}$
defined by
\begin{equation}
\label{6-05}
\begin{aligned}
\<\,f\,,\,g\,\>_{\mc H_R} \;& =\; \<\,f\,,\, (-\, \Delta_R) g\,\> \\
& =\;
\frac{1}{A}\, f(0)\,g(0) \;+\; \int_0^1 (\nabla  f)(x)\,
(\nabla  g)(x) \; dx \;+\; \frac{1}{B}\, f(1)\, g(1) \;.
\end{aligned}
\end{equation}
Denote by $\color{bblue} \Vert f\Vert_{\mc H_R}$ the norm induced by the
scalar product $\<\,\cdot\,,\,\cdot\,\>_{\mc H_R}$. We have that
\begin{equation}
\label{6-16}
\Vert f\Vert^2_{\mc H_R} \;=\;
\sum_{k\ge 1} \lambda_k\, \< f \,,\, f_k\>^2 \;.
\end{equation}
for all $f\in \mc H_R$.

The norms $\Vert \,\cdot\, \Vert_{\mc H_R}$ and $\Vert \,\cdot\, \Vert_{\mc
H_1}$ are equivalent. There exist finite constants $0<C_1<C_2<\infty$
such that
\begin{equation}
\label{6-14}
C_1\, \Vert f\Vert_{\mc H_1} \;\le\; \Vert f\Vert_{\mc H_R}
\;\le\; C_2\, \Vert f\Vert_{\mc H_1}
\end{equation}
for all $f\in C^2([0,1])$. In particular, the spaces $\mc H_R$ and
$\mc H_1$ coincide.

In terms of the
eigenfunctions $f_k$,
\begin{equation}
\label{6-07}
\Vert \, f\, \Vert^2_{\mc H_R} \;=\;
\sum_{k\ge 1} \lambda_k \, |\, \<\,f \,,\, f_k\,\>\,|^2 \;.
\end{equation}
Moreover, a straightforward computation yields that for all
$f\in C^2_{A,B}([0,1])$,
\begin{equation}
\label{6-02}
\Vert \, f \, \Vert^2_\infty \;\le\; 2\, (A\vee 1)\,
\Vert f\Vert^2_{\mc H_R} \; .
\end{equation}

Fix a function $f$ in $\mc H_1$. It is well known that there exists a
continuous function $f^{(c)}: [0,1] \to \bb R$ (actually H\"older
continuous,
$|f^{(c)}(y) - f^{(c)}(x)| \le \Vert f \Vert_{2}
|y-x|^{1/2}$) such that $f = f^{(c)}$ almost surely. Moreover, for all
$h\in C^1([0,1])$,
\begin{equation}
\label{1-03b}
\int_0^1 f \, \nabla h\;dx \;=\; f^{(c)}(1)\, h(1)
\;-\; f^{(c)}(0)\, h(0)
\;-\; \int_0^1 \nabla f \,  h\;dx\;.
\end{equation}
The next result provides an explicit formula for $f^{(c)}$ in terms of
the eigenvectors $f_k$.

\begin{lemma}
\label{l15}
There exists a finite constant $C_0$ such that
\begin{equation*}
\sum_{k\ge 1} \big|\, \< f \,,\, f_k\>\, \big| \;\le\; C_0
\, \Vert \, f\,\Vert_{\mc H_R}
\end{equation*}
for all $f \in \mc H_1$.  In particular,
$\sum_{k\ge 1} \< f \,,\, f_k\>\, f_k(\cdot)$ defines a continuous
function, and, for almost all $x\in [0,1]$,
\begin{equation}
\label{6-15}
f(x) \;=\; \sum_{k\ge 1} \< f \,,\, f_k\>\, f_k(x)\;.
\end{equation}
\end{lemma}

\begin{proof}
By \eqref{6-14}, $f$ belongs to $\mc H_R$. By Schwarz inequality,
\begin{equation*}
\Big( \sum_{k\ge 1} \big|\, \< f \,,\, f_k\>\, \big| \,\Big)^2
\;\le\;
\sum_{k\ge 1} \lambda_k\, \big|\, \< f \,,\, f_k\>\, \big|^2 \;
\sum_{k\ge 1} \frac{1}{\lambda_k}\;\cdot
\end{equation*}
The second sum is finite by \eqref{6-12} and the first one is finite
by \eqref{6-07}. This proves the first assertion.

Since each function $f_k$ is continuous, and a summable sum of
continuous functions is continuous,
$\sum_{k\ge 1} \< f \,,\, f_k\>\, f_k(\cdot)$ defines a continuous
function. As $(f_k : k\ge 1)$ forms an orthonormal basis of
$\ms L^2([0,1])$, $f = \sum_{k\ge 1} \< f \,,\, f_k\>\, f_k$ as an
identity in $\ms L^2([0,1])$. In particular, these functions are equal
almost everywhere.
\end{proof}

Denote by $\color{bblue} (P^{(R)}_t: t\ge 0)$ the semigroup in
$\ms L^2([0,1])$ generated by the Robin Laplacian: For any function
$f\in \ms L^2([0,1])$, $t>0$,
\begin{equation}
\label{5-09}
P^{(R)}_t f \;=\; \sum_{k\ge 1} e^{-\lambda_k t}\, \< f \,,\, f_k\>\, f_k\;.
\end{equation}
In particular, for each $t\ge 0$, $P^{(R)}_t$ is a symmetric operator
in $\ms L^2([0,1])$ and $P^{(R)}_t f\in C^\infty ([0,1])$ for all
$f\in \ms L^2([0,1])$. Moreover, as $P^{(R)}_t$ is symmetric, by
\eqref{6-07}, $P^{(R)}_t$ is a contraction in $\mc H_R$ and
$\ms L^2([0,1])$: 
\begin{equation}
\label{6-18}
\begin{gathered}
\Vert\, P^{(R)}_t f\,\Vert^2_{\mc H_R} \;=\;
\sum_{k\ge 1} e^{-2\lambda_k t}\, \lambda_k\,
|\, \< f \,,\, f_k\>\,|^2  \;\le\;
\Vert\,  f\,\Vert^2_{\mc H_R}\;, \\
\Vert\, P^{(R)}_t f\,\Vert^2_{2} \;=\;
\sum_{k\ge 1} e^{-2\lambda_k t}\, 
|\, \< f \,,\, f_k\>\,|^2  \;\le\;
\Vert\,  f\,\Vert^2_{2}\;.
\end{gathered}
\end{equation}

Let $f\in \ms L^2([0,1])$ be given by
$f = \sum_{k\ge 1} \< f \,,\, f_k\> \, f_k$. For each $t>0$, there
exists a finite constant $C_0 (t)$ such that
\begin{equation}
\label{6-17}
\begin{aligned}
\Vert\, P^{(R)}_t f \,\Vert^2_{\infty} \;\le\;
C_0(t) \, \Vert\, f \,\Vert^2_{2} \;, \quad
\Vert\, P^{(R)}_t f \,\Vert^2_{\mc H_R} \;\le\;
C_0(t) \, \Vert\, f \,\Vert^2_{2}\;.
\end{aligned}
\end{equation}
Indeed, by \eqref{6-07} and since $P^{(R)}_t$ is symmetric and
$P^{(R)}_t f_k = e^{-\lambda_k t} f_k$,
\begin{equation*}
\Vert\, P^{(R)}_t f \,\Vert^2_{\mc H_R} \;=\;
\sum_{k\ge 1} \lambda_k \, e^{-2\lambda_k t}
\big|\, \< f \,,\, f_k\>\, \big|^2 
\;\le\; C_0(t) \, \sum_{k\ge 1} 
\big|\, \< f \,,\, f_k\>\, \big|^2
\;=\; C_0(t) \, \Vert\, f \,\Vert^2_{2} 
\end{equation*}
for some finite constant $C_0(t)$. On the other hand, by Schwarz
inequality and \eqref{6-11},
\begin{equation*}
\Vert\, P^{(R)}_t f \,\Vert^2_{\infty} \;=\;
\Big\Vert\, \sum_{k\ge 1} e^{-\lambda_k t}
\, \< f \,,\, f_k\>\, f_k   \, \Big\Vert^2_{\infty}
\;\le\; \sum_{k\ge 1} e^{-2\lambda_k t}
\sum_{k\ge 1} \, \< f \,,\, f_k\>^2
\;=\; C_0(t) \, \Vert\, f \,\Vert^2_{2} 
\end{equation*}
for some finite constant $C_0(t)$.

\begin{lemma}
\label{l17}
There exists a finite constant $C_0$ such that
\begin{equation*}
\Vert\, P^{(R)}_t f \,-\, f\,\Vert_2 \;\le \;
C_0\,  t^{1/3}\, \, \Vert\, f\, \Vert_{\mc H_R}
\end{equation*}
for all $t\ge 0$, $f \in \mc H_R$.
\end{lemma}

\begin{proof}
Since $(f_k : k\ge 1)$ is an orthonormal basis of $\ms L^2([0,1])$,
\begin{equation*}
\Vert\, P^{(R)}_t f \,-\, f\,\Vert^2_2 \;=\;
\sum_{k\ge 1} \big[\, e^{-\lambda_k \, t} \,-\, 1\,\big]^2 \,
\big|\, \< f \,,\, f_k\>\, \big|^2 \;.
\end{equation*}
Fix $k_0\ge 1$. Since the sequence $\lambda_k$ increases,
the right-hand side can be bounded by
\begin{equation*}
\big[\, e^{-\lambda_{k_0} \, t} \,-\, 1\,\big]^2
\sum_{k= 1}^{k_0-1}  \big|\, \< f \,,\, f_k\>\, \big|^2
\;+\; \frac{1}{\lambda_{k_0}} \,
\sum_{k \ge k_0}  \lambda_k\, \big|\, \< f \,,\, f_k\>\, \big|^2\;.
\end{equation*}
The first sum is bounded by $\Vert\, f\, \Vert^2_2$. In view of
\eqref{6-07}, the second one is bounded by
$\Vert\, f\, \Vert^2_{\mc H_R}$ so that 
\begin{equation*}
\Vert\, P^{(R)}_t f \,-\, f\,\Vert^2_2 \;\le \;
\big[\, 1  \,-\, e^{-\lambda_{k_0} \, t}  \,\big]^2
\,  \Vert\, f\, \Vert^2_2 \;+\;
\frac{1}{\lambda_{k_0}} \Vert\, f\, \Vert^2_{\mc H_R} \;.
\end{equation*}
As $1-e^{-x} \le x$, $x>0$, and since, by \eqref{6-14}, $\Vert\, f\,
\Vert_2 \le C_0 \Vert\, f\, \Vert_{\mc H_r}$ for some finite constant
$C_0$, 
\begin{equation*}
\Vert\, P^{(R)}_t f \,-\, f\,\Vert^2_2 \;\le \;
\Big\{\, C_0\, (\, \lambda_{k_0} \, t  \,)^2
\;+\; \frac{1}{\lambda_{k_0}} \,\Big\}\, \Vert\, f\, \Vert^2_{\mc H_R} \;.
\end{equation*}
To complete the proof, it remains to choose $k_0$ such that
$\lambda^{-3}_{k_0} \sim t^2$.
\end{proof}

\begin{lemma}
\label{l18}
There exists a finite constant $C_0$ such that
\begin{equation*}
\Vert \, P^{(R)}_t f \,-\, f  \,\Vert_\infty \;\le \;
C_0\,  t^{1/5}\, \, \Vert\, f\, \Vert_{\mc H_R} 
\end{equation*}
for all $t\ge 0$, $f \in C([0,1]) \cap \mc H_R$.
\end{lemma}

\begin{proof}
Fix $x\in [0,1]$. Since $f$ is continuous, by \eqref{6-15} and
\eqref{6-11},
\begin{equation*}
\big\{ \, P^{(R)}_t f (x) \,-\, f (x) \,\big\}^2 \;\le\;
C_0\, \Big(\, \sum_{k\ge 1} \big[\, 1 \,-\, e^{-\lambda_k \, t} \,\big] \,
\big|\, \< f \,,\, f_k\>\, \big|\,\Big)^2 
\end{equation*}
for some finite constant $C_0$. By Schwarz inequality and
\eqref{6-07}, the right-hand is bounded by
\begin{equation*}
C_0\, \sum_{k\ge 1}  \frac{1}{\lambda_k} \,
\big[\, 1 \,-\, e^{-\lambda_k \, t} \,\big]^2 \,
\sum_{k\ge 1}  \lambda_k \,
\big|\, \< f \,,\, f_k\>\, \big|^2 \;=\;
C_0\, \Vert\, f\, \Vert^2_{\mc H_R} \,
\sum_{k\ge 1}  \frac{1}{\lambda_k} \,
\big[\, 1 \,-\, e^{-\lambda_k \, t} \,\big]^2 \;.
\end{equation*}

It remains to estimate the sum.  Fix $k_0\ge 1$.  Since the sequence
$\lambda_k$ increases, as $1-e^{-x} \le x$, $x>0$, by \eqref{6-12},
the sum is less than or equal to
\begin{equation*}
C_0\, \big[\, 1 \,-\, e^{-\lambda_{k_0} \, t} \,\big]^2
\;+\; \sum_{k \ge k_0}  \frac{1}{\lambda_{k}} \;\le\;
C_0\, \Big\{ \, (k^2_0 \, t)^2 \;+\; \frac{1}{k_0} \,\Big\} 
\end{equation*}
for some finite constant $C_0$. It remains to choose $k_0$ such that
$k_0^5 \sim t^{-2}$.
\end{proof}

\section{Heat equations with mixed boundary conditions}
\label{sec05}

We present in this section some result on the initial-boundary value
problems \eqref{1-06}, \eqref{5-01}.  Denote by
$\color{bblue} \mc H^1 = \mc H^1([0,1])$ the Hilbert space obtained by
completing the space $C^1([0,1])$ endowed with the scalar product
$\<\,\cdot\,,\,\cdot\,\>_{\mc H^1}$ defined by
\begin{equation}
\label{6-05b}
\<\,f\,,\,g\,\>_{\mc H^1} \;=\;
\<\,f\,,\, g\,\>  \;+\;
\<\, \nabla f\,,\, \nabla g\,\> \;.
\end{equation}
Denote by $\color{bblue} \Vert f\Vert_{\mc H^1}$ the norm induced by
the scalar product $\<\,\cdot\,,\,\cdot\,\>_{\mc H^1}$.  Fix a
function $\phi\in \ms L^2([0,1])$, and consider the initial-boundary
problem

\begin{equation}
\label{1-06h} 
\begin{cases}
\partial_tu \,=\, \Delta u\\
(\nabla u) (t,0) \,=\, A^{-1} \,u(t,0) \\
(\nabla u)(t,1) \,=\, -\, B^{-1}\, u(t,1) \\
u(0,\cdot) \,=\, \phi (\cdot)\;.
\end{cases}
\end{equation}

\begin{definition}
\label{d01}
A function $u$ in $\ms L^2([0,T]; \mc H^{1})$ is said to be a
generalized (or weak) solution in the cylinder $[0,T]\times [0,1]$ of
the equation \eqref{1-06h} if
\begin{equation*}
\begin{aligned}
& \int_0^1  u_t \, H_t \; dx \;-\; \int_0^1 \phi \, H_0 \; dx 
\;-\; \int_0^t ds \, \int_0^1  u_s \, \partial_s H_s \; dx \\
& \quad =\;- \int_0^t ds \, \int_0^1 \nabla u_s\, \nabla H_s \; dx
\;-\; \int_0^t \big\{\,  \frac{1}{B}\, u_s(1)\, H_s(1)
\,+\, \frac{1}{A} \, u_s(0)\, H_s(0) \,\big\} \; ds 
\end{aligned}
\end{equation*}
for every $0<t\le T$, function $H$ in $C^{1,2}([0,T]\times [0,1])$.
\end{definition}

\begin{theorem}
\label{mt6}
For each $\phi\in \ms L^2([0,1])$, there exists one and only one
generalized solution to \eqref{1-06h}. Moreover,
\begin{itemize}
\item[(a)] The solution is smooth in $(0,\infty) \times [0,1]$ and can
be represented as $u(t,x) = (P^{(R)}_t \phi)(x)$, where $P^{(R)}_t$ is
the semigroup associated to the Robin Laplacian.

\item[(b)]  For all $(t,x) \in \bb R_+ \times [0,1]$,
\begin{equation}
\label{10-01}
\min \{\, 0 \,,\, {\rm ess\, inf }\, \phi \,\}
\,\le\, u(t,x) \,\le\, \max \{\, 0 \,,\,
{\rm ess\, sup }\, \phi \,\}\;.
\end{equation}

\item[(c)] If $\phi (x) \le b$ for some $b>0$, then, for each $t_0>0$
there exists $\epsilon >0$ such that $u(t,x) \le b -\epsilon$ for all
$(t,x) \in [t_0, \infty) \times [0,1]$. Analogously, if
$\phi (x) \ge a$ for some $a<0$, then, for each $t_0>0$ there exists
$\epsilon >0$ such that $u(t,x) \ge a + \epsilon$ for all
$(t,x) \in [t_0, \infty) \times [0,1]$.

\item[(d)] If $\phi$ belongs to $C([0,1]) \cap \mc H_1$, then the
solution belongs to $C([0, \infty) \times [0,1])$.
\end{itemize}
\end{theorem}

\begin{proof}
Existence and uniqueness of generalized solutions, as well as their
representation in terms of the semigroup $P^{(R)}_t$ is the content of
Theorems 1 and 3 in \cite[Section VI.2]{M83}.

We turn to \eqref{10-01}. Assume first that $\phi$ belongs to
$\mc H_1$.  By \eqref{6-14}, $\phi \in \mc H_R$, and, by Lemma
\ref{l18}, $u(t)$ converges to $\phi$ in $\ms L^\infty ([0, 1])$ as
$t\to 0$.  Since the solution is smooth in $(0,\infty) \times [0,1]$,
by the maximum principle stated in Theorems 2 and 3 of \cite[Chapter
3]{PW84},
\begin{equation*}
\min \{\, 0 \,,\, \inf_{0\le y\le 1} u(t_0,y)  \,\}
\,\le\, u(t,x) \,\le\, \max \{\, 0 \,,\,
\sup_{0\le y\le 1} u(t_0,y) \,\}
\end{equation*}
for all $(t,x) \in [t_0, \infty) \times [0,1]$. Letting $t_0\to 0$, as
$u(t_0)$ converges to $\phi$ in $\ms L^\infty([0,1])$, yields
\eqref{10-01}.

To extend this result to $\phi\in \ms L^2([0,1])$, we consider a
sequence $\phi_n \in \mc H_1$ which converges to $\phi$ in
$\ms L^2([0,1])$ and such that
${\rm ess\, inf }\, \phi \,\le\, \phi_n (x) \,\le\, {\rm ess\,
sup }\, \phi$ for all $0\le x\le 1$. Denote by $u^n$ the solution to
\eqref{1-06h} with initial condition $\phi_n$. Fix $t>0$. By the
result for initial conditions in $\mc H_1$,
\begin{equation*}
\begin{aligned}
& \min \{\, 0 \,,\, {\rm ess\, inf }\, \phi \,\}
\;\le\;
\min \{\, 0 \,,\, \inf_{0\le y\le 1} \phi_n(y) \,\} \\
& \quad \;\le\;
u^n(t,x)
\;\le\;
\max \{\, 0 \,,\, \sup_{0\le y\le 1} \phi_n(y)  \,\}
\,\le\, \max \{\, 0 \,,\,
{\rm ess\, sup }\, \phi \,\}\;.
\end{aligned}
\end{equation*}
for all $0\le x\le 1$. By \eqref{6-17}, $u^n(t)$ converges to $u(t)$
in $\ms L^\infty ([0,1])$. This completes the proof of \eqref{10-01}.

Assume that $\phi (x) \le b$ for some $b>0$.  By \eqref{10-01},
$u(t,x) \le b$ for all $t \ge 0$, $0\le x\le 1$.  Fix $t_0>0$, and
assume that $\max_{0\le x\le 1} u(t_0,x) = b$. As $b>0$, the boundary
conditions imply that the maximum cannot be attained at the
boundary. On the other hand, if it is attained at the interior, by
Theorem 2 of \cite[Chapter 3]{PW84} and by the smoothness of the
solution, $u(t,x) = b$ for all $(t,x) \in (0,t_0] \times [0,1]$. This
is not possible at the boundary. Therefore,
$\max_{0\le x\le 1} u(t_0,x) < b$. By the maximum principle, this
bound can be extended to all $(t,x) \in [t_0, \infty) \times [0,1]$.
The same argument applies to the lower bound.

Assertion (d) follows from Lemma \ref{l18} and the representation of
the solutions.
\end{proof}

It follows from the previous result that the operator $P^{(R)}_t$ is a
contraction in $\ms L^\infty([0,1])$: for all $t\ge 0$,
$f\in \ms L^\infty([0,1])$,
\begin{equation}
\label{5-05}
\Vert\, P^{(R)}_t f\,\Vert_\infty \;\le\;
\Vert\,  f\,\Vert_\infty\;. 
\end{equation}

Recall from \eqref{1-05} that we denote by
$\bar\rho\in \ms M_{\rm ac}$ the unique stationary solution to the
equation \eqref{1-06}.

\begin{definition}
\label{d02}
Fix $\gamma \in \ms L^2( [0,1])$.  A function $u$ in
$\ms L^2([0,T]; \mc H^1)$ is said to be a generalized (or weak)
solution in the cylinder $[0,T]\times [0,1]$ of the equation
\eqref{1-06} if $u(t,x) - \bar\rho$ is a generalized solution to the
initial-boundary problem \eqref{1-06h} with initial condition
$\gamma - \bar\rho$.
\end{definition}

Therefore, a function $u$ in $\ms L^2([0,T]; \mc H^1)$ is a
generalized solution in the cylinder $[0,T]\times [0,1]$ of the
equation \eqref{1-06} if
\begin{equation*}
\begin{aligned}
& \int_0^1  u_t \, H_t \; dx \;-\; \int_0^1 \gamma \, H_0 \; dx 
\;-\; \int_0^t ds \, \int_0^1  u_s \, \partial_s H_s \; dx
=\;- \int_0^t ds \, \int_0^1 \nabla u_s\, \nabla H_s \; dx \\
& \quad 
\;-\; \int_0^t \Big\{\,  \frac{1}{B}\, [\, u_s(1) - \beta\, ] \, H_s(1)
\,+\, \frac{1}{A} \, [\, u_s(0) - \alpha\,] \, H_s(0) \,\Big\} \; ds 
\end{aligned}
\end{equation*}
for every $0<t\le T$, function $H$ in $C^{1,2}([0,T]\times [0,1])$.

\begin{theorem}
\label{mt4}
Fix $\gamma \in \ms L^2([0,1])$.  There exists a unique generalized
solution to \eqref{1-06}.  The solution is smooth in
$(0,\infty)\times [0,1]$ and satisfies the bounds
\begin{equation}
\label{1-08}
\min \{\, \alpha \,,\, {\rm ess\, inf }\,  \gamma \,\}
\,\le\, u(t,x) \,\le\, \max \{\, \beta \,,\,
{\rm ess\, sup }\, \gamma  \,\}
\end{equation}
for all $(t,x) \in [0,\infty) \times [0,1]$. Moreover, if
$\gamma\in \ms M_{\rm ac}$, for all $0<t_0 \le T$ there exists
$\epsilon>0$ such that $\epsilon \le u(t,x) \le 1-\epsilon$ for all
$(t,x) \in [t_0,\infty)\times [0,1]$. Finally, the solution is
continuous in $[0,\infty)\times [0,1]$ if $\gamma$ belongs to
$C([0,1]) \cap \mc H_1$.
\end{theorem}

\begin{proof}
The proof of this result is similar to the one of Theorem \ref{mt6}. 
\end{proof}

\begin{lemma}
\label{l19}
Let $\gamma$, $(\gamma_n:n\ge 1)$ be a sequence of density profiles in
$\ms M_{\rm ac}$. Assume that $\gamma_n$ converges to $\gamma$ in $\ms
M_{\rm ac}$. Let $u$, $u^n$ be the weak solutions to \eqref{1-06} with
initial conditions $\gamma$, $\gamma_n$, respectively. Then, for all
$t_0>0$, $u^n(t)$ converges to $u(t)$ in $C^2([0,1])$ uniformly in
$t\in [t_0, \infty)$.
\end{lemma}

\begin{proof}
Fix $t_0>0$.  By Definition \ref{d02}, Theorem \ref{mt6}.(a) and
\eqref{5-09}, for every $t\ge 0$,
\begin{equation*}
u^n(t) \,-\, u(t)\; =\; P^{(R)}_t (\gamma_n - \gamma)\;=\;
\sum_{k\ge 1} e^{-\lambda_k t}\, \< \gamma_n - \gamma \,,\, f_k\>\, f_k\;.
\end{equation*}
By this identity and the hypothesis, since $\gamma$, $\gamma_n$ are
bounded, $u^n(t)$ converges to $u(t)$ in $C([0,1])$, uniformly for
$t\ge t_0$. Taking the Laplacian on both sides of this identity
produces on the right-hand side an extra factor $-\,
\lambda_k$. Hence, $\Delta u^n(t)$ converges to $\Delta u(t)$ in
$C([0,1])$, uniformly for $t\ge t_0$. Uniform convergence in
$C^2([0,1])$ follows from the two previous convergences.
\end{proof}

\begin{remark}
\label{rm06}
Fix $k\ge 1$. The same proof yields that, for all $t_0>0$, $u^n(t)$
converges to $u(t)$ in $C^{2k}([0,1])$, uniformly in
$t\in [t_0, \infty)$.
\end{remark}

We conclude this section defining weak solutions to equation
\eqref{5-01} and stating a result on existence and uniqueness.

\begin{definition}
\label{d03}
Fix $\gamma\in\ms L^1([0,1])$, $H\in C^{0,1}([0,T]\times [0,1])$.  A
function $u$ in $\ms L^2([0,T]; \mc H^1)$ is said to be a generalized,
or weak, solution in the cylinder $[0,T]\times [0,1]$ of the equation
\eqref{5-01} if
\begin{equation}
\label{8-01}
\begin{aligned}
& \int_0^1  u_t \, G_t \; dx \;-\; \int_0^1 \gamma \, G_0 \; dx 
\;-\; \int_0^t ds \, \int_0^1  u_s \, \partial_s G_s \; dx \\
&\quad \;=\; 
\int_0^t ds \, \int_0^1 \big\{\,
-\, \nabla u_s \, \nabla G_s \, +\, 2\, \sigma (u_s)\,
\nabla H_s \,\nabla G_s \,\big\} \; dx \\
&\quad  \;+\; \int_0^t \big\{\,  \mf p_{\beta,B} ( u_s(1), H_s(1)) 
\, G_s(1) \,+\, \mf p_{\alpha,A} ( u_s(0) , H_s(0)) \, G_s(0)  \,\big\} \; ds
\end{aligned}
\end{equation}
for every $0<t\le T$ and function $G$ in $C^{1,2}([0,T]\times [0,1])$.
\end{definition}

Next result is taken from \cite{FGLN2021}

\begin{theorem}
\label{mt5}
Fix $\gamma\in \ms M_{\rm ac}$ and $H$ in
$C^{0,1}([0,T]\times [0,1])$. Then, there exists a unique weak
solution to \eqref{5-01}.  Moreover, $0\le u(t,x)\le 1$ a.s. in
$[0,T]\times [0,1]$.
\end{theorem}

\smallskip\noindent{\bf Acknowledgements:} The last author wishes to
thank D. Gabrielli for fruitful discussions.

\end{document}